\documentclass[a4paper,12pt]{amsart}

\usepackage{amsmath} 
\usepackage{amssymb}
\usepackage[mathscr]{eucal}
\usepackage{graphics}
\usepackage{multicol}
\usepackage{braket}

\usepackage{newtxtext, newtxmath}
\usepackage{ascmac}
\usepackage{framed}
\usepackage{ulem}
\usepackage{cancel}
\usepackage{arydshln}
\usepackage{bm}
\allowdisplaybreaks
\theoremstyle{plain}
\newtheorem{theorem}{Theorem}[subsection]
\newtheorem{lemma}[theorem]{Lemma}
\newtheorem{proposition}[theorem]{Proposition}

\makeatletter
\def\thefootnote{\ifnum\c@footnote>\z@\leavevmode\lower.5ex%
      \hbox{$^{\@arabic\c@footnote)}$}\fi}
\makeatother

\theoremstyle{definition}

\theoremstyle{remark}

\theoremstyle{definition} 
\newtheorem*{defi}{\indent Definition}

\newtheorem*{assertion}{\indent Assertion}

\def\dfrac#1#2{\displaystyle \frac{#1}{#2}}
\def\Aut{\mbox{\rm {Aut}}}
\def\det{\mbox{\rm {det}}}
\def\diag{\mbox{\rm {diag}}}

\def\Iso{\mbox{\rm {Iso}}}

\def\Ker{\mbox{\rm {Ker}}}

\def\tr{\mbox{\rm {tr}}}

\def\ov{\overline}

\def\dfrac#1#2{\displaystyle \frac{#1}{#2}}

\def\C{\mbox{\boldmath $C$}}

\def\H{\mbox{\boldmath $H$}}

\def\R{\mbox{\boldmath $R$}}

\def\Z{\mbox{\boldmath $Z$}}

\def\sR{\mbox{\boldmath $\scriptstyle{R}$}}

\def\0{\mbox{\boldmath {0}}}    
\def\1{\mbox{\boldmath {1}}}      
\def\2{\mbox{\boldmath {2}}}      
\def\3{\mbox{\boldmath {3}}}      
\def\4{\mbox{\boldmath {4}}}      
\def\5{\mbox{\boldmath {5}}}      
\def\6{\mbox{\boldmath {6}}}      
\def\7{\mbox{\boldmath {7}}}      
\def\8{\mbox{\boldmath {8}}}      
\def\9{\mbox{\boldmath {9}}}      
\def\a{\mbox{\boldmath $a$}}
\def\b{\mbox{\boldmath $b$}}

\def\m{\mbox{\boldmath $m$}}
\def\n{\mbox{\boldmath $n$}}

\def\v{\mbox{\boldmath $v$}}

\usepackage{epic,eepic}

\begin{document}
	
\title[Realizations of globally exceptional $\varmathbb{Z}_3 \times \varmathbb{Z}_3 $-symmetric spaces  Part I]
{On some realizations of globally exceptional $\varmathbb{Z}_3 \times \varmathbb{Z}_3 $-symmetric spaces $G/K$, $G=G_2, F_4, E_6$, Part I }

\author[Toshikazu Miyashita]{Toshikazu Miyashita}



\address{1365-3 Bessho onsen \endgraf
	Ueda City                \endgraf
	Nagano Prefecture 386-1431    \endgraf
	Japan}
\email{anarchybin@gmail.com}

\begin{abstract}
	R. Lutz introduced the notion of $\varGamma$-symmetric space as a generalization of the classical notion of symmetric space in 1981, where $\varGamma$ is a finite abelian group. In the present article, as $\varGamma=\varmathbb{Z}_3 \times \varmathbb{Z}_3$, we give the automorphisms $\tilde{\sigma}_3, \tilde{\tau}_3$ of order $3$ on the connected compact exceptional Lie groups $G=G_2, F_4,E_6$ 
	explicitly and determine the structure of the group $G^{\sigma_3} \cap G^{\tau_3}$ using homomorphism theorem elementary. These amount to some global realizations of exceptional $\varmathbb{Z}_3 \times \varmathbb{Z}_3$-symmetric spaces $G/K$, where $(G^{\sigma_3} \cap G^{\tau_3})_0 \subseteq K \subseteq G^{\sigma_3} \cap G^{\tau_3}$. 
\end{abstract}

\maketitle

\section {Introduction}
	In \cite{Lutz}, R. Lutz introduced the notion of $\varGamma$-symmetric space. Until now, as far as the author know, the case where $\varGamma=\varmathbb{Z}_2$ corresponds to the ordinary symmetric spaces, and Y. Bahturin and M. Goze classified the $\varmathbb{Z}_2 \times \varmathbb{Z}_2$-symmetric spaces  of classical type (\cite{Bah}) and A. Kollross classified the $\varmathbb{Z}_2 \times \varmathbb{Z}_2$-symmetric spaces of exceptional type (\cite{And}). These classifications were the results as Lie algebras, so using the results of Kollross's classification,  the author realized globally $\varmathbb{Z}_2 \times \varmathbb{Z}_2$-symmetric spaces of exceptional type as the coset space of Lie groups (\cite{miya1}). Besides, J.A.Wolf and A.Gray classified automorphism of order $3$ and its fixed points subgroups of connected compact Lie groups of centerfree. This amounts to the classification of the $\varmathbb{Z}_3$-symmetric spaces in the connected compact Lie groups of centerfree (\cite{wolf}). In \cite{iy1}, I.Yokota realized the inner automorphisms of order $3$ on the connected compact exceptional Lie groups $G=G_2, F_4, E_6$ explicitly and determined the structure of the fixed points subgroups of $G$ by them, and moreover in \cite{miya3}, the author and I. Yokota did similar realizations and determinations for the connected compact exceptional Lie group $E_7$.
	However, the author does not know the classification of $\varmathbb{Z}_3 \times \varmathbb{Z}_3$-symmetric spaces, so in the present article, using the known inner automorphisms of order $3$ on $G_2, F_4, E_6$
	 (\cite{iy1}) we realize some globally $\varmathbb{Z}_3 \times \varmathbb{Z}_3$-symmetric spaces. Unfortunately, we do not obtain their geometric interpretation.
	 
	Now, we describe the definition of $\varGamma$-symmetric space below (\cite{Bah}). 
	\begin{defi}\label{definition}
		Let $\varGamma$ be a finite abelian group and $G$ a connected Lie group. A homogeneous space $G/K$ is called $\varGamma$-symmetric if $G$ acts almost effectively on $G/K$ and there exists an injective homomorphism $\rho: \varGamma \to \Aut(G)$ such that $(G^\varGamma)_0 \subseteq K \subseteq 
		G^\varGamma$, where $G^\varGamma$ is the fixed points subgroup by all $\rho(\varGamma) \in \Aut(G)$ and $(G^\varGamma)_0$ is its connected component.
	\end{defi}

  In the case where $\varGamma=\varmathbb{Z}_3 \times \varmathbb{Z}_3$, since it follows from $G^{\sigma_3}=G^{{\sigma_3}^{-1}}, G^{\tau_3}=G^{{\tau_3}^{-1}}, G^{\sigma_3} \cap G^{\tau_3} \subset G^{\sigma_3\tau_3}=G^{{\sigma_3}^{-1}{\tau_3}^{-1}}$ and so on that  
  \begin{align*}
  G^\varGamma=G \cap G^{\sigma_3} \cap G^{\tau_3} \cap G^{{\sigma_3}^{-1}} \cap \cdots \cap G^{{\sigma_3}^{-1}{\tau_3}^{-1}}=G^{\sigma_3} \cap G^{\tau_3},
  \end{align*}
  we can rephrase the definition above in this case as follows.
	
  A homogeneous space $G/K$ is $\varmathbb{Z}_3 \times \varmathbb{Z}_3$-symmetric space if $G$ acts almost effectively on $G/K$ and there exist $\tilde{\sigma}_3,\tilde{\tau}_3 \in \Aut(G)\backslash\{1\}$ such that $(\tilde{\sigma}_3)^3=(\tilde{\tau}_3)^3=1,\tilde{\sigma}_3 \not= \tilde{\tau}_3$ and 
	$\tilde{\sigma}_3\tilde{\tau}_3=\tilde{\tau}_3\tilde{\sigma}_3$,  
	and moreover the subgroup $K$ of $G$ satisfies the condition $(G^{\sigma_3} \cap G^{\tau_3})_0 \subseteq K \subseteq G^{\sigma_3} \cap G^{\tau_3}$. 
	
  In particular, in the case where $G$ is a connected compact exceptional Lie group, we say the globally $\varmathbb{Z}_3 \times \varmathbb{Z}_3$-symmetric spaces $G/K$ as the globally {\it exceptional} $\varmathbb{Z}_3 \times \varmathbb{Z}_3$-symmetric spaces. 
	
	Here, for a globally exceptional $\varmathbb{Z}_3 \times \varmathbb{Z}_3$-symmetric space $G/K$, the normal subgroup $N \subset K$ of $G$ is discrete, and we explain concretely its result as follows:  
	\begin{align*}
	&N=\{1\}\,\, {\text {in the case where}}\,\, G=G_2,F_4, E_8,
	\\
	&N=\Z_3 \,\, {\text {in the case where}} \,\, G=E_6,\,\,
	N=\Z_2 \,\, {\text {in the case where}} \,\,G=E_7.
	\end{align*}
	Hence, when
	we define the action to $G/K$ of $G$ as $f:G \times G/K \to G/K, f(g,g'K)=g(g'K)$, we see that $G$ acts almost effectively on $G/K$ from the result above. 
	Besides, note that it follows from $\sigma_3\tau_3=\tau_3\sigma_3$ that $\tilde{\sigma}_3 \in \Aut((G)^{\tau_3})$ and $\tilde{\tau}_3 \in \Aut((G)^{\sigma_3})$.
	
	Now, in order to construct the globally exceptional $\varmathbb{Z}_3 \times \varmathbb{Z}_3$-symmetric spaces, we give inner automorphisms $\tilde{\sigma}_3, \tilde{\tau}_3$ of order $3$ on $G=G_2, F_4,E_6$ explicitly and determine the structure of the group $G^{\sigma_3} \cap G^{\tau_3}$.
Our results are as follows.
\vspace{3mm}
	
		\begin{center}
		
      {Table}
      \vspace{3mm}
      
			\begin{tabular}{ccll}
			\noalign{\hrule height 0.8pt}
			&&&
			\\[-4mm]
			Case & $G$ & $\Aut(G)$ & \hspace*{40mm}$G^{\sigma_3} \cap G^{\tau_3}$
			\\
			&&&
			\\[-5mm]
			\hline
			&&&
			\\[-4mm]
			1 &$G_2$ 
			& $\tilde{\gamma}_3,\tilde{w}_3$	
			& $(U(1) \times U(1))/\Z_2$
			\\
			&&&
			\\[-5mm]
			\hline
			&&&
			\\[-4mm]
			2 & 
			& $\tilde{\gamma}_3,\tilde{\sigma}_3$
			& $(U(1) \times Sp(1) \times U(2))/\Z_2$
			\\
			&&&
			\\[-5mm]
			&&&
			\\[-4mm]
			3 &$F_4$ & $\tilde{\gamma}_3,\tilde{w}_3$ 
			& $(U(1) \times U(1)\times SU(3))/\Z_3$
			\\
			&&&
			\\[-5mm]
			&&&
			\\[-4mm]
			4 & 
			& $\tilde{\sigma}_3,\tilde{w}_3$
			& $(SU(3) \times U(1) \times U(1))/\Z_3$
			\\
			&&&
			\\[-5mm]
			\hline
			&&&
			\\[-4mm]
			5 & 
			&$\tilde{\gamma}_3,\tilde{\sigma}_3$ 
			& $(U(1)\times U(1)\times U(1)
			\times SU(2)\times SU(2)\times SU(2))/(\Z_2^{\,\times 4})$
			\\
			&&&
			\\[-5mm]
			&&&
			\\[-4mm]
			6 & 
			& $\tilde{\gamma}_3,\tilde{\nu}_3$ 
			& $(U(1) \times U(1)\times SU(5))/(\Z_2\times \Z_5)$
			\\
			&&&
			\\[-5mm]
			&&&
			\\[-4mm]
			7 & 
			& $\tilde{\gamma}_3,\tilde{\mu}_3$ 
			& $(U(1)\times U(1) \times U(1)\allowbreak \times U(1) \times SU(2)\times SU(2))/(\Z_2\times\Z_2\times\Z_4)$
			\\
			&&&
			\\[-5mm]
			&&&
			\\[-4mm]
			8 & 
			& $\tilde{\gamma}_3,\tilde{w}_3$ 
			& $(U(1) \times U(1) \times SU(3)\times SU(3))/(\Z_2 \times \Z_3)$
			\\
			&&&
			\\[-5mm]
			&&&
			\\[-4mm]
			9 & $E_6$ 
			& $\tilde{\sigma}_3,\tilde{\nu}_3$ 
			& $(Sp(1)\times U(1) \times U(1)\times U(1) \allowbreak \times SU(2)\times SU(2))/(\Z_2\times\Z_2\times\Z_4)$
			\\
			&&&
			\\[-5mm]
			&&&
			\\[-4mm]
			10 & 
			& $\tilde{\sigma}_3,\tilde{\mu}_3$ 
			& $(U(1)\times Spin(2)\times Spin(8))/(\Z_2\times \Z_4)$
			\\
			&&&
			\\[-5mm]
			&&&
			\\[-4mm]
			11 & 
			& $\tilde{\sigma}_3,\tilde{w}_3$ 
			& $(SU(3)\times U(1)\times U(1)\times U(1)\times U(1))/\Z_3$
			\\
			&&&
			\\[-5mm]
			&&&
			\\[-4mm]
			12 &  
			& $\tilde{\nu}_3,\tilde{\mu}_3$ 
			& $(Sp(1)\times U(1) \times U(1)\times U(1) \allowbreak \times SU(2)\times SU(2))/(\Z_2\times\Z_2\times\Z_4)$
			\\
			&&&
			\\[-5mm]
			&&&
			\\[-4mm]
			13 & 
			& $\tilde{\nu}_3,\tilde{w}_3$ 
			& $(Sp(1) \times U(1)\times U(1) \times SU(2)\times SU(3))/(\Z_2\times \Z_2\times \Z_3)$
			\\
			&&&
			\\[-5mm]
			&&&
			\\[-4mm]
			14 & 
			& $\tilde{\mu}_3,\tilde{w}_3$  
			& $(SU(3)\times U(1)\times U(1)\times U(1)\times U(1))/\Z_3$
			\\
			&&&
			\\[-5mm]
			\noalign{\hrule height 0.9pt}
			\end{tabular}
		\end{center}
\vspace{4mm}

We use the same notations as in \cite{miya1}, \cite{iy1}, \cite{realization G_2} and \cite{iy0}.
Finally, the author would like to say that the features of this article are to give elementary proofs of the isomorphism of groups using homomorphism theorem. In the near future, for the case where $G=E_7$ and $E_8$, we will provide some realizations of globally exceptional $\varmathbb{Z}_3 \times \varmathbb{Z}_3$-symmetric spaces as Part II and Part III.
	
\section{Preliminaries}

\subsection{Cayley algebra and compact Lie group of type ${\rm G}_2$} 

Let $\mathfrak{C}=\{e_0 =1, e_1, e_2, e_3, e_4, e_5, e_6, e_7 \}_{\sR}$ be the division Cayley algebra. In $\mathfrak{C}$, since the multiplication and the inner product are well known, these are omitted.
\vspace{1mm}

The connected compact Lie group of type ${\rm G_2}$ is given by
$$
G_2 =\{\alpha \in \Iso_{\sR}(\mathfrak{C})\,|\, \alpha(xy)=(\alpha x) (\alpha y) \}\vspace{1mm}.
$$ 
\subsection{Exceptional Jordan algebra and compact Lie group of type ${\rm F}_4$} 

Let  
$\mathfrak{J}(3,\mathfrak{C} ) = \{ X \in M(3, \mathfrak{C}) \, | \, X^* = X \}$ be the 
exceptional Jordan algebra. 
In $\mathfrak{J}(3,\mathfrak{C} )$, the Jordan multiplication $X \circ Y$, the 
inner product $(X,Y)$ and a cross multiplication $X \times Y$, called the Freudenthal multiplication, are defined by
$$
\begin{array}{c}
X \circ Y = \dfrac{1}{2}(XY + YX), \quad (X,Y) = \tr(X \circ Y),
\vspace{1mm}\\
X \times Y = \dfrac{1}{2}(2X \circ Y-\tr(X)Y - \tr(Y)X + (\tr(X)\tr(Y) 
- (X, Y))E), 
\end{array}$$
respectively, where $E$ is the $3 \times 3$ unit matrix. Moreover, we define the trilinear form $(X, Y, Z)$, the determinant $\det \,X$ by
$$
(X, Y, Z)=(X, Y \times Z),\quad \det \,X=\dfrac{1}{3}(X, X, X),
$$
respectively, and briefly denote $\mathfrak{J}(3, \mathfrak{C})$
by $\mathfrak{J}$.
\vspace{1mm}

The connected compact Lie group of type ${\rm F_4}$ is given by
\begin{align*}
	F_4 &= \{\alpha \in \Iso_{\sR}(\mathfrak{J}) \, | \, \alpha(X \circ Y) = \alpha X \circ \alpha Y \}
	\\[1mm]
	&=  \{\alpha \in \Iso_{\sR}(\mathfrak{J}) \, | \, \alpha(X \times Y) = \alpha X \times \alpha Y \}. 
\end{align*}
Then we have naturally the inclusion $G_2 \subset F_4$ as follows:
\begin{align*}
\varphi:G_2 \to F_4,\,\,\varphi(\alpha)X=\begin{pmatrix}
\xi_1 & \alpha x_3 & \ov{\alpha x_2} \\
\ov{\alpha x_3} & \xi_2 & \alpha x_1 \\ 
\alpha x_2 & \ov{\alpha x_1} & \xi_3
\end{pmatrix},\,\, X \in \mathfrak{J}.
\end{align*} 
\subsection{Complex exceptional Jordan algebra and Compact Lie group of type ${\rm E}_6$} 
Let $\mathfrak{J}(3,\mathfrak{C})^C = \{ X \in M(3, \mathfrak{C})^C \, | \, X^* = X \}$ be the complexification of the exceptional Jordan algebra $\mathfrak{J}$. In $\mathfrak{J}(3,\mathfrak{C})^C$, as in $\mathfrak{J}$, we can also define the multiplication $X \circ Y, X \times Y$, the inner product $(X, Y)$, the trilinear forms $(X, Y, Z)$ and the determinant $\det \, X$ in the same manner, and those have the same properties. The  $\mathfrak{J}(3,\mathfrak{C} )^C$ is called the complex exceptional Jordan algebra, and briefly denote $\mathfrak{J}(3, \mathfrak{C})^C$ by $\mathfrak{J}^C$. 
\vspace{1mm}

The connected compact Lie group of type ${\rm E_6}$ is given by
\begin{align*}
		E_6 &= \{\alpha \in \Iso_C(\mathfrak{J}^C) \, | \, \det\, \alpha X = \det\, X, \langle \alpha X, \alpha Y \rangle = \langle X, Y \rangle \}
		\\ 
		 &=\{\alpha \in \Iso_C(\mathfrak{J}^C) \, | \,\alpha X \times \alpha Y=\tau\alpha\tau(X \times Y) , \langle \alpha X, \alpha Y \rangle = \langle X, Y \rangle \}
\end{align*}
where $\tau$ is a complex conjugation in $\mathfrak{J}^C$: $\tau(X+iY)=X-iY, \,X, Y \in \mathfrak{J}$ and the Hermite inner product $\langle X, Y \rangle$ is defined by $(\tau X, Y)$.

\noindent Then we have naturally the inclusion $F_4 \subset E_6$ as follows:
\begin{align*}
   \varphi:F_4 \to E_6,\,\,\varphi(\alpha)(X_1+iX_2)=(\alpha X_1)+i(\alpha X_2),\,\,X_1+iX_2 \in \mathfrak{J}^C, X_i \in \mathfrak{J}.
\end{align*}

\section{The inner automorphisms of order $3$ and the fixed points subgroups by them}\label{section 3}

In this section, we will rewrite the inner automorphisms of order $3$ on $G=G_2, F_4,E_6$ and the fixed points subgroups of $G$ by them which were realized and determined in \cite{iy1}, in association with the involutive inner automorphisms. However, the detailed proofs are omitted.

\subsection{In $G_2$}\label{subsection 3.1}

Let $\mathfrak{C}=\H \oplus \H e_4$ be Cayley devision algebra, where $\H$ is the field of quaternion number. Since a multiplication, a conjugation and inner product in $\mathfrak{C}=\H \oplus \H e_4$ are well known, these are ommited. If necessary, refer to \cite{miya1},\cite{realization G_2} and \cite{iy0}.

We define an $\R$-linear transformation $\gamma$ of $\mathfrak{C}$ by 
\begin{align*}
		\gamma(m+ne_4)=m-ne_4, \,\, m+ne_4 \in \H \oplus \H e_4 = \mathfrak{C}.
\end{align*}
Then we have that $\gamma \in G_2$ and $\gamma^2 =1$. Hence $\gamma$ induces the involutive inner automorphism $\tilde{\gamma}$ on $G_2: \tilde{\gamma}(\alpha)=\gamma\alpha\gamma, \alpha \in G_2$, so we have the following well-known result.

\begin{proposition}\label{proposition 3.1.1}
	The group $(G_2)^\gamma$ is isomorphism to the group $(Sp(1) \times Sp(1))/\Z_2${\rm:}  $(G_2)^\gamma \cong  (Sp(1) \times Sp(1))/\Z_2, $ $ \Z_2=\{ (1,1), (-1,-1) \}$.
\end{proposition}
\begin{proof}
	We define a mapping $\varphi_{{}_{311}}: Sp(1) \times Sp(1) \to (G_2)^\gamma$ by 
	\begin{align*}
	\varphi_{{}_{G_2,\gamma}}(p, q)(m+n e_4)=qm \ov{q}+(pn \ov{q}) e_4, \,\,\,m+n e_4 \in \H \oplus \H e_4 =\mathfrak{C}.
	\end{align*}
	This mapping induces the required isomorphism (see \cite [Theorem 1.10.1]{iy0} in detail).
\end{proof}

Let $\bm{\omega}=-(1/2)+(\sqrt{3}/2)e_1 \in U(1) \subset \C \subset \H \subset \mathfrak{C}$. We define an $\R$-linear transformation $\gamma_3$ of $\mathfrak{C}$ by
\begin{align*}
		\gamma_3(m+ne_4)=m+(\bm{\omega} n)e_4, \,\,m+ne_4 \in \H \oplus \H e_4=\mathfrak{C}.
\end{align*}
Then, using the mapping $\varphi_{{}_{G_2, \gamma}}$ above, since $\gamma_3$ is expressed by $\varphi_{{}_{G_2,\gamma}}(\bm{\omega},1)$: $\gamma_3=\varphi_{{}_{G_2,\gamma}}(\bm{\omega},1)$, it is clear that $\gamma_3 \in G_2$ and $(\gamma_3)^3=1$. Hence $\gamma_3$ induces the inner automorphism $\tilde{\gamma}_3$ of order $3$ on $G_2: \tilde{\gamma}_3(\alpha)={\gamma_3}^{-1}\alpha\gamma_3, \alpha \in G_2$. 
\vspace{1mm}

Now, we have the following theorem.

\begin{theorem}\label{theorem 3.1.2}
	The group $(G_2)^{\gamma_3}$ is isomorphism to the group $(U(1) \times Sp(1))/\Z_2${\rm:}  $(G_2)^{\gamma_3} \cong  (U(1) \times Sp(1))/\Z_2,  \Z_2=\{ (1,1), (-1,-1) \}$.
\end{theorem}
\begin{proof}
	Let $U(1)=\{a \in \C \,|\,\ov{a}a=1 \} \subset Sp(1)$, where $\C=\{x+ye_1\,|\, x,y \in \R \}$. Then we define a mapping $\varphi_{{}_{G_2,\gamma_3}}:U(1) \times Sp(1) \to (G_2)^{\gamma_3}$ by the restriction of the mapping $\varphi_{{}_{G_2,\gamma}}$ (Proposition \ref{proposition 3.1.1}). This mapping induces the required isomorphism (see \cite [Theorem 1.2]{iy1} in detail).
\end{proof}

Thus, since the group $(G_2)^{\gamma_3}$ is connected, together with the result of Theorem \ref{theorem 3.1.2}, we have an exceptional $\varmathbb{Z}_3$-symmetric space $G_2/(U(1) \times Sp(1))/\Z_2)$.
\vspace{2mm}

Let $x = m_0 + m_1e_2 + m_2e_4 + m_3e_6 \in \mathfrak{C}, m_i \in \C$. Then we associate such elements $x$ of $\mathfrak{C}$ with the elements 
\begin{align*}
			m_0 + \begin{pmatrix}
			m_1 \\
			m_2 \\
			m_3
			\end{pmatrix}(=:m_0+\m)
\end{align*}
of $\C \oplus \C^3$ and we can define a multiplication, a conjugation and an inner product in $\C \oplus \C^3$ corresponding to the same ones in $\mathfrak{C}$ (see \cite[Subsection 1.5]{iy0} in detail). Hence we have that $\C \oplus \C^3$ is isomorphic to $\mathfrak{C}$ as algebra. Hereafter, if necessary, we identify $\mathfrak{C}$ with $\C \oplus \C^3$: $\mathfrak{C}=\C \oplus \C^3$. 


Again let $\bm{\omega}=-(1/2)+(\sqrt{3}/2)e_1 \in U(1) \subset \C \subset \H \subset \mathfrak{C}$. We define an $\R$-linear transformation $w_3$ of $\mathfrak{C}=\C \oplus \C^3$ by
\begin{align*}
		w_3(m_0+\m)=m_0+\bm{\omega} \m, \,\,m_0+\m \in \C \oplus \C^3=\mathfrak{C}.
\end{align*}
Then we have that $w_3 \in G_2$ (\cite[Proposition 1.4]{iy1}) and $(w_3)^3=1$. Hence $w_3$ induces the inner automorphism $\tilde{w}_3$ of order $3$ on $G_2$: $\tilde{w}_3(\alpha)={w_3}^{-1}\alpha w_3, \alpha \in G_2$.
\vspace{1mm}

Now, we have the following theorem.

\begin{theorem}\label{theorem 3.1.3}
	The group $(G_2)^{w_3}$ is isomorphic to the group $SU(3)${\rm :} $(G_2)^{w_3} \cong SU(3)$.
\end{theorem}
\begin{proof}
We define a mapping $\varphi_{{}_{G_2,w_3}}: SU(3) \to (G_2)^{w_3}$ by
	\begin{align*}
			\varphi_{{}_{G_2,w_3}}(A)(m_0+\m)=m_0+A\m, \,\,m_0+\m \in \C \oplus \C^3=\mathfrak{C}.
	\end{align*}
	This mapping induces the required isomorphism (see \cite[Theorem 1.6]{iy1} in detail).
\end{proof}

Thus, since the group $(G_2)^{w_3}$ is connected, together with the result of Theorem \ref{theorem 3.1.3}, we have an exceptional $\varmathbb{Z}_3$-symmetric space $G_2/SU(3)$. As is well known, this space is homeomorphic to a $6$-dimensional sphere $S^6$: $G_2/SU(3) \simeq S^6$.   
\vspace{2mm}

The following lemma are useful to determine the structure of groups $G^{\sigma_3} \cap G^{\tau_3}$ in $G_2$.

\begin{lemma}\label{lemma 3.1.4}
	{\rm (1)} The mapping $\varphi_{{}_{G_2,\gamma_3}}:U(1) \times Sp(1) \to (G_2)^{\gamma_3}$ of \,Theorem {\rm \ref{theorem 3.1.2}} satisfies the relational formulas 
	\begin{align*}
 	\gamma_3&=\varphi_{{}_{G_2,\gamma_3}}(\bm{\omega},1),
 	\\
 	w_3&=\varphi_{{}_{G_2,\gamma_3}}(1, \ov{\bm{\omega}}),
	\end{align*}
 	where $\bm{\omega}=-(1/2)+(\sqrt{3}/2)e_1 \in U(1)$.
\vspace{1mm}

	{\rm (2)} The mapping $\varphi_{{}_{G_2,w_3}}:SU(3) \to (G_2)^{w_3}$ of \,Theorem {\rm \ref{theorem 3.1.3}} satisfies the relational formulas
	\begin{align*}
	\gamma_3&=\varphi_{{}_{G_2,w_3}}(\diag(1,\bm{\omega},\ov{\bm{\omega}})),
	\\
	w_3&=\varphi_{{}_{G_2,w_3}}(\bm{\omega}E),
	\end{align*}
	where $\bm{\omega}=-(1/2)+(\sqrt{3}/2)e_1 \in U(1)$. 
\end{lemma}
\begin{proof}
	(1), (2) By doing straightforward computation we obtain the results above. 
\end{proof}

\subsection{In $F_4$}\label{subsection 3.2}

Let $\mathfrak{J}$ be the exceptional Jordan algebra. As is well known, the elements $X$ of $\mathfrak{J}$ take the form 
$$
X = \begin{pmatrix}
\xi_1 & x_3 & \ov{x_2} \\
\ov{x_3} & \xi_2 & x_1 \\ 
x_2 & \ov{x_1} & \xi_3
\end{pmatrix},\,\, \xi_i \in \R,\, x_i \in \mathfrak{C},\, i=1, 2, 3.
$$
Hereafter, in $\mathfrak{J}$, we use the following nations:
\begin{align*}
E_1 &= \left(\begin{array}{ccc}
1 & 0 & 0 \\
0 & 0 & 0 \\
0 & 0 & 0
\end{array}
\right),  \,\,\,\,\,\,\,\,
E_2 = \left(\begin{array}{ccc}
0 & 0 & 0 \\
0 & 1 & 0 \\
0 & 0 & 0
\end{array}
\right),  \,\,\,\,\,\,\,\,\,\,
E_3 = \left(\begin{array}{ccc}
0 & 0 & 0 \\
0 & 0 & 0 \\
0 & 0 & 1
\end{array}
\right),    
\\[2mm]
F_1 (x) &= \left(\begin{array}{ccc}
0 &      0 & 0 \\
0 &      0 & x \\
0 & \ov{x} & 0
\end{array}
\right),  \,\,
F_2(x) = \left(\begin{array}{ccc}
0 & 0 & \ov{x} \\
0 & 0 & 0 \\
x & 0 & 0
\end{array}
\right),  \,\,
F_3 (x) = \left(\begin{array}{ccc}
0 & x & 0 \\
\ov{x} & 0 & 0 \\
0 & 0 & 0
\end{array}
\right).
\end{align*}

\vspace{1mm}

We define an $\R$-linear transformation $\gamma$ of $\mathfrak{J}$ by
$$
\gamma X= \begin{pmatrix} \xi_1 & \gamma x_3 & \ov{\gamma x_2} \\
\ov{\gamma x_3} & \xi_2 & \gamma x_1 \\
\gamma x_2 & \ov{\gamma x_1} & \xi_3   \end{pmatrix}
,\,\,X \in \mathfrak{J},
$$
where $\gamma$ on right hand side is the same one as $\gamma \in G_2$. Then we have that $\gamma \in F_4$ and $\gamma^2 =1$. Hence $\gamma$ induce involutive inner automorphism $\tilde{\gamma}$ of $F_4{\rm :}\,\tilde{\gamma}(\alpha)=\gamma\alpha\gamma, \alpha \in F_4$.
\vspace{1mm}

Here, we associate the elements $X$ of $\mathfrak{J}$ with the elements 
\begin{align*}
\begin{pmatrix}
\xi_1 & m_3 & \ov{m_2} \\
\ov{m_3} & \xi_2 & m_1 \\ 
m_2 & \ov{m_1} & \xi_3
\end{pmatrix}
+ (\a_1, \a_2, \a_3)(=:M + \a) 
\end{align*}
of $\mathfrak{J}(3, \H) \oplus \H^3$ and we can define a multiplication, a conjugation and an inner product in $\mathfrak{J}(3, \H) \oplus \H^3$ corresponding to the same ones in $\mathfrak{J}$ (see \cite[Subsection 2.11]{iy0} in detail). 
Hence we have that $\mathfrak{J}(3, \H) \oplus \H^3$ is isomorphic to the exceptional Jordan algebra $\mathfrak{J}$ as algebra. From now on, if necessary we identify $\mathfrak{J}$ with $\mathfrak{J}(3, \H) \oplus \H^3$: $\mathfrak{J}=\mathfrak{J}(3, \H) \oplus \H^3$.
Note that the action to $\mathfrak{J}(3, \H) \oplus \H^3$ of $\gamma$ is as follows.
\begin{align*}
		\gamma(M+\a)=M-\a,\,\,M+\a \in \mathfrak{J}(3, \H) \oplus \H^3=\mathfrak{J}.
\end{align*}

Then we have the following well-known result.

\begin{proposition}\label{proposition 3.2.1}
	The group $(F_4)^\gamma$ is isomorphic to the group $(Sp(1) \times Sp(3))/\Z_2${\rm:}  $(F_4)^\gamma \cong (Sp(1) \times Sp(3))/\Z_2, \,$ $\Z_2 =\{(1, E), (-1, -E) \}$.
\end{proposition}
\begin{proof}
	We define a mapping $\varphi_{{}_{F_4,\gamma}}: Sp(1) \times Sp(3) \to (F_4)^\gamma$ by
	$$
	\varphi_{{}_{F_4,\gamma}}(p, A)(M+\a)=AMA^* +p\a A^*,\,\,\, M+\a \in \mathfrak{J}(3, \H) \oplus \H^3=\mathfrak{J}.
	$$
	This mapping induces the required isomorphism (see \cite[Theorem 2.11.2]{iy0} in detail).
\end{proof}
\vspace{1mm}

Let $\gamma_3 \in G_2$ be the $\R$-linear transformation of $\mathfrak{C}$. Using the inclusion $G_2 \subset F_4$, $\gamma_3$ is naturally extended to the $\R$-linear transformation of $\mathfrak{J}$. The explicit form of 
$\gamma_3$ as action to $\mathfrak{J}$ is as follows.
\begin{align*}
		\gamma_3 X=
		\begin{pmatrix} \xi_1 & \gamma_3 x_3 & \ov{\gamma_3 x_2} \\
		\ov{\gamma_3 x_3} & \xi_2 & \gamma_3 x_1 \\
		\gamma_3 x_2 & \ov{\gamma_3 x_1} & \xi_3   
		\end{pmatrix},\,\,X \in \mathfrak{J},
\end{align*}
where $\gamma_3$ on the right hand side is the same one as $\gamma_3 \in G_2$. Needless to say, $\gamma_3 \in F_4$ and $(\gamma_3)^3=1$. Hence $\gamma_3$ induces the automorphism $\tilde{\gamma}_3$ of order $3$ on $F_4$: $\tilde{\gamma}_3(\alpha)={\gamma_3}^{-1}\alpha\gamma_3, \alpha \in F_4$. Note that the action to $\mathfrak{J}(3, \H) \oplus \H^3$ of $\gamma_3$ is as follows.
\begin{align*}
\gamma_3(M+\a)=M+\bm{\omega}\a,\,\,M+\a \in \mathfrak{J}(3, \H) \oplus \H^3=\mathfrak{J}.
\end{align*}

Now, we have the following theorem.

\begin{theorem}\label{theorem 3.2.2}
	The group $(F_4)^{\gamma_3}$ is isomorphic to the group $(U(1) \times Sp(3))/\Z_2$ {\rm :} $(F_4)^{\gamma_3} \cong (U(1) \times Sp(3))/\Z_2, \Z_2=\{(1,E), 
	(-1,-E) \}$.
\end{theorem}
\begin{proof}
	As in the proof of Theorem \ref{theorem 3.1.2}, let $U(1)=\{a \in \C \,|\,\ov{a}a=1 \} \subset Sp(1)$. We define a mapping $\varphi_{{}_{F_4,\gamma_3}}:U(1) \times Sp(3) \to (F_4)^{\gamma_3}$ by the restriction of the mapping $\varphi_{{}_{F_4,\gamma}}$ (Proposition \ref{proposition 3.2.1}). This mapping induces the required isomorphism (see \cite [Theorem 2.2]{iy1} in detail).	
\end{proof}

Thus, since the group $(F_4)^{\gamma_3}$ is connected, together with the result of Theorem \ref{theorem 3.2.2}, we have an exceptional $\varmathbb{Z}_3$-symmetric space $F_4/((U(1) \times Sp(3))/\Z_2)$.
\vspace{1mm}

We define an $\R$-linear transformation $\sigma$ of $\mathfrak{J}$ by
\begin{align*}
\sigma X= \begin{pmatrix} \xi_1 & -x_3 & -\ov{x_2} \\
-\ov{x_3} & \xi_2 & x_1 \\
-x_2 & \ov{x_1} & \xi_3   \end{pmatrix}
,\,\,X \in \mathfrak{J},
\end{align*}
Then we have that $\sigma \in F_4$ and $\sigma^2 =1$. Hence $\sigma$ induce involutive inner automorphism $\tilde{\sigma}$ on $F_4{\rm :}\,\tilde{\sigma}(\alpha)=\sigma\alpha\sigma, \alpha \in F_4$.
\vspace{1mm}

Then we have the following well-known result.

\begin{proposition}\label{proposition 3.2.3}
	The group $(F_4)^\sigma$ is isomorphic to the group $Spin(9)${\rm:}$(F_4)^\sigma \!\cong \!Spin(9)$.
\end{proposition}
\begin{proof}
	From \cite[Thorem 2.7.4]{iy0}
	, we have $(F_4)_{E_1} \cong Spin(9)$, so by proving that $(F_4)^\sigma \cong (F_4)_{E_1}$ (\cite[Thorem 2.9.1]{iy0}) we have the required isomorphism (see \cite[Sections 2.7, 2.9 ]{iy0} in detail).
\end{proof}
\vspace{1mm}

Let $U(1)=\{a \in \C \,|\,\ov{a}a=1 \}$. For $a \in U(1)$, we define an $\R$-linear transformation $D_a$ of $\mathfrak{J}$ by
\begin{align*}
		D_a X= 
		\begin{pmatrix} \xi_1 & x_3 a & \ov{ax_2} \\
		\ov{x_3 a} & \xi_2 & \ov{a}x_1\ov{a} \\
		a x_2 & a\ov{x_1}a & \xi_3   
		\end{pmatrix},\,\, X \in \mathfrak{J}.
\end{align*}
Then, since $D_a=\varphi_{{}_{F_4,\gamma}}(1,\diag(1,\ov{a},a))$, we have that $D_a \in F_4$. Hence, by corresponding $a \in U(1)$ to $D_a \in F_4$, $U(1)$ is embedded into $F_4$.
In addition, we can express $\sigma$ defined above by $D_{-1}$: $\sigma=D_{-1}$.

Let $\bm{\omega}=-(1/2)+(\sqrt{3}/2)e_1 \in U(1)$. Then we define an $\R$-linear transformation $\sigma_3$ of $\mathfrak{J}$ by
\begin{align*}
		\sigma_3X= 
		\begin{pmatrix} \xi_1 & x_3 \bm{\omega} & \ov{\bm{\omega} x_2} \\
		\ov{x_3 \bm{\omega}} & \xi_2 & \ov{\bm{\omega}}x_1\ov{\bm{\omega}} \\
		\bm{\omega} x_2 & \bm{\omega}\ov{x_1}\bm{\omega} & \xi_3   
		\end{pmatrix},\,\, X \in \mathfrak{J}.
\end{align*}
Needless to say, since $\sigma_3=D_\omega=\varphi_{{}_{F_4,\gamma}}(1,\diag(1,\ov{\bm{\omega}},\bm{\omega}))$, we have that $\sigma_3 \in F_4$. Hence $\sigma_3$ induces the automorphism $\tilde{\sigma}_3$ of order $3$ on $F_4$: $\tilde{\sigma}_3(\alpha)={\sigma_3}^{-1}\alpha\sigma_3, \alpha \in F_4$.
\vspace{1mm}

Now, we have the following theorem.

\begin{theorem}\label{theorem 3.2.4}
	The group $(F_4)^{\sigma_3}$ is isomorphic to the group $(Spin(2) \times Spin(7))/\Z_2${\rm:} $(F_4)^{\sigma_3} \cong (Spin(2) \times Spin(7))/\Z_2, \Z_2=\{(1,1), (\sigma,\sigma)\}$.
\end{theorem}
\begin{proof}
	Let $Spin(2)$ as the group $\{D_a \in F_4 \,|\,a \in U(1) \}$ defined above which is isomorphic to the group $U(1)$ and $Spin(7)$ as the subgroup $(F_4)_{E_1, F_1(1),F_1(e_1)}$ of $F_4$ (cf. \cite[Propsition 2.9 (1)]{iy2}, \cite[Subsection 2.2]{iy1}). We define a mapping $\varphi_{{}_{F_4,\sigma_3}}: Spin(2) \times Spin(7) \to (F_4)^{\sigma_3}$ by
	\begin{align*}
			\varphi_{{}_{F_4,\sigma_3}}(D_a, \beta)=D_a \beta.
	\end{align*}	
	This mapping induces the required isomorphism (see \cite[Lemmas 2.5, 2.6, Theorem 2.7]{iy1} in detail).
\end{proof}

Thus, since the group $(F_4)^{\sigma_3}$ is connected, together with the result of Theorem \ref{theorem 3.2.4}, we have an exceptional $\varmathbb{Z}_3$-symmetric space $F_4/((Spin(2) \times Spin(7))/\Z_2)$.
\vspace{2mm}	

 We define an $\R$-linear transformation $w_3$ of $\mathfrak{J}$ by
\begin{align*}
w_3X= 
\begin{pmatrix} \xi_1 & w_3 x_3 & \ov{w_3 x_2} \\
\ov{w_3 x_3} & \xi_2 & w_3 x_1 \\
w_3 x_2 & \ov{w_3 x_1} & \xi_3   
\end{pmatrix},\,\, X \in \mathfrak{J},
\end{align*}
where $w_3$ on the right hand side is the same one as $w_3 \in G_2$. Needless to say, $w_3 \in F_4$ and $(w_3)^3=1$. Hence $w_3$ induces the automorphism $\tilde{w}_3$ of order $3$ on $F_4$: $\tilde{w}_3(\alpha)={w_3}^{-1}\alpha w_3, \alpha \in F_4$.

We associate the elements $X$ of $\mathfrak{J}$ with the elements 
\begin{align*}
		\begin{pmatrix}
		\xi_1 & c_3 & \ov{c_2} \\
		\ov{c_3} & \xi_2 & c_1 \\ 
		c_2 & \ov{c_1} & \xi_3
		\end{pmatrix} +
		\begin{pmatrix}
		   &  &  \\
		\m_1 \!\!\!& \m_2 \!\!\!& \m_3 \\ 
		   &  & 
		\end{pmatrix}(=:X_{\bm{C}}+M)
\end{align*}
of $\mathfrak{J}(3,\C) \oplus M(3,\C)$, where $\m_i \in \C^3$, 
and we can define a multiplication, a conjugation and an inner product in 
$\mathfrak{J}(3, \C) \oplus M(3,\C)$ corresponding to the same ones in $\mathfrak{J}$ (see \cite[Subsection 2.12]{iy0} in detail). Hence we have that $\mathfrak{J}(3, \C) \oplus M(3,\C)$ is isomorphic to $\mathfrak{J}$ as algebra. Hereafter, if necessary we identify $\mathfrak{J}$ with $\mathfrak{J}(3, \C) \oplus M(3,\C)$: $\mathfrak{J}=\mathfrak{J}(3, \C) \oplus M(3,\C)$. Note that using $\bm{\omega}=-(1/2)+(\sqrt{3}/2)e_1 \in \C$, the action to $\mathfrak{J}=\mathfrak{J}(3, \C) \oplus M(3,\C)$ of $w_3$ is as follows.
\begin{align*}
w_3(X_{\bm{C}}+M)=X_{\bm{C}}+\bm{\omega} M,\,\,X_{\bm{C}}+M \in \mathfrak{J}(3, \C) \oplus M(3,\C)=\mathfrak{J}.
\end{align*}

Now, we have the following theorem.

\begin{theorem}\label{theorem 3.2.5}
		The group $(F_4)^{w_3}$ is isomorphic to the group $(SU(3) \times SU(3))/\Z_3${\rm :} $(F_4)^{w_3} \cong (SU(3) \times SU(3))/\Z_3, \Z_3=\{(E,E),(\bm{\omega} E,\bm{\omega} E),({\bm{\omega}}^{-1}E,{\bm{\omega}}^{-1}E) \}$.
\end{theorem}
\begin{proof}
	We define a mapping $\varphi_{F_4,w_3}:SU(3) \times SU(3) \to  (F_4)^{w_3}$ by
	\begin{align*}
				\varphi_{F_4,w_3}(B, A)(X_{\bm{C}}+M)=AX_{\bm{C}}A^* + BMA^*,\,\,X_{\bm{C}}+M \in \mathfrak{J}(3, \C) \oplus M(3,\C)=\mathfrak{J}.
	\end{align*}
	This mapping induces the required isomorphism (see \cite[Theorem 2.9]{iy1} in detail).
\end{proof}

Thus, since the group $(F_4)^{w_3}$ is connected, together with the result of Theorem \ref{theorem 3.2.5}, we have an exceptional $\varmathbb{Z}_3$-symmetric space $F_4/((SU(3) \times SU(3))/\Z_3)$.
\vspace{1mm}

As in Section 3.1, the following lemma are useful to determine the structure of a group $G^{\sigma_3} \cap G^{\tau_3}$ in $F_4$.

\begin{lemma}\label{lemma 3.2.6}
	{\rm (1)} The mapping $\varphi_{{}_{F_4,\gamma_3}}:U(1) \times Sp(3) \to (G_2)^{\gamma_3}$ of \,Theorem {\rm \ref{theorem 3.2.2}} satisfies the relational formulas 
	\begin{align*}
	\gamma_3&=\varphi_{{}_{F_4,\gamma_3}}(\bm{\omega},E), 
	\\
	\sigma_3&=\varphi_{{}_{F_4,\gamma_3}}(1,\diag(1,\ov{\bm{\omega}},\bm{\omega})),
	\\
	w_3&=\varphi_{{}_{F_4,\gamma_3}}(1, \ov{\bm{\omega}}E), 
	\end{align*}
 where $\bm{\omega}=-(1/2)+(\sqrt{3}/2)e_1 \in U(1)$.
	\vspace{1mm}
	
	{\rm (2)} The mapping $\varphi_{{}_{F_4,w_3}}:SU(3)\times SU(3) \to (F_4)^{w_3}$ of \,Theorem {\rm \ref{theorem 3.2.5}} satisfies the relational formulas
	\begin{align*}
	\gamma_3&=\varphi_{{}_{F_4,w_3}}(\diag(1,\bm{\omega},\ov{\bm{\omega}}),E), 
	\\
	\sigma_3&=\varphi_{{}_{F_4,w_3}}(E,\diag(1,\ov{\bm{\omega}},\bm{\omega}))\\
	w_3&=\varphi_{{}_{F_4,w_3}}(\bm{\omega}E,E), 
	\end{align*}
 where $\bm{\omega}=-(1/2)+(\sqrt{3}/2)e_1 \in U(1)$. 	
\end{lemma}
\begin{proof}
	(1), (2) By doing straightforward computation we obtain the results above. 
\end{proof}

\subsection{In $E_6$}\label{subsection 3.3}

Let $\gamma, \gamma_3 \in G_2 \subset F_4$, and using the inclusion $F_4 \subset E_6$, 
$\gamma, \gamma_3$ are naturally extended to an $C$-linear 
transformation of $\mathfrak{J}^C$. Needless to say, $\gamma, \gamma_3 \in E_6$ and $\gamma^2=(\gamma_3)^3=1$. Hence $\gamma, \gamma_3$ induce the involutive automorphism $\tilde{\gamma}$, the automorphism $\tilde{\gamma}_3$ of order $3$ on $E_6$, respectively: $\tilde{\gamma}(\alpha)=\gamma\alpha\gamma, \tilde{\gamma}_3(\alpha)={\gamma_3}^{-1}\alpha\gamma_3, \alpha \in E_6$. 
\vspace{1mm}

Then we have the following proposition and theorem.

\begin{proposition}\label{proposition 3.3.1}
	The group $(E_6)^\gamma$ isomorphic to the group $(Sp(1) \times SU(6))/\Z_2${\rm:}
	$(E_6)^\gamma \cong (Sp(1) \times SU(6))/\Z_2,\Z_2 =\{(1, E), (-1, -E) \}$.
\end{proposition}
\begin{proof}
	Let $SU(6)=\{A \in M(6, C)\,|\,(\tau\,{}^t A) A$ $=1, \det\, A=1) \}$, where $\tau$ is the complex conjugation of $C=\{x+iy \,|\,x,y \in \R \}$, that is, $\tau(x+yi)=x-yi, x,y \in \R$.
	We define a mapping $\varphi_{{}_{E_6,\gamma}}:Sp(1) \times SU(6) \to (E_6)^\gamma $ by
	\begin{align*}
	\varphi_{{}_{E_6,\gamma}}(p, A)(M+\a)={k_J}^{-1}(A(k_J M){}^t\!A)+p\a k^{-1}(\tau \,{}^t\!A), M+\a \in \mathfrak{J}(3, \H)^C \oplus (\H^3)^C=\mathfrak{J}^C,
	\end{align*}
	where both of $k_J:\mathfrak{J}(3, \H)^C \to \mathfrak{S}(6, C)$ and $k:M(3, \H)^C \to M(6, C)$ are the $C$-linear isomorphisms.
	This mapping induces the required isomorphism (see \cite[Theorem 3.11.4 ]{iy0} in detail).
\end{proof}

\begin{theorem}\label{theorem 3.3.2}
	The group $(E_6)^{\gamma_3}$ is isomorphic to the group $(U(1) \times SU(6))/\Z_2${\rm :} $(E_6)^{\gamma_3} \cong (U(1) \times SU(6))/\Z_2, \Z_2=\{(1,E),
	(-1,-E) \}$.	
\end{theorem}
\begin{proof}
	Let $U(1)=\{a \in \C\,|\, \ov{a}a=1 \} \subset Sp(1)$. We define a mapping $\varphi_{{}_{E_6,\gamma_3}}: U(1) \times SU(6) \to (E_6)^{\gamma_3}$ by the restriction of the mapping $\varphi_{{}_{E_6,\gamma}}$ (Proposition \ref{proposition 3.3.1}). This mapping induces the required isomorphism (see \cite [Theorem 3.2]{iy1} in detail). 
\end{proof}

Thus, since the group $(E_6)^{\gamma_3}$ is connected, together with the result of Theorem \ref{theorem 3.3.2}, we have an exceptional $\varmathbb{Z}_3$-symmetric space $E_6/((U(1) \times SU(6))/\Z_2)$.
\vspace{2mm}

Let $\sigma, \sigma_3 \in F_4$. Then, as in the case above, using the inclusion $F_4 \subset E_6$, $\sigma, \sigma_3$ are naturally extended to
transformations of $\mathfrak{J}^C$. Needless to say, $\sigma, \sigma_3 \in E_6$ and $\sigma^2=(\sigma_3)^3=1$. Hence $\sigma$ and $\sigma_3$ induce the involutive automorphism $\tilde{\sigma}$ and the automorphism $\tilde{\sigma}_3$ of order $3$ on $E_6$, respectively: $\tilde{\sigma}(\alpha)=\sigma\alpha\sigma, \tilde{\sigma}_3(\alpha)={\sigma_3}^{-1}\alpha\sigma_3, \alpha \in E_6$. 
\vspace{1mm}

Then we have the following proposition and theorem.

\begin{proposition}\label{proposition 3.3.3}
	The group $(E_6)^\sigma$ is isomorphic to the group $(U(1) \times Spin(10))/\Z_4${\rm:}\,
	$(E_6)^\sigma \!\cong (U(1) \times Spin(10))/\Z_4,\Z_4=\{ (1, \phi_{{}_{6,\sigma}}(1)), (-1, \phi_{{}_{6,\sigma}}(-1)), (i, \phi_{{}_{6,\sigma}}(-i)), (-i, \phi_{{}_{6,\sigma}}(i)) \}$.
\end{proposition}
\begin{proof}
	Let $Spin(10)$ as the group $(E_6)_{E_1}=\{\alpha \in E_6\,|\,\alpha E_1=E_1 \}$ (\cite[Theorem 3.10.4]{iy0}).
	We define a mapping $\varphi_{{}_{E_6,\sigma}}:U(1) \times Spin(10) \to (E_6)^\sigma $ by
	$$
	\varphi_{{}_{E_6,\sigma}}(\theta, \delta)=\phi_{{}_{6,\sigma}}(\theta)\delta,
	$$
	where $\phi_{{}_{6,\sigma}}:U(1) \to E_6$ is defined by
	\begin{align*}
	\phi_{{}_{6,\sigma}}(\theta)X=\begin{pmatrix}
	\theta^4 \xi_1 & \theta x_3 & \theta \ov{x_2} \\
	\theta \ov{x_3} & {\theta}^{-2}\xi_2 &  {\theta}^{-2}x_1 \\ 
	\theta x_2 &  {\theta}^{-2}\ov{x_1} &  {\theta}^{-2}\xi_3
	\end{pmatrix}, \,\, X \in \mathfrak{J}^C.
	\end{align*}
	This mapping induces the required isomorphism (see \cite[Theorem 3.10.7 ]{iy0} in detail).
\end{proof}

\begin{theorem}\label{theorem 3.3.4}
	The group $(E_6)^{\sigma_3}$ is isomorphic to the group $(U(1) \times Spin(2) \times Spin(8))/(\Z_4 \allowbreak \times \Z_2)${\rm :} $(E_6)^{\sigma_3} \cong (U(1) \times Spin(2) \times Spin(8))/(\Z_2 \times \Z_4), \Z_2=\{(1,1,1),(1,\sigma,\sigma) \}, \Z_4=\{(1,1,1),(i,D_{e_1},\phi_{{}_{6,\sigma}}(-i)D_{-e_1}),(-1,\allowbreak\sigma,1),(-i,D_{-e_1},\phi_{{}_{6,\sigma}}(i)D_{e_1}) \} \}$.
\end{theorem}
\begin{proof}
	Let $U(1)=\{\theta \in C\,|\,(\tau \theta)\theta=1 \}$ and $Spin(2)$, which is isomorphic to the group $U(1)$, as the group $\{D_a \in F_4 \,|\,a \in U(1) \}$ defined in $F_4$, moreover let $Spin(8)$ as the group $(E_6)_{E_1, F_1(1),F_1(e_1)}=\{ \alpha \in E_6 \,|\,\alpha E_1=E_1, \alpha F_1(1)=F_1(1), \alpha F_1(e_1)=F_1(e_1)\}$ (cf.\cite[Proposition 3.22]{iy2}, \cite[Subsection 3.2]{iy1}), respectively. We define a mapping $\varphi_{{}_{E_6,\sigma_3}}: U(1) \times Spin(2) \times Spin(8) \to (E_6)^{\sigma_3}$ by
	\begin{align*}
			\varphi_{{}_{E_6,\sigma_3}}(\theta, D_a, \beta)=\phi_{{}_{6,\sigma}}(\theta)D_a \beta.
	\end{align*}
	This mapping induces the required isomorphism (see \cite[Theorem 3.9]{iy1} in detail).
\end{proof}

Thus, since the group $(E_6)^{\sigma_3}$ is connected, together with the result of Theorem \ref{theorem 3.3.4}, we have an exceptional $\varmathbb{Z}_3$-symmetric space $E_6/((U(1) \times Spin(2) \times Spin(8))/(\Z_2 \times \Z_4))$.
\vspace{2mm}

Let $\nu=\exp(2\pi i/9) \in U(1)=\{ \theta \in C \,|\, (\tau \theta)\theta=1\} \subset C$. We consider the element $A_\nu \in SU(6) \subset M(6, C)$ as follows.
\begin{align*}
		A_\nu=\diag(\nu^5, \nu^{-1}, \nu^{-1}, \nu^{-1}, \nu^{-1},\nu^{-1}),
\end{align*} 
and using this $A_\nu$, set $\nu_3=\varphi_{{}_{E_6,\gamma}}(1,A_\nu)$. Then we have that $\nu_3 \in (E_6)^\gamma \subset E_6$ and $(\nu_3)^9=1$. Since ${A_\nu}^3= \nu^6 E \in z(SU(6))$ (the center of $SU(6)$) and $(\nu_3)^3=\varphi_{{}_{E_6,\gamma}}(1, {A_\nu}^3)=\omega 1$, where $\omega= -(1/2)+(\sqrt{3}/2)i \in C$, $\nu_3$ induces the automorphism $\tilde{\nu}_3$ of order $3$ on $E_6$: $\tilde{\nu}_3(\alpha)={\nu_3}^{-1}\alpha\nu_3, \alpha \in E_6$.
\vspace{1mm}

Now, we have the following theorem.

\begin{theorem}\label{theorem 3.3.5}
	The group $(E_6)^{\nu_3}$ is isomorphic to the group $(Sp(1) \times S(U(1) \times U(5)))/\Z_2${\rm :} $(E_6)^{\nu_3} \cong (Sp(1) \times S(U(1) \times U(5)))/\Z_2, \Z_2=\{(1,E), (-1,-E) \}$.
\end{theorem}
\begin{proof}
	Let $S(U(1) \times U(5)) \subset SU(6)$. We define a mapping $\varphi_{{}_{E_6, \nu_3}}:Sp(1) \times S(U(1) \times U(5)) \to  (E_6)^{\nu_3}$ by the restriction of the mapping $\varphi_{{}_{E_6,\gamma}}$. This mapping induces the required isomorphism (see \cite[Theorem 3.4]{iy1} in detail).
\end{proof}

Thus, since the group $(E_6)^{\nu_3}$ is connected, together with the result of Theorem \ref{theorem 3.3.5}, we have an exceptional $\varmathbb{Z}_3$-symmetric space $E_6/((U(1) \times S(U(1) \times U(5)))/ \Z_2)$.
\vspace{2mm}

Let $\phi_{{}_{6,\sigma}}:U(1) \to E_6$ be the embedding defined in the proof of Proposition \ref{proposition 3.3.3}, and again let $\nu=\exp(2\pi i/9) \in U(1) \subset C$. Set $\mu_3=\phi_{{}_{6,\sigma}}(\nu)$. Then, needless to say, $\mu_3 \in E_6$ and $\nu^9=1$. 
Hence, since $\mu^3=\omega 1 \in z(E_6)$ (the center of $E_6$), $\mu_3$ induces the automorphism $\tilde{\mu_3}$ of order $3$ on $E_6$: $\tilde{\mu_3}(\alpha)={\mu_3}^{-1}\alpha\mu_3, \alpha \in E_6$.
\vspace{1mm}

Now, we have the following theorem.

\begin{theorem}\label{theorem 3.3.6}
	The group $(E_6)^{\mu_3}$ coincides with the group $(E_6)^\sigma$, that is, this group is isomorphic to the group $(U(1) \times Spin(10))/\Z_4${\rm :} $(E_6)^{\mu_3} \cong (U(1) \times Spin(10))/\Z_4, \Z_4=\{ (1, 1), (-1, \sigma), (i, \phi_{{}_{6,\sigma}}(-i)), (-i, \phi_{{}_{6,\sigma}}(i)) \}$
\end{theorem}
\begin{proof}
	We have to prove that $(E_6)^{\mu_3}=(E_6)^\sigma$.
	However the details of proof is omitted (see \cite[Theorem 3.11]{iy1} in detail).
\end{proof}
\vspace{2mm}

Let $w_3 \in G_2 \subset F_4$. Then, as in the cases above, using the inclusion $F_4 \subset E_6$, $w_3$ are naturally extended to
transformation of $\mathfrak{J}^C$.
Needless to say, $w_3 \in E_6$ by inclusion $F_4 \subset E_6$ and $(w_3)^3=1$. Hence $w_3$ induces the automorphism $\tilde{w}_3$ of order $3$ on $E_6$: $\tilde{w}_3(\alpha)={w_3}^{-1}\alpha w_3, \alpha \in E_6$.
Note that using $\bm{\omega}=-(1/2)+(\sqrt{3}/2)e_1 \in \C$, the action to $\mathfrak{J}^C=\mathfrak{J}(3, \C)^C \oplus M(3,\C)^C$ of $w_3$ is as follows.
\begin{align*}
		w_3(X_{\bm{C}}+M)=X_{\bm{C}}+\bm{\omega}M,\,\,X_{\bm{C}}+M \in \mathfrak{J}(3,\C)^C \oplus M(3, \C)^C=\mathfrak{J}^C.
\end{align*}

Now, we have the following theorem.

\begin{theorem}\label{theorem 3.3.7}
	The group $(E_6)^{w_3}$ is isomorphic to the group $(SU(3) \times SU(3) \times SU(3))/\Z_3${\rm:} $(E_6)^{w_3} \cong (SU(3) \times SU(3) \times SU(3))/\Z_3, \Z_3=\{(E,E,E),(\bm{\omega}E,\bm{\omega}E,\bm{\omega}E),(\bm{\omega}^{-1}E,\bm{\omega}^{-1}E,\allowbreak \bm{\omega}^{-1}E)  \}$.
\end{theorem}
\begin{proof}
	We define a mapping $\varphi_{{}_{E_6,w_3}}:SU(3) \times SU(3) \times SU(3) \to (E_6)^{w_3}$ by
	\begin{align*}
			\varphi_{{}_{E_6,w_3}}(L,A,B)(X_{C}+M)&=h(A,B)X_{C}h(A,B)^*+LM\tau h(A,B)^*, 
			\\
			&\hspace*{20mm} X_{C}+M \in \mathfrak{J}(3, \C)^C \oplus 
			M(3,\C)^C=\mathfrak{J}^C,
	\end{align*}
	where $h:M(3,\C) \times M(3,\C) \to M(3,\C)^C$ is defined by 
	\begin{align*}
			h(A,B)=\dfrac{A+B}{2}+i\dfrac{(B-A)e_1}{2}.
	\end{align*}
	This mapping induces the required isomorphism (see \cite[Theorem 13]{iy0} in detail). Note that there is a mistake for the numbering of theorems in \cite{iy0}, so Theorem 13 above is corresponding to the last theorem.
\end{proof}

Thus, since the group $(E_6)^{w_3}$ is connected, together with the result of Theorem \ref{theorem 3.3.7}, we have an exceptional $\varmathbb{Z}_3$-symmetric space $E_6/((SU(3) \times SU(3))/ \Z_3)$.

As in Subsections 3.1, 3.2, the following lemma are useful to determine the structure of groups $G^{\sigma_3} \cap G^{\tau_3}$ in $E_6$.

\begin{lemma}\label{lemma 3.3.8}
	{\rm (1)} The mapping $\varphi_{{}_{E_6,\gamma_3}}:U(1) \times SU(6) \to (E_6)^{\gamma_3}$ of \,Theorem {\rm \ref{theorem 3.3.2}} satisfies the relational formulas 
	\begin{align*}
	\gamma_3&=\varphi_{{}_{E_6,\gamma_3}}(\omega,E), 
	\\
	\sigma_3&=\varphi_{{}_{E_6,\gamma_3}}(1,\diag(1,1,\tau\omega,\omega,\omega,\tau\omega)), 
	\\
	\nu_3&=\varphi_{{}_{E_6,\gamma_3}}(1,\diag(\nu^5,\nu^{-1},\nu^{-1},\nu^{-1},\nu^{-1},\nu^{-1})),
	\\
	\mu_3&=\varphi_{{}_{E_6,\gamma_3}}(1,\diag(\nu^{-2},\nu^2,\nu^{-1},\nu,\nu^{-1},\nu)),
	\\
	w_3&=\varphi_{{}_{E_6,\gamma_3}}(1,\diag(\tau\omega,\omega,\tau\omega,\omega,\tau\omega,\omega)),
	\end{align*}
  where $\bm{\omega}=-(1/2)+(\sqrt{3}/2)i \in U(1), \nu=\exp(2\pi i/9)$.
	\vspace{1mm}
	
	{\rm (2)} The mapping $\varphi_{{}_{E_6,w_3}}:SU(3)\times SU(3) \times SU(3) \to (E_6)^{w_3}$ of \,Theorem {\rm \ref{theorem 3.3.7}} satisfies the relational formulas
	\begin{align*}
	\gamma_3&=\varphi_{{}_{E_6,w_3}}(\diag(1,\bm{\omega},\ov{\bm{\omega}}),E,E),
	\\
	\sigma_3&=\varphi_{{}_{E_6,w_3}}(E,\diag(1,\ov{\bm{\omega}},\bm{\omega}),\diag(1,\ov{\bm{\omega}},\bm{\omega})),
	\\
	\mu_3&=\varphi_{{}_{E_6,w_3}}(E,\diag({\bm{\varepsilon}}^{-2},\bm{\varepsilon},\bm{\varepsilon}),\diag({\bm{\varepsilon}}^2,{\bm{\varepsilon}}^{-1},{\bm{\varepsilon}}^{-1})),
	\\
	w_3&=\varphi_{{}_{E_6,w_3}}(\bm{\omega}E,E,E),
	\end{align*}
	 where $\bm{\omega}=-(1/2)+(\sqrt{3}/2)e_1 \in U(1), \bm{\varepsilon}=\exp(2\pi e_1/9)$. 	
\end{lemma}
\begin{proof}
	(1), (2) By doing straightforward computation we obtain the results above. 
\end{proof}

\section{Globally exceptional $\varmathbb{Z}_3 \times \varmathbb{Z}_3$-symmetric spaces}

In this section, we construct a finite abelian group $\varGamma=\varmathbb{Z}_3 \times \varmathbb{Z}_3$ by using the inner automorphisms $\tilde{\sigma}_3, \tilde{\tau}_3$ of order $3$ on $G=G_2, F_4,E_6$ as the Case 1 below and determine the structure of the group $G^{\sigma_3} \cap G^{\tau_3}$.

\subsection{Case 1: $\{1, \tilde{\gamma}_3,  \tilde{\gamma}_3{}^{-1}\} \times \{1, \tilde{w}_3,  \tilde{w}_3{}^{-1}\}$-symmetric space}

Let the $\R$-linear transformations $\gamma_3, w_3$ of $\mathfrak{C}$  defined in Subsection \ref{subsection 3.1}. 

\noindent From Lemma \ref{lemma 3.1.4} (1), since we can easily confirm that $\gamma_3$ and $w_3$ are commutative, $\tilde{\gamma}_3$ and $\tilde{w}_3$ are commutative in $\Aut(G_2)$: $\tilde{\gamma}_3\tilde{w}_3=\tilde{w}_3\tilde{\gamma}_3$.
\vspace{1mm}

Now, we will determine the structure of the group $(G_2)^{\gamma_3} \cap (G_2)^{w_3}$.

\begin{theorem}\label{theorem 4.1.1}
	The group  $(G_2)^{\gamma_3} \cap (G_2)^{w_3}$ is isomorphic to the group $(U(1) \times U(1))/\Z_2${\rm :} $(G_2)^{\gamma_3} \cap (G_2)^{w_3} \cong (U(1) \times U(1))/\Z_2, \Z_2=\{(1,1), (-1,-1) \}$.
\end{theorem}
\begin{proof}
	Let $U(1) \subset Sp(1)$. 
	We define a mapping $\varphi_{{}_{G_2,\gamma_3, w_3}}: U(1) \times U(1) \to (G_2)^{\gamma_3} \cap (G_2)^{w_3}$ by 
	\begin{align*}
				\varphi_{{}_{G_2,\gamma_3, w_3}}(s,t)(m+ne_4)=tm\ov{t}+(sn\ov{t})e_4,\,\,m+ne_4 \in \H \oplus \H e_4=\mathfrak{C}.
	\end{align*}
	Needless to say, this mapping is the restriction of the mapping $\varphi_{{}_{G_2,\gamma_3}}$ (Theorem \ref{theorem 3.1.2}).
	
	First, we will prove that $\varphi_{{}_{G_2,\gamma_3, w_3}}$ is well-defined. Since this mapping is also the restriction of the mapping $\varphi_{{}_{G_2,\gamma_3}}$, it is trivial that $\varphi_{{}_{G_2,\gamma_3, w_3}}(s,t) \in (G_2)^{\gamma_3}$, and from $w_3=\varphi_{{}_{G_2,\gamma_3}}(1,\ov{\bm{\omega}})$ (Lemma \ref{lemma 3.1.4} (1)), it is almost clear that  $\varphi_{{}_{G_2,\gamma_3, w_3}}(s,t) \in (G_2)^{w_3}$. Hence $\varphi_{{}_{G_2,\gamma_3, w_3}}$ is well-defined. Subsequently, since $\varphi_{{}_{G_2,\gamma_3, w_3}}$ is the restriction of $\varphi_{{}_{G_2,\gamma_3}}$, we easily see that $\varphi_{{}_{G_2,\gamma_3, w_3}}$ is a homomorphism.
	
	Next, we will prove that $\varphi_{{}_{G_2,\gamma_3, w_3}}$ is surjective. Let $\alpha \in (G_2)^{\gamma_3} \cap (G_2)^{w_3} \subset (G_2)^{\gamma_3}$. There exist $s \in U(1)$ and $q \in Sp(1)$ such that $\alpha=\varphi_{{}_{G_2,\gamma_3}}(s,q)$ (Theorem \ref{theorem 3.1.2}). Moreover, since $\alpha=\varphi_{{}_{G_2,\gamma_3}}(s,q)$ commutes with $w_3$, again using $w_3=\varphi_{{}_{G_2,\gamma_3}}(1,\ov{\bm{\omega}})$, we have that 
	\begin{align*}
			\left\{ \begin{array}{l}
			s=s \\
			\bm{\omega}q\ov{\bm{\omega}}=q 
			\end{array} \right. 
			\quad \text{or}\quad
			\left\{ \begin{array}{l}
			s=-s \\
			\bm{\omega}q\ov{\bm{\omega}}=-q.
			\end{array} \right.
	\end{align*}
	The latter case is impossible because $s \not=0$. As for the former case, from the relational formula $\bm{\omega}q\ov{\bm{\omega}}=q$ we easily see that $q \in U(1)$, and needless to say, $s \in U(1)$. Hence there exist $s,t \in U(1)$ such that $\alpha=\varphi_{{}_{G_2,\gamma_3}}(s,t)$. Namely, there exist $s,t \in U(1)$ such that $\alpha=\varphi_{{}_{G_2,\gamma_3,w_3}}(s,t)$. The proof of surjective is completed.
	
	Finally, we determine $\Ker \,\varphi_{{}_{G_2,\gamma_3, w_3}}$. However, since $\varphi_{{}_{G_2,\gamma_3, w_3}}$ is the restriction of $\varphi_{{}_{G_2,\gamma_3}}$, it is easily obtain that $\Ker \,\varphi_{{}_{G_2,\gamma_3, w_3}}=\{(1,1),(-1,-1) \} \cong \Z_2$.
	
	Therefore we have the required isomorphism
	\begin{align*}
			(G_2)^{\gamma_3} \cap (G_2)^{w_3} \cong (U(1) \times U(1))/\Z_2.
	\end{align*}
\end{proof}

Thus, since the group $(G_2)^{\gamma_3} \cap (G_2)^{w_3}$ is connected from Theorem \ref{theorem 4.1.1}, we have an exceptional $\varmathbb{Z}_3 \times \varmathbb{Z}_3$-symmetric space
\begin{align*}
							G_2/((U(1) \times U(1))/\Z_2).
\end{align*}

\subsection{Case 2: $\{1, \tilde{\gamma}_3,  \tilde{\gamma}_3{}^{-1}\} \times \{1, \tilde{\sigma}_3,  \tilde{\sigma}_3{}^{-1}\}$-symmetric space}

Let the $\R$-linear transformations $\gamma_3, \sigma_3$ of $\mathfrak{J}$  defined in Subsection \ref{subsection 3.2}. 

\noindent From Lemma \ref{lemma 3.2.6} (1), since we can easily confirm that $\gamma_3$ and $\sigma_3$ are commutative, $\tilde{\gamma}_3$ and $\tilde{\sigma}_3$ are commutative in $\Aut(F_4)$: $\tilde{\gamma}_3\tilde{\sigma}_3=\tilde{\sigma}_3\tilde{\gamma}_3$.
\vspace{1mm}

Before determining the structure of the group $(F_4)^{\gamma_3} \cap (F_4)^{\sigma_3}$, we prove proposition needed in the proof of theorem below.
\vspace{1mm}

We define subgroups $G_{1,2}$ and $G'_{1,2}$ of the group $Sp(3)$ by
\begin{align*}
	G_{1,2}&=\left\{ A=\begin{pmatrix}
	                         h & 0 & 0 \\
	                         0 & a & c \\
	                         0 & d & b
	                        \end{pmatrix} \in Sp(3)\,\left|\,h \in Sp(1), \begin{pmatrix}	                      
	                         a & c \\
	                         d & b
	                        \end{pmatrix} \in U(2) \subset Sp(2) 
	                       \right. \right\}, 
	\\
	G'_{1,2}&=\left\{ A'=\begin{pmatrix}
	h' & 0 & 0 \\
	0 & a' & c'e_2 \\
	0 & \ov{e_2}d' & b'
	\end{pmatrix} \in Sp(3)\,\left|\,h' \in Sp(1), 
	\begin{array}{l}
	(c'e_2)(\ov{c'e_2})+a'\ov{a'}=1\\
	b'\ov{b'}+(\ov{e_2}d')(\ov{\ov{e_2}d'})=1\\
	(c'e_2)\ov{b'}+a'(\ov{\ov{e_2}d'})=0\\
	a',b',c',d' \in \C
	\end{array}
	\right. \right\}, 
\end{align*}
where $e_2$ is one of basis in $\mathfrak{C}$.

 It goes without saying that $\begin{pmatrix}
                       a & c \\
                       d & b
                     \end{pmatrix} \in U(2)$ is equivalent to the conditions
\begin{align*}
		c\ov{c}+a\ov{a}=1, \,\,b\ov{b}+d\ov{d}=1,\,\,c\ov{b}+a\ov{d}=0,
\end{align*}
moreover, that $(c'e_2)(\ov{c'e_2})+a'\ov{a'}=1$ above is same as $c'\ov{c}+a'\ov{a'}=1$, so is others.
\vspace{1mm}

\begin{proposition}\label{proposition 4.2.1}
	The group $G'_{1,2}$ is isomorphic to the group $Sp(1) \times U(2)${\rm :} $G'_{1,2} \cong Sp(1) \times U(2)$.
\end{proposition}
\begin{proof}
	First, we will prove that the group $G'_{1,2}$ is isomorphic to the group $G_{1,2}$. 
	We define a mapping $g_{{}_{421}}: G_{1,2} \to G'_{1,2}$ by
	\begin{align*}
			g_{{}_{421}}(\begin{pmatrix}
			h & 0 & 0 \\
			0 & a & c \\
			0 & d & b
			\end{pmatrix})
			&=\begin{pmatrix}
			1 & 0 & 0 \\
			0 & 1 & 0 \\
			0 & 0 & \ov{e_2}
			\end{pmatrix}\begin{pmatrix}
			h & 0 & 0 \\
			0 & a & c \\
			0 & d & b
			\end{pmatrix}\begin{pmatrix}
			1 & 0 & 0 \\
			0 & 1 & 0 \\
			0 & 0 & e_2
			\end{pmatrix}\left(=\begin{pmatrix}
			h & 0 & 0 \\
			0 & a & ce_2 \\
			0 & \ov{e_2}d & b
			\end{pmatrix} \right).
	\end{align*}
	First, it is clear that $g_{{}_{421}}$ is well-defined and a homomorphism. Moreover, it is easy to verify that $g_{{}_{421}}$ is bijective. Thus we have the isomorphism $G'_{1,2} \cong G_{1,2}$. 
	
	Here, by defining a mapping $f_{{}_{421}}:Sp(1) \times U(2) \to G_{1,2}$ as follows:
	\begin{align*}
		f_{{}_{421}}(p,U)=\scalebox{0.8}{$
			\left( \begin{array}{cccccccc@{\!}}
			&\multicolumn{2}{c}{\raisebox{-15pt}[0pt][0pt]{\Large$p$}}&&&&
			\\
			&&&&\multicolumn{2}{c}{\raisebox{-5pt}[0pt]{\Large $0$}}&
			\\
			&&&&&&&
			\\
			&&&&\multicolumn{2}{c}{\raisebox{-18pt}[0pt][0pt]{\huge $U$}}&
			\\
			&\multicolumn{2}{c}{\raisebox{-5pt}[0pt]{\Large $0$}}&&&&
			\\[-2mm]
			&&&&&&&
			\end{array}\right)$},
		\end{align*}
	we have the isomorphism $G_{1,2} \cong Sp(1) \times U(2)$.
	
	Therefore, together with the result of $G'_{1,2} \cong G_{1,2}$, we have the required isomorphism 
	\begin{align*}
			G'_{1,2} \cong Sp(1) \times U(2).
	\end{align*}
\end{proof}

Now, we will determine the structure of the group $(F_4)^{\gamma_3} \cap (F_4)^{\sigma_3}$.

\begin{theorem} \label{theorem 4.2.2}
	The group $(F_4)^{\gamma_3} \cap (F_4)^{\sigma_3}$ is isomorphic to the group $(U(1) \times Sp(1) \times U(2))/\Z_2$ {\rm: } $(F_4)^{\gamma_3} \cap (F_4)^{\sigma_3} \cong (U(1) \times Sp(1) \times U(2))/\Z_2, \Z_2=\{(1,1,E),(-1,-1,-E) \}$.
\end{theorem}
\begin{proof}
	First, we denote the composition of $g_{{}_{421}}$ and $f_{{}_{421}}$ by $h$: $h=g_{{}_{421}}f_{{}_{421}}$ (in the proof of Proposition \ref{proposition 4.2.1}). Then we define a mapping $\varphi_{{}_{F_4,\gamma_3,\sigma_3}}:U(1) \times Sp(1) \times U(2) \to (F_4)^{\gamma_3} \cap (F_4)^{\sigma_3}$ by
		\begin{align*}
		\varphi_{{}_{F_4,\gamma_3,\sigma_3}}(s,p,U)(M+\a)=h(p,U)Mh(p,U)^*+s\a h(p,U)^*,\,
		M+\a \in \mathfrak{J}(3,\H) \oplus \H^3=\mathfrak{J}.
		\end{align*}
		Needless to say, this mapping is the restriction of the mapping $\varphi_{{}_{F_4,\gamma_3}}$, that is, $\varphi_{{}_{F_4,\gamma_3,\sigma_3}}(s,p,U) \allowbreak =\varphi_{{}_{F_4,\gamma_3}}(s,h(p,U))$ (Theorem \ref{theorem 3.2.2}).
		
		First, we will prove that $\varphi_{{}_{F_4,\gamma_3,\sigma_3}}$ is well-defined. It is clear that $\varphi_{{}_{F_4,\gamma_3,\sigma_3}}(s,p,U) \in (F_4)^{\gamma_3}$, and using $\sigma_3=\varphi_{{}_{F_4,\gamma_3}}(1,\diag(1,\ov{\bm{\omega}}, \bm{\omega}))$ (Lemma \ref{lemma 3.2.6} (1)), it follows that
		\begin{align*}
		{\sigma_3}^{-1}\varphi_{{}_{F_4,\gamma_3,\sigma_3}}(s,p,U)\sigma_3
		&=\varphi_{{}_{F_4,\gamma_3}}(1,\diag(1,\ov{\bm{\omega}},\bm{\omega}))^{-1}\varphi_{{}_{F_4,\gamma_3,\sigma_3}}(s,p,U)
		\varphi_{{}_{F_4,\gamma_3}}(1,\diag(1,\ov{\bm{\omega}},\bm{\omega}))
		\\
		&=\varphi_{{}_{F_4,\gamma_3}}(1,\diag(1,\bm{\omega},\ov{\bm{\omega}}))
		\varphi_{{}_{F_4,\gamma_3}}(s,h(p,U))\varphi_{{}_{F_4,\gamma_3}}(1,\diag(1,\ov{\bm{\omega}},\bm{\omega}))
		\\
		&=\varphi_{{}_{F_4,\gamma_3}}(s,\diag(1,\bm{\omega},\ov{\bm{\omega}})h(p,U)\diag(1,\ov{\bm{\omega}},\bm{\omega})), h(p,U)\!=
		\begin{pmatrix}
		p & 0 & 0 \\
		0 & a & c \\
		0 & d & b
		\end{pmatrix}
		\\
		&=\varphi_{{}_{F_4,\gamma_3}}(s,
		\begin{pmatrix}
		p & 0 & 0 \\
		0 & \bm{\omega}a\ov{\bm{\omega}} & \bm{\omega}(ce_2)\bm{\omega} \\
		0 & \ov{\bm{\omega}}(\ov{e_2}d)\ov{\bm{\omega}} & \ov{\bm{\omega}}b\bm{\omega}
		\end{pmatrix})
		\\
		&=\varphi_{{}_{F_4,\gamma_3}}(s,
		\begin{pmatrix}
		p & 0 & 0 \\
		0 & a & c e_2\\
		0 & \ov{e_2}d & b
		\end{pmatrix})
		\\
		&=\varphi_{{}_{F_4,\gamma_3}}(s,h(p,U))
		\\
		&=\varphi_{{}_{F_4,\gamma_3,\sigma_3}}(s,p,U).
		\end{align*}
	Hence we have that $\varphi_{{}_{F_4,\gamma_3,\sigma_3}}(s,p,U) \in (F_4)^{\sigma_3}$. Thus $\varphi_{{}_{F_4,\gamma_3,\sigma_3}}$ is well-defined.
	Subsequently, since $\varphi_{{}_{F_4,\gamma_3,\sigma_3}}$ is the restriction of the mapping $\varphi_{{}_{F_4,\gamma_3}}$, we easily see  that $\varphi_{{}_{F_4,\gamma_3,\sigma_3}}$ is a homomorphism. 

	Next, we will prove that $\varphi_{{}_{F_4,\gamma_3,\sigma_3}}$ is surjective. Let $\alpha \in (F_4)^{\gamma_3} \cap (F_4)^{\sigma_3} \subset (F_4)^{\gamma_3}$. There exist  $s \in U(1)$ and $A \in Sp(3)$ such that $\alpha=\varphi_{{}_{F_4,\gamma_3}}(s,A)$ (Theorem \ref{theorem 3.2.2}). Moreover, from the condition $\alpha\in	(F_4)^{\sigma_3}$, that is, ${\sigma_3}^{-1}\varphi_{{}_{F_4,\gamma_3}}(s,A)\sigma_3=\varphi_{{}_{F_4,\gamma_3}}(s,A)$, and using ${\sigma_3}^{-1}\varphi_{{}_{F_4,\gamma_3}}(s,A)\sigma_3\!=\!\varphi_{{}_{F_4,\gamma_3}}(s,\diag(1,\bm{\omega},\ov{\bm{\omega}})A\,\diag(1,\ov{\bm{\omega}},\bm{\omega}))$ (Lemma \ref{lemma 3.2.6} (1)), we have that
	\begin{align*}
	\left\{ 
	\begin{array}{l}
	s=s \\
	\diag(1,\bm{\omega},\ov{\bm{\omega}})A\,\diag(1,\ov{\bm{\omega}},\bm{\omega})=A 
	\end{array} \right.
	\quad {\text{or}}\quad
	\left\{ 
	\begin{array}{l}
	s=-s \\
	\diag(1,\bm{\omega},\ov{\bm{\omega}})A\,\diag(1,\ov{\bm{\omega}},\bm{\omega})=-A. 
	\end{array} \right.
	\end{align*}
	The latter case is impossible because of $s\not=0$. As for the former case, from the second condition, by doing straightforward computation $A$ takes the following form $\begin{pmatrix}
	p & 0 & 0 \\
	0 & a & c e_2\\
	0 & \ov{e_2}d & b
	\end{pmatrix} \allowbreak \in Sp(3)$, that is, $A \in G'_{1,2}$. Hence there exist $s \in U(1)$ and $h(p,U) \in Sp(3)$ such that $\alpha=\varphi_{{}_{F_4,\gamma_3}}(s,h(p,U))$. 
	Moreover, from Lemma \ref{lemma 3.2.6} (1) there exist $p \in Sp(1)$ and $U \in U(2)$ such that $A=h(p,U)$. Needless to say, $s \in U(1)$.
  Thus, there exist $s \in U(1),p \in Sp(1)$ and $U \in U(2)$ such that $\alpha=\varphi_{{}_{F_4,\gamma_3, \sigma_3}}(s,p,U)$. The proof of surjective is completed.
	
	Finally, we will determine $\Ker\,\varphi_{{}_{F_4,\gamma_3,\sigma_3}}$. However, from $\Ker\,\varphi_{{}_{F_4,\gamma_3}}=\{(1,E), (-1,-E) \}$ we easily obtain that $\Ker\,\varphi_{{}_{F_4,\gamma_3,\sigma_3}}=\{(1,1,E), (-1,-1,-E) \} \cong \Z_2$. 

 Therefore we have the required isomorphism
 \begin{align*}
 (F_4)^{\gamma_3} \cap (F_4)^{\sigma_3} \cong (U(1) \times Sp(1) \times U(2))/\Z_2.
  \end{align*}
\end{proof}
\vspace{1mm}

Thus, since the group $(F_4)^{\gamma_3} \cap (F_4)^{\sigma_3}$ is connected from Theorem \ref{theorem 4.2.2}, we have an exceptional $\varmathbb{Z}_3 \times \varmathbb{Z}_3$-symmetric space
\begin{align*}
     F_4/((U(1) \times Sp(1) \times U(2))/\Z_2).
\end{align*}

\subsection{Case 3: $\{1, \tilde{\gamma}_3,  \tilde{\gamma}_3{}^{-1}\} \times \{1, \tilde{w}_3,  \tilde{w}_3{}^{-1}\}$-symmetric space}

Let the $\R$-linear transformations $\gamma_3, w_3$ of $\mathfrak{J}$  defined in Subsection \ref{subsection 3.2}. 

\noindent From Lemma \ref{lemma 3.2.6} (2), since we can easily confirm that $\gamma_3$ and $w_3$ are commutative, $\tilde{\gamma}_3$ and $\tilde{w}_3$ are commutative in $\Aut(F_4)$: $\tilde{\gamma}_3\tilde{w}_3=\tilde{w}_3\tilde{\gamma}_3$.
\vspace{1mm}

Before determining the structure of the group $(F_4)^{\gamma_3} \cap (F_4)^{w_3}$, we prove lemma needed in the proof of theorem below.

\begin{lemma}\label{lemma 4.3.1}
	The group $S(U(1)\times U(1)\times U(1))$ is isomorphic to the group $U(1)\times U(1)${\rm :} $S(U(1)\times U(1)\times U(1)) \cong U(1)\times U(1)$.
\end{lemma}
\begin{proof}
	We define a mapping $f_{{}_{431}}: U(1)\times U(1) \to S(U(1)\times U(1)\times U(1))$ by 
	\begin{align*}
	f_{{}_{431}}(a,b)=\left( 
	\begin{array}{ccc}
	s & & {\raisebox{-7pt}[0pt]{\large $0$}}
	\\[2mm]
	& t & 
	\\[2mm]
	{\raisebox{1pt}[0pt]{\large $0$}}&& (st)^{-1}
	\end{array}\right) \in SU(3).
	\end{align*}
	Then this mapping induces the required isomorphism.
\end{proof}

Now, we will determine the structure of the group $(F_4)^{\gamma_3} \cap (F_4)^{w_3}$.

\begin{theorem}\label{theorem 4.3.2}
	The group $(F_4)^{\gamma_3} \cap (F_4)^{w_3}$ is isomorphic to the group $(U(1) \times U(1) \times SU(3))/\Z_3${\rm :} $(F_4)^{\gamma_3} \cap (F_4)^{w_3} \cong (U(1) \times U(1)\times SU(3))/\Z_3, \Z_3=\{(1,1,E), (\bm{\omega},\bm{\omega}, \bm{\omega}E),(\bm{\omega}^{-1},\allowbreak\bm{\omega}^{-1}, \bm{\omega}^{-1}E)\}$.
\end{theorem}
\begin{proof}
	Let $S(U(1) \times U(1) \times U(1)) \subset SU(3)$.
	We define a mapping $\varphi_{{}_{F_4,\gamma_3, w_3}}: S(U(1) \times U(1) \times U(1)) \times SU(3) \to (F_4)^{\gamma_3} \cap (F_4)^{w_3}$ by
	\begin{align*}
			\varphi_{{}_{F_4,\gamma_3, w_3}}(L,A)(X_{\bm{C}}+M)=AX_{\bm{C}}A^*+LMA^*,\,\,X_{\bm{C}}+M \in \mathfrak{J}(3,\C)\oplus M(3,\C)=\mathfrak{J}.
	\end{align*}
	Needless to say, this mapping is the restriction of the mapping $\varphi_{{}_{F_4,w_3}}$, that is, $\varphi_{{}_{F_4,\gamma_3, w_3}}(L,A)=\varphi_{{}_{F_4,w_3}}(L,A)$ (Theorem \ref{theorem 3.2.5}).
	
	As usual, we will prove that $\varphi_{{}_{F_4,\gamma_3, w_3}}$ is well-defined. It is clear that $\varphi_{{}_{F_4,\gamma_3, w_3}}(L,A) \in (F_4)^{w_3}$, and using $\gamma_3=\varphi_{{}_{F_4,w_3}}(\diag(1,\ov{\bm{\omega}},\bm{\omega}), E)$ (Lemma \ref{lemma 3.2.6} (2)), it follows that 
	\begin{align*}
	     {\gamma_3}^{-1}\varphi_{{}_{F_4,\gamma_3, w_3}}(L,A)\gamma_3
	     &=\varphi_{{}_{F_4,w_3}}(\diag(1,\ov{\bm{\omega}},\bm{\omega}), E)^{-1}\varphi_{{}_{F_4,\gamma_3, w_3}}(L,A)\varphi_{{}_{F_4,w_3}}(\diag(1,\ov{\bm{\omega}},\bm{\omega}), E)
	     \\
	     &=\varphi_{{}_{F_4,w_3}}(\diag(1,\bm{\omega},\ov{\bm{\omega}}), E)\varphi_{{}_{F_4,w_3}}(L,A)\varphi_{{}_{F_4,w_3}}(\diag(1,\ov{\bm{\omega}},\bm{\omega}), E)
	     \\
	     &=\varphi_{{}_{F_4,w_3}}(\diag(1,\bm{\omega},\ov{\bm{\omega}})L\diag(1,\ov{\bm{\omega}},\bm{\omega}),A),L=\diag(a,b,c), abc=1
	     \\
	     &=\varphi_{{}_{F_4,w_3}}(L,A)
	     \\
	     &=\varphi_{{}_{F_4,\gamma_3,w_3}}(L,A).
	\end{align*}
	Hence we have that $\varphi_{{}_{F_4,\gamma_3, w_3}}(L,A) \in (F_4)^{\gamma_3}$.  Thus $\varphi_{{}_{F_4,\gamma_3, w_3}}$ is well-defined. Subsequently, since $\varphi_{{}_{F_4,\gamma_3,w_3}}$ is the restriction of the mapping $\varphi_{{}_{F_4,w_3}}$, we easily see that $\varphi_{{}_{F_4,\gamma_3,w_3}}$ is a homomorphism. 
	
	Next, we will prove that $\varphi_{{}_{F_4,\gamma_3,w_3}}$ is surjective. Let $\alpha \in (F_4)^{\gamma_3} \cap (F_4)^{w_3} \subset (F_4)^{w_3}$. There exist $P, A \in SU(3)$ such that $\alpha=\varphi_{{}_{F_4,w_3}}(P,A)$ (Theorem \ref{theorem 3.2.5}). Moreover, from the condition $\alpha \in (F_4)^{\gamma_3}$, that is, ${\gamma_3}^{-1}\varphi_{{}_{F_4,w_3}}(P,A)\gamma_3=\varphi_{{}_{F_4,w_3}}(P,A)$, and using ${\gamma_3}^{-1}\varphi_{{}_{F_4,w_3}}(P,A)\gamma_3=\varphi_{{}_{F_4,w_3}}(\diag(1,\bm{\omega},\ov{\bm{\omega}})P\,\diag(1,\ov{\bm{\omega}},\bm{\omega}),A)$ (Lemma \ref{lemma 3.2.6} (2)), we have that
	\begin{align*}
	&\,\,\,{\rm(i)}\,\left\{
	\begin{array}{l}
	\diag(1,\bm{\omega},\ov{\bm{\omega}})P\diag(1,\ov{\bm{\omega}},\bm{\omega})=P \\
	A=A,
	\end{array} \right.
	\qquad
	 {\rm(ii)}\,\left\{
 	\begin{array}{l}
 	\diag(1,\bm{\omega},\ov{\bm{\omega}})P\diag(1,\ov{\bm{\omega}}, \bm{\omega})=\bm{\omega}P \\
	A=\bm{\omega}A,
	\end{array} \right.
	\\[2mm]
	&{\rm(iii)}\,\left\{
	\begin{array}{l}
	\diag(1,\bm{\omega},\ov{\bm{\omega}})P\diag(1,\ov{\bm{\omega}},\bm{\omega})=\bm{\omega}^{-1}P \\
	A=\bm{\omega}^{-1}A.
	\end{array} \right.
	\end{align*}
	The Cases (ii) and (iii) are impossible because of $A\not=0$. As for the Case (i), from the first condition, by doing straightforward computation $P$ takes the form  $\diag(a,b,c)
  \in SU(3)$, that is, $P \in S(U(1)\times U(1)\times U(1))$. Needless to say, $A \in SU(3)$. Hence there exist $L \in S(U(1)\times U(1) \times U(1))$ and $A \in SU(3)$ such that $\alpha=\varphi_{{}_{F_4,w_3}}(L,A)$. Namely, there exist $L \in S(U(1)\times U(1) \times U(1))$ and $A \in SU(3)$ such that $\alpha=\varphi_{{}_{F_4,\gamma_3,w_3}}(L,A)$. With above, the proof of surjective is completed.
	
	Finally, we will determine $\Ker\,\varphi_{{}_{F_4,\gamma_3,w_3}}$. However, from $\Ker\,\varphi_{{}_{F_4,w_3}}=\{(E,E),(\bm{\omega}E,\bm{\omega}E), \allowbreak (\bm{\omega}^{-1}E,\bm{\omega}^{-1}E)\}$, we easily obtain that $\Ker\,\varphi_{{}_{F_4,\gamma_3,w_3}}=\{(E,E),(\bm{\omega}E,\bm{\omega}E), (\bm{\omega}^{-1}E,\bm{\omega}^{-1}E)\} \cong \Z_3$. Thus we have the isomorphism $(F_4)^{\gamma_3} \cap (F_4)^{w_3} \cong (S(U(1) \times U(1) \times U(1))\times SU(3))/\Z_3$.
	
	Therefore, by Lemma \ref{lemma 4.3.1} we have the required isomorphism 
	\begin{align*}
	(F_4)^{\gamma_3} \cap (F_4)^{w_3} \cong (U(1) \times U(1)\times SU(3))/\Z_3,
	\end{align*}
	where $\Z_3=\{(1,1,E), (\bm{\omega},\bm{\omega}, \bm{\omega}E),(\bm{\omega}^{-1},\bm{\omega}^{-1}, \bm{\omega}^{-1}E)\}$.
\end{proof}
\vspace{1mm}

Thus, since the group $(F_4)^{\gamma_3} \cap (F_4)^{w_3}$ is connected  from Theorem \ref{theorem 4.3.2}, we have an exceptional $\varmathbb{Z}_3 \times \varmathbb{Z}_3$-symmetric space
\begin{align*}
F_4/((U(1) \times U(1) \times SU(3))/\Z_3).
\end{align*}

\subsection{Case 4: $\{1, \tilde{\sigma}_3,  \tilde{\sigma}_3{}^{-1}\} \times \{1, \tilde{w}_3,  \tilde{w}_3{}^{-1}\}$-symmetric space}\label{case 4}

Let the $\R$-linear transformations $\sigma_3, w_3$ of $\mathfrak{J}$  defined in Subsection \ref{subsection 3.2}. 

\noindent From Lemma \ref{lemma 3.2.6} (1), since we can easily confirm that $\gamma_3$ and $\sigma_3$ are commutative, $\tilde{\sigma}_3$ and $\tilde{w}_3$ are commutative in $\Aut(F_4)$: $\tilde{\sigma}_3\tilde{w}_3=\tilde{w}_3\tilde{\sigma}_3$.
\vspace{1mm}

Now, we will determine the structure of the group $(F_4)^{\sigma_3} \cap (F_4)^{w_3}$. Note that we can prove theorem below as in the proof of Theorem \ref{theorem 4.3.2}, however we give the proof as detailed as possible.

\begin{theorem}\label{theorem 4.4.1}
	The group $(F_4)^{\sigma_3} \cap (F_4)^{w_3}$ is isomorphic to the group $(SU(3)\times U(1) \times U(1))/\Z_3${\rm :} $(F_4)^{\sigma_3} \cap (F_4)^{w_3} \cong (SU(3)\times U(1) \times U(1))/\Z_3, \Z_3=\{(E,1,1), (\bm{\omega}E,\bm{\omega},\bm{\omega}),( \bm{\omega}^{-1}E,\allowbreak \bm{\omega}^{-1},\bm{\omega}^{-1}\}$.
\end{theorem}
\begin{proof}
	Let $S(U(1) \times U(1) \times U(1)) \subset SU(3)$.
	We define a mapping $\varphi_{{}_{F_4,\sigma_3, w_3}}: SU(3) \times S(U(1) \times U(1) \times U(1)) \to (F_4)^{\sigma_3} \cap (F_4)^{w_3}$ by
	\begin{align*}
	\varphi_{{}_{F_4,\sigma_3, w_3}}(P,L)(X_{\bm{C}}+M)=LX_{\bm{C}}L^*+PML^*,\,\,X_{\bm{C}}+M \in \mathfrak{J}(3,\C)\oplus M(3,\C)=\mathfrak{J}.
	\end{align*}
	Needless to say, this mapping is the restriction of the mapping $\varphi_{{}_{F_4,w_3}}$, that is, $\varphi_{{}_{F_4,\gamma_3, w_3}}(P,L)=\varphi_{{}_{F_4,w_3}}(P,L)$ (Theorem \ref{theorem 3.2.5}).
	
	As usual, we will prove that $\varphi_{{}_{F_4,\sigma_3, w_3}}$ is well-defined. It is clear that $\varphi_{{}_{F_4,\sigma_3, w_3}}(P,L) \in (F_4)^{w_3}$, and using $\sigma_3=\varphi_{{}_{F_4,w_3}}(E,\diag(1,\ov{\bm{\omega}},\bm{\omega}))$ (Lemma \ref{lemma 3.2.6} (2)), it follows that 
	\begin{align*}
	{\sigma_3}^{-1}\varphi_{{}_{F_4,\sigma_3, w_3}}(P,L)\sigma_3
	&=\varphi_{{}_{F_4,w_3}}(E,\diag(1,\ov{\bm{\omega}},\bm{\omega}))^{-1}\varphi_{{}_{F_4,\gamma_3, w_3}}(P,L)\varphi_{{}_{F_4,w_3}}(E,\diag(1,\ov{\bm{\omega}},\bm{\omega}))
	\\
	&=\varphi_{{}_{F_4,w_3}}(E,\diag(1,\bm{\omega},\ov{\bm{\omega}}))\varphi_{{}_{F_4,w_3}}(P,L)\varphi_{{}_{F_4,w_3}}(E,\diag(1,\ov{\bm{\omega}},\bm{\omega}))
	\\
	&=\varphi_{{}_{F_4,w_3}}(P,\diag(1,\bm{\omega},\ov{\bm{\omega}})L\diag(1,\ov{\bm{\omega}},\bm{\omega})),\,\,L=\diag(a,b,c)
	\\
	&=\varphi_{{}_{F_4,w_3}}(P,L)
	\\
	&=\varphi_{{}_{F_4,\sigma_3,w_3}}(P,L).
	\end{align*}
	Hence we have that $\varphi_{{}_{F_4,\sigma_3, w_3}}(P,L) \in (F_4)^{\sigma_3}$.  Thus $\varphi_{{}_{F_4,\sigma_3, w_3}}$ is well-defined. Subsequently, since $\varphi_{{}_{F_4,\sigma_3,w_3}}$ is the restriction of the mapping $\varphi_{{}_{F_4,w_3}}$, we easily see that $\varphi_{{}_{F_4,\sigma_3,w_3}}$ is a homomorphism. 
	
	Next, we will prove that $\varphi_{{}_{F_4,\sigma_3,w_3}}$ is surjective. Let $\alpha \in (F_4)^{\sigma_3} \cap (F_4)^{w_3} \subset (F_4)^{w_3}$.
	 There exist $P, A \in SU(3)$ such that $\alpha=\varphi_{{}_{F_4,w_3}}(P,A)$ (Theorem \ref{theorem 3.2.5}). Moreover, from the condition $\alpha \in (F_4)^{\sigma_3}$, that is, ${\sigma_3}^{-1}\varphi_{{}_{F_4,w_3}}(P,A)\sigma_3=\varphi_{{}_{F_4,w_3}}(P,A)$, and using ${\sigma_3}^{-1}\varphi_{{}_{F_4,w_3}}(P,A)\sigma_3\allowbreak=\varphi_{{}_{F_4,w_3}}(P,\diag(1,\bm{\omega},\ov{\bm{\omega}})A\diag(1,\ov{\bm{\omega}},\bm{\omega}))$ (Lemma \ref{lemma 3.2.6} (2)), we have that
	\begin{align*}
	&\,\,\,{\rm(i)}\,\left\{
	\begin{array}{l}
	P=P\\
	\diag(1,\bm{\omega},\ov{\bm{\omega}})A\,\diag(1,\ov{\bm{\omega}},\bm{\omega})=A, 
	\end{array} \right.
	\qquad
	{\rm(ii)}\,\left\{
	\begin{array}{l}
	P=\bm{\omega}P\\
	\diag(1,\bm{\omega},\ov{\bm{\omega}})A\,\diag(1,\ov{\bm{\omega}},\bm{\omega})=\bm{\omega}A, 
	\end{array} \right.
	\\[2mm]
	&{\rm(iii)}\,\left\{
	\begin{array}{l}
	P=\bm{\omega}^{-1}P\\
	\diag(1,\bm{\omega},\ov{\bm{\omega}})A\,\diag(1,\ov{\bm{\omega}},\bm{\omega})=\bm{\omega}^{-1}A.
	\end{array} \right.
	\end{align*}
	The Cases (ii) and (iii) are impossible because of $P\not=0$. As for the Case (i), from the first condition, by doing straightforward computation $A$ takes the following form  $\diag(a,b,c), a,b,c \in U(1), abc=1$, that is, $A \in S(U(1)\times U(1)\times U(1))$. Needless to say, $P \in SU(3)$.  Hence there exist $P \in SU(3)$ and $A \in S(U(1)\times U(1) \times U(1))$ such that $\alpha=\varphi_{{}_{F_4,w_3}}(P,A)$. Namely, there exist $P \in SU(3)$ and $A \in S(U(1)\times U(1) \times U(1))$ such that $\alpha=\varphi_{{}_{F_4,\sigma_3,w_3}}(P,A)$. The proof of surjective is completed.
	
	Finally, we will determine $\Ker\,\varphi_{{}_{F_4,\sigma_3,w_3}}$. However, from $\Ker\,\varphi_{{}_{F_4,w_3}}=\{(E,E),(\bm{\omega}E,\bm{\omega}E), \allowbreak (\bm{\omega}^{-1}E,\bm{\omega}^{-1}E)\}$, we easily obtain that $\Ker\,\varphi_{{}_{F_4,\sigma_3,w_3}}=\{(E,E),(\bm{\omega}E,\bm{\omega}E), (\bm{\omega}^{-1}E,\bm{\omega}^{-1}E)\} \cong \Z_3$. Thus we have the isomorphism $(F_4)^{\sigma_3} \cap (F_4)^{w_3} \cong (SU(3)\times S(U(1) \times U(1) \times U(1)))/\Z_3$.
	
	Here, as in the proof of Theorem \ref{theorem 4.3.2} we have the isomorphism $U(1) \times U(1) \cong S(U(1) \times U(1) \times U(1))$.
	
	Therefore we have the required isomorphism 
	\begin{align*}
	(F_4)^{\sigma_3} \cap (F_4)^{w_3} \cong (SU(3)\times U(1) \times U(1))/\Z_3,
	\end{align*}
	where $\Z_3=\{(E,1,1), (\bm{\omega}E,\bm{\omega},\bm{\omega}),( \bm{\omega}^{-1}E,\allowbreak \bm{\omega}^{-1},\bm{\omega}^{-1}\}$.
\end{proof}
\vspace{1mm}

Thus, since the group $(F_4)^{\sigma_3} \cap (F_4)^{w_3}$ is connected from Theorem \ref{theorem 4.4.1}, we have an exceptional $\varmathbb{Z}_3 \times \varmathbb{Z}_3$-symmetric space
\begin{align*}
F_4/((SU(3)\times U(1) \times U(1))/\Z_3).
\end{align*} 

\begin{assertion}\label{assertion}
	On Theorem \ref{theorem 4.4.1} from a different view point. 
\end{assertion}

	First, let $U(3) \subset Sp(3)$. Then, we can embed $U(3)$ into $F_4$ using the mapping $\varphi_{{}_{F_4,\gamma_3}}$ as follows:
	\begin{align*}
		\varphi_{{}_{F_4,\gamma_3}}(1,U)(M+\a)=UMU^*+\a U^*,\,\,M+\a \in \mathfrak{J}(3,\H) \oplus \H^3=\mathfrak{J},
	\end{align*}
	more detail, since $w_3$ induces an automorphism of the group $(F_4)_{E_1, F_1(1),F_1(e_1)}$, it follows that $\varphi_{{}_{F_4,\gamma_3}}(1,U) \in ((F_4)_{E_1, F_1(1),F_1(e_1)})^{w_3} \cong (Spin(7))^{w_3}$ , where $Spin(7)$ is defined in Theorem \ref{theorem 3.2.4}. Here, we denote $\varphi_{{}_{F_4,\gamma_3}}(1,U)$ by $\varphi(U)$: $\varphi(U)=\varphi_{{}_{F_4,\gamma_3}}(1,U)$, and we define a mapping $\psi: U(1) \times U(3) \to (F_4)^{\sigma_3} \cap (F_4)^{w_3}$ by
	\begin{align*}
		\psi(a,U)=D_a\varphi(U),
	\end{align*}
	where $D_a$ is defined in Subsection 3.2. Then the mapping $\psi$ induces the isomorphism $(F_4)^{\sigma_3} \cap (F_4)^{w_3} \cong (U(1)\times U(3))/\Z_3$, where $\Z_3=\{(1,E), (\bm{\omega},\bm{\omega}^{-1}E), (\bm{\omega}^{-1}, \bm{\omega}E) \}$.

\subsection{Case 5: $\{1, \tilde{\gamma}_3,  \tilde{\gamma}_3{}^{-1}\} \times \{1, \tilde{\sigma}_3,  \tilde{\sigma}_3{}^{-1}\}$-symmetric space}

Let the $C$-linear transformations $\gamma_3, \sigma_3$ of $\mathfrak{J}^C$  defined in Subsection \ref{subsection 3.3}.

\noindent From Lemma \ref{lemma 3.3.8} (1), since we can easily confirm that $\gamma_3$ and $\sigma_3$ are commutative, $\tilde{\gamma}_3$ and $\tilde{\sigma}_3$ are commutative in $\Aut(E_6)$: $\tilde{\gamma}_3\tilde{\sigma}_3=\tilde{\sigma}_3\tilde{\gamma}_3$.
\vspace{1mm}

Before determining the structure of the group $(E_6)^{\gamma_3} \cap (E_6)^{\sigma_3}$, we prove proposition and lemma needed in the proof of theorem below.
\vspace{1mm}

We define a $C$-linear transformation $\sigma'_3$ of $\mathfrak{J}^C$ by
\begin{align*}
			\sigma'_3=\varphi_{{}_{E_6,\gamma_3}}(1,\diag(1,1,\omega,\omega,\tau\omega,\tau\omega)) \in (E_6)^{\gamma_3} \subset E_6,
\end{align*}
where $\omega=-(1/2)+(\sqrt{3}/2)i \in C$.

Let an element 
\begin{align*}R:=\scalebox{0.8}{$\begin{pmatrix}
	1&&&&&\\
	&1&&&&\\
	&&&&1&\\
	&&&1&&\\
	&&-1&&&\\
	&&&&&1
	\end{pmatrix}$} \in SO(6) \subset SU(6), 
\end{align*}
where the blanks are $0$, and we consider an element $\varphi_{{}_{E_6,\gamma_3}}(1,R) \in (E_6)^{\gamma_3} \subset E_6$. Here, we denote this element by $\delta_R$: $\delta_R=\varphi_{{}_{E_6,\gamma_3}}(1,R)$.
Then by doing straightforward computation, we have that $\sigma_3\delta_R=\delta_R\sigma'_3$, that is, $\sigma_3$ is conjugate to $\sigma'_3$ under $\delta_R \in (E_6)^{\gamma_3} \subset E_6$: $\sigma_3 \sim \sigma'_3$. Moreover, $\sigma'_3$ induces the automorphism $\tilde{\sigma'}_3$ of order $3$ on $E_6$: $\tilde{\sigma'}_3(\alpha)={\sigma'_3}^{-1}\alpha\sigma'_3, \alpha \in E_6$.
\vspace{1mm}

Then we have the following proposition.

\begin{proposition}\label{proposition 4.5.1}
	The group $(E_6)^{\gamma_3} \cap (E_6)^{\sigma_3}$ is isomorphic to the group $(E_6)^{\gamma_3} \cap (E_6)^{\sigma'_3}${\rm :} $(E_6)^{\gamma_3} \cap (E_6)^{\sigma_3} \cong (E_6)^{\gamma_3} \cap (E_6)^{\sigma'_3}$.
\end{proposition}
\begin{proof}
	We define a mapping $g_{{}_{451}}: (E_6)^{\gamma_3} \cap (E_6)^{\sigma_3} \to (E_6)^{\gamma_3} \cap (E_6)^{\sigma'_3}$ by
	\begin{align*}
			g_{{}_{451}}(\alpha)={\delta_R}^{-1}\alpha\delta_R.	
	\end{align*}
	In order to prove this isomorphism, it is sufficient to show that $g_{{}_{452}}$ is well-defined. 
	
	\noindent First, we will show that $g_{{}_{451}} \in (E_6)^{\gamma_3}$. Since it follows from $\delta_R=\varphi_{{}_{E_6,\gamma_3}}(1,R)$ and $\gamma_3=\varphi_{{}_{E_6,\gamma_3}}(\omega,E)$ that $\delta_R\gamma_3=\gamma_3\delta_R$, we have that $g_{{}_{451}} \in (E_6)^{\gamma_3}$. Similarly, from $\sigma_3\delta_R=\delta_R\sigma'_3$ we have that $g_{{}_{451}} \in (E_6)^{\sigma'_3}$. Hence $g_{{}_{451}}$ is well-defined. With above, the proof of this proposition is completed.	
\end{proof}
\vspace{1mm}

Subsequently, we will prove the following lemma. 

\begin{lemma}\label{lemma 4.5.2}
	The group $S(U(2)\times U(2)\times U(2))$ is isomorphic to the group $(U(1) \times U(1)\times SU(2)\times SU(2)\times SU(2))/(\Z_2\times\Z_2)${\rm :} $S(U(2)\times U(2)\times U(2)) \cong (U(1) \times U(1)\times SU(2)\times SU(2)\times SU(2))/(\Z_2\times\Z_2), \Z_2=\!\{(1,1,E,E,E), (1,-1,E,-E,E) \}, \Z_2=\!\{(1,1,E,E,E), (-1,1,-E,\allowbreak E,E) \}$.
\end{lemma}
\begin{proof}
	We define a mapping $f_{{}_{452}}:U(1) \times U(1)\times SU(2)\times SU(2)\times SU(2) \to S(U(2)\times U(2)\times U(2))$ by
	\begin{align*}
			f_{{}_{452}}(a,b,A,B,C)=\left( 
			\begin{array}{ccc}
			 a\mbox{\large {$A$}} & & {\raisebox{-7pt}[0pt]{\large $0$}}
			 \\[2mm]
			 & b\mbox{\large {$B$}} & 
			 \\[2mm]
			 {\raisebox{1pt}[0pt]{\large $0$}}&& (ab)^{-2}\mbox{\large {$C$}}
			\end{array}\right) \in SU(6).
	\end{align*}
	Then it is clear that $f_{{}_{452}}$ is well-defined and a homomorphism. 
	
	We will prove that $f_{{}_{452}}$ is surjective. Let $P \in S(U(2)\times U(2)\times U(2))$. Then $P$ takes the form of $\diag(P_1,P_2,P_3),P_j \in U(2), (\det\,P_1)(\det\,P_2)(\det\,P_3)=1$. Here, since $P_1 \in U(2)$, we see that $\det\,P_1 \in U(1)$. We choose $a \in U(1)$ such that $a^2=\det\,P_1$, and set $A=(1/a)P_1$. Then we have that $ A \in SU(2)$. Similarly, for $P_2 \in U(2)$, there exist $b \in U(1)$ and $B \in SU(2)$ such that $P_2=bB, b^2=\det\,P_2$. From $(\det\,P_1)(\det\,P_2)(\det\,P_3)=1$, we have that $\det\,P_3=(ab)^{-2}$. Set $C=(ab)^2P_3$. Then we have that $C \in SU(2)$. With above, the proof of surjective is completed.
	
	Finally, we will determine $\Ker\,f_{{}_{452}}$. It follows from the kernel of definition that
	\begin{align*}
			\Ker\,f_{{}_{452}}&=\{(a,b,A,B,C)\in U(1)^{\times 2}\times SU(2)^{\times 3} \,|\,f_{{}_{453}}(a,b,A,B,C)=E \}
			\\
			&=\{(a,b,a^{-1}E,b^{-1}E,(ab)^2E)\in U(1)^{\times 2}\times SU(2)^{\times 3} \,|\,a^2=b^2=1 \}
			\\
			&=\{(1,1,E,E,E), (1,-1,E,-E,E),(-1,1,-E,E,E), (-1,-1,-E,-E,E) \}
			\\
			&=\{(1,1,E,E,E), (1,-1,E,-E,E) \} \times \{(1,1,E,E,E), (-1,1,-E,E,E) \}
			\\
			& \cong \Z_2 \times \Z_2.
	\end{align*}
	
	Therefore we have the required isomorphism 
	\begin{align*}
				S(U(2)\times U(2)\times U(2)) \cong (U(1) \times U(1)\times SU(2)\times SU(2)\times SU(2))/(\Z_2\times\Z_2).
	\end{align*}
\end{proof}

Now, we will determine the structure of the group $(E_6)^{\gamma_3} \cap (E_6)^{\sigma_3}$.

\begin{theorem}\label{theorem 4.5.3}
	The group $(E_6)^{\gamma_3} \cap (E_6)^{\sigma_3}$ is isomorphic the group $(U(1)\times U(1) \times U(1)\allowbreak \times SU(2) \times SU(2)\times SU(2))/(\Z_2\times\Z_2\times\Z_2\times\Z_2)${\rm :} $(E_6)^{\gamma_3} \cap (E_6)^{\sigma_3} \cong (U(1)\times U(1) \times U(1)\times SU(2) \allowbreak \times SU(2)\times SU(2))/(\Z_2\times\Z_2\times\Z_2\times\Z_2), \Z_2=\{(1,1,1,E,E,E), (-1,1,1,-E,-E,E) \},\,\Z_2=\{(1,1,1,E,E,E), (-1,1,-1,-E,E,E) \},\Z_2\!=\!\{(1,1,1,E,E,E), (-1,-1,1,-E,-E,E) \},\!\Z_2\allowbreak=\{(1,1,1,E,E,E), (-1,-1,-1,E,E,E) \}$.
\end{theorem}
\begin{proof}
	Let $S(U(2)\times U(2)\times U(2)) \subset SU(6)$. 
	We define a mapping $\varphi_{{}_{E_6,\gamma_3,\sigma'_3}}: U(1)\times S(U(2)\times U(2)\times U(2)) \to (E_6)^{\gamma_3} \cap (E_6)^{\sigma'_3}$ by
	\begin{align*}
	\varphi_{{}_{E_6,\gamma_3,\sigma'_3}}(s, P)(M+\a)&={k_J}^{-1}(P(k_J M){}^t\!P)+s\a k^{-1}(\tau \,{}^t\!P), 
  \\
	&\hspace*{40mm}M+\a \in \mathfrak{J}(3, \H)^C \oplus (\H^3)^C\!\!=\!\mathfrak{J}^C.
	\end{align*}
	Needless to say, this mapping is the restriction of the mapping $\varphi_{{}_{E_6,\gamma_3}}$, that is, $\varphi_{{}_{E_6,\gamma_3,\sigma'_3}}(s, P)=\varphi_{{}_{E_6,\gamma_3}}(s,P)$ (Theorem \ref{theorem 3.3.2}). 
	
	First, we will prove that $\varphi_{{}_{E_6,\gamma_3,\sigma'_3}}$ is well-defined. It is clear that $\varphi_{{}_{E_6,\gamma_3,\sigma'_3}}(s,P) \in (E_6)^{\gamma_3}$, and it follows from $\sigma'_3=\varphi_{{}_{E_6,\gamma_3}}(1,\diag(1,1,\omega,\omega,\tau\omega,\tau\omega))$ that
	\begin{align*}
			&\quad {\sigma'_3}^{-1}\varphi_{{}_{E_6,\gamma_3,\sigma'_3}}(s,P)\sigma'_3
			\\
			&=\varphi_{{}_{E_6,\gamma_3}}(1,\diag(1,1,\omega,\omega,\tau\omega,\tau\omega))^{-1}\varphi_{{}_{E_6,\gamma_3,\sigma'_3}}(s,P)\varphi_{{}_{E_6,\gamma_3}}(1,\diag(1,1,\omega,\omega,\tau\omega,\tau\omega))
			\\
			&=\varphi_{{}_{E_6,\gamma_3}}(1,\diag(1,1,\tau\omega,\tau\omega,\omega,\omega))\varphi_{{}_{E_6,\gamma_3}}(s,P)\varphi_{{}_{E_6,\gamma_3}}(1,\diag(1,1,\omega,\omega,\tau\omega,\tau\omega))
			\\
			&=\varphi_{{}_{E_6,\gamma_3}}(s,\diag(1,1,\tau\omega,\tau\omega,\omega,\omega)P\diag(1,1,\omega,\omega,\tau\omega,\tau\omega)),P=\diag(P_1,P_2,P_3)
			\\
			&=\varphi_{{}_{E_6,\gamma_3}}(s,\diag(P_1,(\tau\omega E) P_2(\omega E),(\omega E) P_3(\tau\omega E)))
			\\
			&=\varphi_{{}_{E_6,\gamma_3}}(s,P)
			\\
			&=\varphi_{{}_{E_6,\gamma_3,\sigma'_3}}(s,P).
	\end{align*}
	Hence we have that $\varphi_{{}_{E_6,\gamma_3,\sigma'_3}}(s,P) \in (E_6)^{\sigma'_3}$. Thus $\varphi_{{}_{E_6,\gamma_3,\sigma'_3}}$ is well-defined. Subsequently, since $\varphi_{{}_{E_6,\gamma_3,\sigma'_3}}$ is the restriction of the mapping $\varphi_{{}_{E_6,\gamma_3}}$, we easily see that $\varphi_{{}_{E_6,\gamma_3,\sigma'_3}}$ is a homomorphism.
	
	Next, we will prove that $\varphi_{{}_{E_6,\gamma_3,\sigma'_3}}$ is surjective. Let $\alpha \in (E_6)^{\gamma_3} \cap (E_6)^{\sigma'_3} \subset (E_6)^{\gamma_3}$. There exist $s \in U(1)$ and $A \in SU(6)$ such that $\alpha=\varphi_{{}_{E_6,\gamma_3}}(s,A)$ (Theorem \ref{theorem 3.3.2}). Moreover, from the condition $\alpha \in (E_6)^{\sigma'_3}$, that is, ${\sigma'_3}^{-1}\varphi_{{}_{E_6,\gamma_3}}(s,A)\sigma'_3=\varphi_{{}_{E_6,\gamma_3}}(s,A)$, and using ${\sigma'_3}^{-1}\varphi_{{}_{E_6,\gamma_3}}(s,A)\sigma'_3=\varphi_{{}_{E_6,\gamma_3}}(s,\diag(1,1,\tau\omega,\tau\omega,\omega,\omega)A\,\diag(1,1,\omega,\omega,\tau\omega,\tau\omega))$ (Lemma \ref{lemma 3.3.8} (1)), we have that
	\begin{align*}
	&\left\{
	    \begin{array}{l}
	      s=s \\
	      \diag(1,1,\tau\omega,\tau\omega,\omega,\omega)A\,\diag(1,1,\omega,\omega,\tau\omega,\tau\omega)=A   
	    \end{array}\right. 
	 \\
	&\hspace*{45mm}{\text{or}}
	 \\
	&\left\{
	\begin{array}{l}
	s=-s \\
	\diag(1,1,\tau\omega,\tau\omega,\omega,\omega)A\,\diag(1,1,\omega,\omega,\tau\omega,\tau\omega)=-A.   
	\end{array}\right.     
	\end{align*}
	The latter case is impossible because of $s\not=0$. As for the former case, from the second condition, by doing straightforward computation $A$ takes the following form $\diag(A_1, A_2, A_3), A_j \in U(2), (\det\,A_1)(\det\,A_2)(\det\,A_3)=1$, that is, $A \in S(U(2)\times U(2)\times U(2))$.
	Needless to say, $s \in U(1)$. Hence there exist $s \in U(1)$ and $P \in S(U(2)\times U(2) \times U(2))$ such that $\alpha=\varphi_{{}_{E_6,\gamma_3}}(s,P)$. Namely, there exist $s \in U(1)$ and $P \in S(U(2)\times U(2) \times U(2))$ such that $\alpha=\varphi_{{}_{E_6,\gamma_3,\sigma'_3}}(s,P)$. The proof of surjective is completed.
	
	Finally, we will determine $\Ker\,\varphi_{{}_{E_6,\gamma_3,\sigma'_3}}$. However, from $\Ker\,\varphi_{{}_{E_6,\gamma_3}}=\{(1,E),(-1,-E) \}$, we easily obtain that $\Ker\,\varphi_{{}_{E_6,\gamma_3,\sigma'_3}}=\{(1,E),(-1,-E) \} \cong \Z_2$. Thus we have the isomorphism $(E_6)^{\gamma_3} \cap (E_6)^{\sigma'_3} \cong (U(1)\times S(U(2)\times U(2)\times U(2)))/\Z_2$. Here, from Proposition \ref{proposition 4.5.1} we have the isomorphism $(E_6)^{\gamma_3} \cap (E_6)^{\sigma_3} \cong (U(1)\times S(U(2)\times U(2)\times U(2)))/\Z_2$. Moreover, by Lemma \ref{lemma 4.5.2} we have the required isomorphism 
	\begin{align*}
	(E_6)^{\gamma_3} \cap (E_6)^{\sigma_3} \!\cong (U(1)\times U(1) \times U(1)\times SU(2) \allowbreak \times SU(2)\times SU(2))/(\Z_2\times\Z_2\times\Z_2\times\Z_2), 
	\end{align*}
	where 
	\begin{align*}
	&\Z_2=\{(1,1,1,E,E,E), (-1,1,1,-E,-E,E) \},
	\\
	&\Z_2=\{(1,1,1,E,E,E), (-1,1,-1,-E,E,E) \},
	\\
	&\Z_2=\{(1,1,1,E,E,E), (-1,-1,1,-E,-E,E) \},
	\\
	&\Z_2\allowbreak=\{(1,1,1,E,E,E), (-1,-1,-1,E,E,E) \}.
	\end{align*}
\end{proof}

Thus, since the group $(E_6)^{\gamma_3} \cap (E_6)^{\sigma_3}$ is connected from Theorem \ref{theorem 4.5.3}, we have an exceptional $\varmathbb{Z}_3 \times \varmathbb{Z}_3$-symmetric space
\begin{align*}
					E_6/((U(1)\times U(1) \times U(1)\times SU(2) \allowbreak \times SU(2)\times SU(2))/(\Z_2\times\Z_2\times\Z_2\times\Z_2)).
\end{align*}

\subsection{Case 6: $\{1, \tilde{\gamma}_3,  \tilde{\gamma}_3{}^{-1}\} \times \{1, \tilde{\nu}_3,  \tilde{\nu}_3{}^{-1}\}$-symmetric space}

Let the $C$-linear transformations $\gamma_3, \nu_3$ of $\mathfrak{J}^C$  defined in Subsection \ref{subsection 3.3}.

\noindent From Lemma \ref{lemma 3.3.8} (1), together with $\gamma_3=\varphi_{{}_{E_6,\gamma_3}}(\omega,E)$, since we can easily confirm that $\gamma_3$ and $\nu_3$ are commutative, $\tilde{\gamma}_3$ and $\tilde{\nu}_3$ are commutative in $\Aut(E_6)$: $\tilde{\gamma}_3\tilde{\nu}_3=\tilde{\nu}_3\tilde{\gamma}_3$. 

Before determining the structure of the group $(E_6)^{\gamma_3} \cap (E_6)^{\nu_3}$, we prove lemma needed in the proof of theorem below.

\begin{lemma}\label{lemma 4.6.1}
	The group $S(U(1)\times U(5))$ is isomorphism the group $(U(1)\times SU(5))/\Z_5${\rm :} $S(U(1)\times U(5))\! \cong\! (U(1)\times SU(5))/\Z_5, \Z_5\!=\!\{(\varepsilon_k, {\varepsilon_k}^{-1}E) | \varepsilon_k\!=\!\exp((2\pi i/5)k), k\!=0,1,2,3,4\}$.
\end{lemma}
\begin{proof}
	We define a mapping $f_{{}_{461}}:U(1) \times SU(5) \to S(U(1)\times U(5))$ by
	\begin{align*}
	f_{{}_{461}}(t, T)=\scalebox{0.7}{$
		\left(\begin{array}{cccccccc@{\!}}
		&\multicolumn{2}{c}{\raisebox{-15pt}[0pt][0pt]{\Large$t^{-5}$}}&&&&
		\\
		&&&&\multicolumn{2}{c}{\raisebox{-5pt}[0pt]{\Large $0$}}&
		\\
		&&&&&&&
		\\
		&&&&\multicolumn{2}{c}
		{\raisebox{-15pt}[0pt][0pt]{\Large $t$}\,\raisebox{-18pt}[0pt][0pt]{\huge $T$}}&
		\\
		&\multicolumn{2}{c}{\raisebox{0pt}[0pt]{\Large $0$}}&&&&
		\\[-2mm]
		&&&&&&&
		\end{array}\right)$}.
	\end{align*}
	Then it is clear that $f_{{}_{461}}$ is well-defined and a homomorphism. 
	
	Now, we will prove that $f_{{}_{461}}$ is surjective. Let $P \in S(U(1) \times U(5))$. Then $P$ takes the form of 
	\scalebox{0.6}
	{$\left(\begin{array}{cccccccc@{\!}}
			&\multicolumn{2}{c}{\raisebox{-15pt}[0pt][0pt]{\Large $s$}}&&&&
			\\
			&&&&\multicolumn{2}{c}{\raisebox{-5pt}[0pt]{\Large $0$}}&
			\\
			&&&&&&&
			\\
			&&&&\multicolumn{2}{c}{\raisebox{-18pt}[0pt][0pt]{\huge $S$}}&
			\\
			&\multicolumn{2}{c}{\raisebox{0pt}[0pt]{\Large $0$}}&&&&
			\\[-2mm]
			&&&&&&&
		\end{array}\right)$},\,\,$s \in U(1), S \in U(5), s(\det S)=1$.
	Here, since $S \in U(5)$, we see that $\det\,S \in U(1)$, and so we choose $t \in U(1)$ such that $t^5=\det\,S$. Set $T=t^{-1}S$, then we have that $T \in SU(5)$ and $s=t^{-5}$. With above, the proof of surjective is completed.
	
	Finally, we will determine $\Ker\,f_{{}_{461}}$. It follows from the definition of kernel that 
	\begin{align*}
	\Ker\,f_{{}_{461}}&=\{(t,T) \in U(1)\times SU(5)\,|\, f_{{}_{461}}(t,T)=E \}
	\\
	&=\{(t,T) \in U(1)\times SU(5)\,|\,t^5=1, T=t^{-1}E \}
	\\
	&=\{(\varepsilon_k, {\varepsilon_k}^{-1}E) \,|\, \varepsilon_k=\exp((2\pi i/5)k), k=0,1,2,3,4\}
	\\
	& \cong \Z_5.
	\end{align*}
	
	Therefore we have the required isomorphism 
	\begin{align*}
	S(U(1) \times U(5)) \cong (U(1)\times SU(5))/\Z_5.
	\end{align*}
\end{proof}
 
Now, we will determine the structure of the group $(E_6)^{\gamma_3} \cap (E_6)^{\nu_3}$.

\begin{theorem}\label{theorem 4.6.2}
	The group $(E_6)^{\gamma_3} \cap (E_6)^{\nu_3}$ is isomorphic to the group $(U(1)\times U(1)\times SU(5))/\Z_2${\rm :} $(E_6)^{\gamma_3} \cap (E_6)^{\nu_3} \cong (U(1)\times U(1)\times SU(5))/(\Z_2\times \Z_5), \Z_2=\{(1,1,E), (-1,-1,\allowbreak -E) \}, \Z_5=\{(1,\varepsilon_i, {\varepsilon_i}^{-1}E) \,|\, \varepsilon_i=\exp ((2\pi i/5)k), k=0,1,2,3,4\}$.
\end{theorem}
\begin{proof}
	Let $S(U(1)\times U(5)) \subset SU(6)$. Then we define a mapping $\varphi_{{}_{E_6, \gamma, \nu_3}}: U(1)\times S(U(1)\times U(5)) \to (E_6)^{\gamma_3} \cap (E_6)^{\nu_3}$ by
	\begin{align*}
	\varphi_{{}_{E_6,\gamma_3,\nu_3}}(s, P)(M+\a)&={k_J}^{-1}(P(k_J M){}^t\!P)+s\a k^{-1}(\tau \,{}^t\!P), 
	\\
	&\hspace*{40mm}M+\a \in \mathfrak{J}(3, \H)^C \oplus (\H^3)^C=\mathfrak{J}^C.
	\end{align*}
	Needless to say, this mapping is the restriction of the mapping $\varphi_{{}_{E_6,\gamma_3}}$, that is, $\varphi_{{}_{E_6,\gamma_3,\nu_3}}(s, P)=\varphi_{{}_{E_6,\gamma_3}}(s,P)$ (Theorem \ref{theorem 3.3.2}).  

	First, we will prove that $\varphi_{{}_{E_6,\gamma_3,\nu_3}}$ is well-defined. It is clear that $\varphi_{{}_{E_6,\gamma_3,\nu_3}}(s, P) \in (E_6)^{\gamma_3}$, and using $\nu_3=\varphi_{{}_{E_6,\gamma_3}}(1,\diag(\nu^5, \nu^{-1},\nu^{-1},\nu^{-1},\nu^{-1},\nu^{-1}))$ (Lemma \ref{lemma 3.3.8} (1)), it follows that 
	\begin{align*}
			&\quad {\nu_3}^{-1}\varphi_{{}_{E_6, \gamma_3, \nu_3}}(s,P)\nu_3
			\\
			&=\varphi_{{}_{E_6,\gamma_3}}(1,\diag(\nu^5, \nu^{-1},\ldots,\nu^{-1}))^{-1}\varphi_{{}_{E_6, \gamma_3, \nu_3}}(s,P)\varphi_{{}_{E_6,\gamma_3}}(1,\diag(\nu^5, \nu^{-1},\ldots,\nu^{-1}))
			\\
			&=\varphi_{{}_{E_6,\gamma_3}}(1,\diag(\nu^{-5}, \nu,\ldots,\nu))\varphi_{{}_{E_6, \gamma_3}}(s,P)\varphi_{{}_{E_6,\gamma_3}}(1,\diag(\nu^5, \nu^{-1},\ldots,\nu^{-1}))
			\\
			&=\varphi_{{}_{E_6, \gamma_3}}(s,\diag(\nu^{-5}, \nu,\ldots,\nu)P\,\diag(\nu^5, \nu^{-1},\ldots,\nu^{-1})), P=\scalebox{0.6}{$
				\left( \begin{array}{cccccccc@{\!}}
				&\multicolumn{2}{c}{\raisebox{-15pt}[0pt][0pt]{\Large$t$}}&&&&
				\\
				&&&&\multicolumn{2}{c}{\raisebox{-5pt}[0pt]{\Large $0$}}&
				\\
				&&&&&&&
				\\
				&&&&\multicolumn{2}{c}{\raisebox{-18pt}[0pt][0pt]{\huge $U$}}&
				\\
				&\multicolumn{2}{c}{\raisebox{0pt}[0pt]{\Large $0$}}&&&&
				\\[-2mm]
				&&&&&&&
				\end{array}\right)$}
			\\
			&=\varphi_{{}_{E_6, \gamma_3}}(s,P)
			\\
			&=\varphi_{{}_{E_6, \gamma_3, \nu_3}}(s,P)
	\end{align*}
	Hence we have that $\varphi_{{}_{E_6, \gamma_3, \nu_3}}(s,P) \in (E_6)^{\nu_3}$. Thus $\varphi_{{}_{E_6,\gamma_3,\nu_3}}$ is well-defined. Subsequently, since $\varphi_{{}_{E_6,\gamma_3,\nu_3}}$ is the restriction of the mapping $\varphi_{{}_{E_6,\gamma_3}}$, we easily see  that $\varphi_{{}_{E_6,\gamma_3,\nu_3}}$ is a homomorphism.
	
	Next, we will prove that $\varphi_{{}_{E_6, \gamma_3, \nu_3}}$ is surjective. Let $\alpha \in  (E_6)^{\gamma_3} \cap (E_6)^{\nu_3} \subset (E_6)^{\nu_3}$. There exist $ q \in Sp(1)$ and $P \in S(U(1) \times U(5))$ such that $\alpha=\varphi_{{}_{E_6, \nu_3}}(q, P)$ (Theorem \ref{theorem 3.3.5}). Moreover, from the condition $\alpha \in (E_6)^{\gamma_3}$, that is, ${\gamma_3}^{-1}\varphi_{{}_{E_6, \nu_3}}(q, P)\gamma_3=\varphi_{{}_{E_6, \nu_3}}(q, P)$, and note that $\gamma_3=\varphi_{{}_{E_6, \nu_3}}(\omega,E)(=\varphi_{{}_{E_6,\gamma_3}}(\omega,E))$ (Lemma \ref{lemma 3.3.8} (1)), since it follows that
	${\gamma_3}^{-1}\varphi_{{}_{E_6, \nu_3}}(q, P)\gamma_3=\varphi_{{}_{E_6,\nu_3}}(\omega^{-1}q\omega, P)$, we have that
	\begin{align*}
				\left\{
				\begin{array}{l}
				\omega^{-1}q\omega=q \\
				P=P
				\end{array}\right.
				\quad {\text{or}}\quad
				\left\{
				\begin{array}{l}
				\omega^{-1}q\omega=-q \\
				P=-P.
				\end{array}\right.
	\end{align*}
	The latter case is impossible because of $P\not=0$. As for the former case, from the first condition, we easily see that $q \in U(1)$, and needless to say, $P \in S(U(1)\times U(5))$.  Hence there exist $s \in U(1)$ and $P \in S(U(1)\times U(5))$ such that $\alpha=\varphi_{{}_{E_6,\nu_3}}(s,P)$. Namely, there exist $s \in U(1)$ and $P \in S(U(1)\times U(5))$ such that $\alpha=\varphi_{{}_{E_6,\gamma_3,\nu_3}}(s,P)$. The proof of surjective is completed.
	
	Finally, we will determine $\Ker\,\varphi_{{}_{E_6, \gamma_3, \nu_3}}$. However, from $\Ker\,\varphi_{{}_{E_6,\gamma_3}}=\{(1,E),(-1,-E) \}$, we easily obtain that $\Ker\,\varphi_{{}_{E_6, \gamma_3, \nu_3}}=\{(1,(1,E)),(-1,(-1,-E)) \} \cong \Z_2$. Thus we have the isomorphism $(E_6)^{\gamma_3} \cap (E_6)^{\nu_3} \cong (U(1)\times S(U(1)\times U(5)))/\Z_2$.

	Therefore, by Lemma \ref{lemma 4.6.1} we have the required isomorphism
	\begin{align*}
				(E_6)^{\gamma_3} \cap (E_6)^{\nu_3} \cong (U(1)\times U(1)\times SU(5))/(\Z_2\times \Z_5),
	\end{align*}
	where 
	\begin{align*}
			&\Z_2=\{(1,1,E), (-1,-1,-E) \},
			\\
			&\Z_5=\{(1,\varepsilon_i, {\varepsilon_i}^{-1}E) \,|\, \varepsilon_i=\exp ((2\pi i/5)k), k=0,1,2,3,4\}.
	\end{align*}
\end{proof}

Thus, since the group $(E_6)^{\gamma_3} \cap (E_6)^{\nu_3}$ is connected from Theorem \ref{theorem 4.6.2}, we have an exceptional $\varmathbb{Z}_3 \times \varmathbb{Z}_3$-symmetric space
\begin{align*}
E_6/((U(1)\times U(1)\times SU(5))/(\Z_2\times \Z_5)).
\end{align*}

\subsection{Case 7: $\{1, \tilde{\gamma}_3,  \tilde{\gamma}_3{}^{-1}\} \times \{1, \tilde{\mu}_3,  \tilde{\mu}_3{}^{-1}\}$-symmetric space}

Let the $C$-linear transformations $\gamma_3, \mu_3$ of $\mathfrak{J}^C$  defined in Subsection \ref{subsection 3.3}. 

\noindent From Lemma \ref{lemma 3.3.8} (1), since we can easily confirm that $\gamma_3$ and $\mu_3$ are commutative, $\tilde{\gamma}_3$ and $\tilde{\mu}_3$ are commutative in $\Aut(E_6)$: $\tilde{\gamma}_3\tilde{\mu}_3=\tilde{\mu}_3\tilde{\gamma}_3$.
\vspace{1mm}

Before determining the structure of the group $(E_6)^{\gamma_3} \cap (E_6)^{\mu_3}$, we prove proposition and lemma needed in the proof of theorem below.
\vspace{1mm}

We define a $C$-linear transformation $\mu'_3$ of $\mathfrak{J}^C$ by
\begin{align*}
\mu'_3=\varphi_{{}_{E_6,\gamma_3}}(1,\diag(\nu^{-2},\nu^2,\nu^{-1},\nu^{-1},\nu,\nu))\in (E_6)^{\gamma_3} \subset E_6,
\end{align*}
where ${\nu}=\exp(2\pi i/9)\in C$.

Let an element 
\begin{align*}Q:=\scalebox{0.8}{$\begin{pmatrix}
	1&&&&&\\
	&1&&&&\\
	&&1&&&\\
	&&&&1&\\
	&&&-1&&\\
	&&&&&1
	\end{pmatrix}$} \in SO(6) \subset SU(6), 
\end{align*}
where the blanks are $0$, and we consider an element $\varphi_{{}_{E_6,\gamma_3}}(1,Q) \in (E_6)^{\gamma_3} \subset E_6$. Here, we denote this element by $\delta_Q$: $\delta_Q=\varphi_{{}_{E_6,\gamma_3}}(1,Q)$.
Then by doing straightforward computation, we have that $\mu_3\delta_Q=\delta_Q\mu'_3$, that is, $\mu_3$ is conjugate to $\mu'_3$ under $\delta_Q \in (E_6)^{\gamma_3} \subset E_6$: $\mu_3 \sim \mu'_3$. Moreover, $\mu'_3$ induces the automorphism $\tilde{\mu'}_3$ of order $3$ on $E_6$: $\tilde{\mu'}_3(\alpha)={\mu'_3}^{-1}\alpha\mu'_3, \alpha \in E_6$.
\vspace{1mm}

Then we have the following proposition.

\begin{proposition}\label{proposition 4.7.1}
	The group $(E_6)^{\gamma_3} \cap (E_6)^{\mu_3}$ is isomorphic to the group $(E_6)^{\gamma_3} \cap (E_6)^{\mu'_3}${\rm :} $(E_6)^{\gamma_3} \cap (E_6)^{\mu_3} \cong (E_6)^{\gamma_3} \cap (E_6)^{\mu'_3}$.
\end{proposition}
\begin{proof}
	We define a mapping $g_{{}_{471}}: (E_6)^{\gamma_3} \cap (E_6)^{\mu'_3} \to (E_6)^{\gamma_3} \cap (E_6)^{\mu_3}$ by
	\begin{align*}
	g_{{}_{471}}(\alpha)=\delta_Q\alpha{\delta_Q}^{-1}.	
	\end{align*}
	In order to prove this isomorphism, it is sufficient to show that $g_{{}_{471}}$ is well-defined. 
	
	\noindent First, we will show that $g_{{}_{471}} \in (E_6)^{\gamma_3}$. Since it follows from $\delta_Q=\varphi_{{}_{E_6,\gamma_3}}(1,Q)$ and $\gamma_3=\varphi_{{}_{E_6,\gamma_3}}(\omega,E)$ that $\delta_Q\gamma_3=\gamma_3\delta_Q$, we have that $g_{{}_{471}} \in (E_6)^{\gamma_3}$. Similarly, from $\mu_3\delta_Q=\delta_Q\mu'_3$ we have that $g_{{}_{471}} \in (E_6)^{\sigma_3}$. Hence $g_{{}_{471}}$ is well-defined. With above, the proof of this proposition is completed.	
\end{proof}
\vspace{1mm}

Subsequently, we will prove the following lemma. 

\begin{lemma}\label{lemma 4.7.2}
	The group $S(U(1)\times U(1)\times U(2)\times U(2))$ is isomorphic to the group $(U(1) \times U(1)\times  U(1)\times SU(2)\times SU(2))/(\Z_2\times\Z_2\times \Z_2)${\rm :} $S(U(1)\times U(1)\times U(2)\times U(2)) \cong (U(1) \times U(1)\times U(1)\times SU(2)\times SU(2)\times SU(2))/(\Z_2\times\Z_2\times  \Z_2), \Z_2=\{(1,1,1,E,E), (1,-1,1,E,-E) \}, \Z_2=\{(1,1,1,E,E), (1,-1,-1,-E,E) \}, \Z_2=\{ (1,1,1,E,E), (-1,1,1,E,-E)\}$.
\end{lemma}
\begin{proof}
	We define a mapping $f_{{}_{472}}:U(1)\times U(1) \times U(1)\times SU(2)\times SU(2) \to S(U(1)\times U(1)\times U(2)\times U(2))$ by
	\begin{align*}
	f_{{}_{472}}(a,b,c,A,B)=\left( 
	\begin{array}{cccc}
	a^{-2} && &{\raisebox{-7pt}[0pt]{\large $0$}}
	\\[2mm]
	& b^{-2} &&
	\\[2mm]
	&& c^{-1}\mbox{\large {$A$}}&
	\\[2mm]
	{\raisebox{3pt}[0pt]{\large $0$}}&&&(abc)\mbox{\large {$B$}}
	\end{array}\right) \in SU(6).
	\end{align*}
	Then it is clear that $f_{{}_{472}}$ is well-defined and a homomorphism. 
	
	Now, we will prove that $f_{{}_{472}}$ is surjective. Let $P \in S(U(1)\times U(1)\times U(2)\times U(2))$. Then $P$ takes the form of $\diag(s,t,P_1,P_2),s,t \in U(1),P_j \in U(2), (st)(\det\,P_1)(\det\,P_2)=1$. Here, first we choose $a \in C$ such that $s=a^{-2}$. Then it is clear that $a \in U(1)$, so is $b \in C$ such that $t=b^{-2}$, that is, $b \in U(1)$. 
	Moreover, since $P_1 \in U(2)$, we see that $\det\,P_1 \in U(1)$, and so we choose $c \in U(1)$ such that $c^2=\det\,P_1$. Set $A=c^{-1}P_1$, then we have that $A \in SU(2)$. Similarly, for $P_2 \in U(2)$, set $B=(stc)P_2$. Since $stc=(\det\,P_2)^{-1}$, we have that $B \in SU(2)$. With above, the proof of surjective is completed.
	
	Finally, we will determine $\Ker\,f_{{}_{472}}$. It follows from the kernel of definition that
	\begin{align*}
	\Ker\,f_{{}_{472}}&=\{(a,b,c,A,B)\in U(1)^{\times 3}\times SU(2)^{\times 2} \,|\,f_{{}_{472}}(a,b,c,A,B)=E \}
	\\
	&=\{(a,b,c,A,B)\in U(1)^{\times 3}\times SU(2)^{\times 2} \,|\,a^2=b^2=1,A=cE, B=(abc)^{-1}E \}
	\\
	&=\{(1,1,1,E,E), (1,1,-1,-E,-E),(1,-1,1,E,-E), (1,-1,-1,-E,E) \}
	\\
	& \quad \cup \{ (-1,1,1,E,-E), (-1,1,-1,-E,E),(-1,-1,1,E,E), (-1,-1,-1,-E,-E)\}
	\\
	&=\{(1,1,1,E,E), (1,-1,1,E,-E) \}\times \{(1,1,1,E,E), (1,-1,-1,-E,E) \}
	\\
	&\quad \times \{(1,1,1,E,E), (-1,1,1,E,-E) \}
	\\
	& \cong \Z_2 \times \Z_2 \times\Z_2.
	\end{align*}
	
	Therefore we have the required isomorphism 
	\begin{align*}
	&\quad S(U(1)\times U(1)\times U(2)\times U(2)) 
	\\
	&\cong (U(1)\times U(1) \times U(1)\times SU(2)\times SU(2))/(\Z_2\times\Z_2\times \Z_2).
	\end{align*}
\end{proof}

Now, we will determine the structure of the group $(E_6)^{\gamma_3} \cap (E_6)^{\mu_3}$.

\begin{theorem}\label{theorem 4.7.3}
	The group $(E_6)^{\gamma_3} \cap (E_6)^{\mu_3}$ is isomorphic the group $(U(1)\times U(1) \times U(1)\allowbreak \times U(1) \times SU(2)\times SU(2))/(\Z_2\times\Z_2\times\Z_4)${\rm :} $(E_6)^{\gamma_3} \cap (E_6)^{\mu_3} \cong (U(1)\times U(1) \times U(1)\times U(1) \allowbreak \times SU(2)\times SU(2))/(\Z_2\times\Z_2\times\Z_4), \Z_2=\{(1,1,1,1,E,E), (1,1,-1,1,E,-E) \},
	\Z_2=\{(1,1,1,1,E,E), (1,1,-1,-1,-E,E) \},
	\Z_4=\{(1,1,1,E,E,E), (1,-1,1-,1,E,E), (-1,i,i,\allowbreak 1,-E,E), (-1,-i,-i,1,-E,E) \}$.
\end{theorem}
\begin{proof}
	Let $S(U(1)\times U(1)\times U(2)\times U(2)) \subset SU(6)$. 
	We define a mapping $\varphi_{{}_{E_6,\gamma_3,\mu'_3}}: U(1)\times S(U(1)\times U(1)\times U(2)\times U(2)) \to (E_6)^{\gamma_3} \cap (E_6)^{\mu'_3}$ by
	\begin{align*}
	\varphi_{{}_{E_6,\gamma_3,\mu'_3}}(s, P)(M+\a)&={k_J}^{-1}(P(k_J M){}^t\!P)+s\a k^{-1}(\tau \,{}^t\!P), 
	\\
	&\hspace*{40mm}M+\a \in \mathfrak{J}(3, \H)^C \oplus (\H^3)^C=\mathfrak{J}^C.
	\end{align*}
	Needless to say, this mapping is the restriction of the mapping $\varphi_{{}_{E_6,\gamma_3}}$, that is, $\varphi_{{}_{E_6,\gamma_3,\mu'_3}}(s, P)=\varphi_{{}_{E_6,\gamma_3}}(s,P)$ (Theorem \ref{theorem 3.3.2}). 
	
	As usual, we will prove that $\varphi_{{}_{E_6,\gamma_3,\mu'_3}}$ is well-defined. It is clear that $\varphi_{{}_{E_6,\gamma_3,\mu'_3}}(s,P) \in (E_6)^{\gamma_3}$, and it follows from $\mu'_3=\varphi_{{}_{E_6,\gamma_3}}(1,\diag({\nu}^{-2},\nu^2,{\nu}^{-1},{\nu}^{-1},\nu,\nu))$ that
	\begin{align*}
	&\quad {\mu'_3}^{-1}\varphi_{{}_{E_6,\gamma_3,\nu'_3}}(s,P)\mu'_3
	\\
	&=\varphi_{{}_{E_6,\gamma_3}}(1,\diag({\nu}^{-2},\nu^2,{\nu}^{-1},{\nu}^{-1},\nu,\nu))^{-1}\varphi_{{}_{E_6,\gamma_3,\mu'_3}}(s,P)\varphi_{{}_{E_6,\gamma_3}}(1,\diag({\nu}^{-2},\nu^2,{\nu}^{-1},{\nu}^{-1},\nu,\nu))
	\\
	&=\varphi_{{}_{E_6,\gamma_3}}(1,\diag({\nu}^2,\nu^{-2},\nu,\nu,{\nu}^{-1},{\nu}^{-1}))\varphi_{{}_{E_6,\gamma_3}}(s,P)\varphi_{{}_{E_6,\gamma_3}}(1,\diag({\nu}^{-2},\nu^2,{\nu}^{-1},{\nu}^{-1},\nu,\nu))
	\\
	&=\varphi_{{}_{E_6,\gamma_3}}(s,\diag({\nu}^2,\nu^{-2},\nu,\nu,{\nu}^{-1},{\nu}^{-1})P\diag({\nu}^{-2},\nu^2,{\nu}^{-1},{\nu}^{-1},\nu,\nu)),P=\diag(a,b,P_1,P_2)
	\\
	&=\varphi_{{}_{E_6,\gamma_3}}(s,\diag(\nu^2 a \nu^{-2}, {\nu}^{-2} b \nu^2, (\nu E) P_1(\nu^{-1}E), ({\nu}^{-1}E) P_2 (\nu E) ))
	\\
	&=\varphi_{{}_{E_6,\gamma_3}}(s,P)
	\\
	&=\varphi_{{}_{E_6,\gamma_3,\mu'_3}}(s,P).
	\end{align*}
	Hence we have that $\varphi_{{}_{E_6,\gamma_3,\mu'_3}}(s,P) \in (E_6)^{\sigma'_3}$. Thus $\varphi_{{}_{E_6,\gamma_3,\mu'_3}}$ is well-defined. Subsequently, since $\varphi_{{}_{E_6,\gamma_3,\mu'_3}}$ is the restriction of the mapping $\varphi_{{}_{E_6,\gamma_3}}$, we easily see that $\varphi_{{}_{E_6,\gamma_3,\mu'_3}}$ is a homomorphism.
	
	Next, we will prove that $\varphi_{{}_{E_6,\gamma_3,\mu'_3}}$ is surjective. Let $\alpha \in (E_6)^{\gamma_3} \cap (E_6)^{\mu'_3} \subset (E_6)^{\gamma_3}$. There exist $s \in U(1)$ and $A \in SU(6)$ such that $\alpha=\varphi_{{}_{E_6,\gamma_3}}(s,A)$ (Theorem \ref{theorem 3.3.2}). Moreover, from the condition $\alpha \in (E_6)^{\mu'_3}$, that is, ${\mu'_3}^{-1}\varphi_{{}_{E_6,\gamma_3}}(s,A)\mu'_3=\varphi_{{}_{E_6,\gamma_3}}(s,A)$, and using ${\mu'_3}^{-1}\varphi_{{}_{E_6,\gamma_3}}(s,A)\mu'_3=\varphi_{{}_{E_6,\gamma_3}}(s,\diag({\nu}^2,\nu^{-2},\nu,\nu,{\nu}^{-1},{\nu}^{-1}) A \,\diag({\nu}^{-2},\nu^2,{\nu}^{-1},{\nu}^{-1},\nu,\nu))$, we have that
	\begin{align*}
	&\left\{
	\begin{array}{l}
	s=s \\
	\diag({\nu}^2,\nu^{-2},\nu,\nu,{\nu}^{-1},{\nu}^{-1}) A \diag({\nu}^{-2},\nu^2,{\nu}^{-1},{\nu}^{-1},\nu,\nu)=A   
	\end{array}\right. 
	\\
	&\hspace*{50mm}{\text{or}}
	\\
	&\left\{
	\begin{array}{l}
	s=-s \\
	\diag({\nu}^2,\nu^{-2},\nu,\nu,{\nu}^{-1},{\nu}^{-1}) A \diag({\nu}^{-2},\nu^2,{\nu}^{-1},{\nu}^{-1},\nu,\nu)=-A.   
	\end{array}\right.     
	\end{align*}
	The latter case is impossible because of $s\not=0$. As for the former case, from the second condition, by doing straightforward computation $A$ takes the following form $\diag(a, b, C, D), a,b \in U(1),C, D \in U(2),  (ab)(\det\,C)(\det\,D)=1$, that is, $A \in S(U(1)\times U(1)\times  U(2)\times U(2))$. Needless to say, $s \in U(1)$.
	Hence there exist $s \in U(1)$ and $P \in S(U(1)\times U(1)\times U(2)\times U(2))$ such that $\alpha=\varphi_{{}_{E_6,\gamma_3}}(s,P)$. Namely, there exist $s \in U(1)$ and $P \in S(U(1)\times U(1)\times U(2)\times U(2))$ such that $\alpha=\varphi_{{}_{E_6,\gamma_3,\mu'_3}}(s,P)$. With above, the proof of surjective is completed.
	
	Finally, we will determine $\Ker\,\varphi_{{}_{E_6,\gamma_3,\mu'_3}}$. However, from $\Ker\,\varphi_{{}_{E_6,\gamma_3}}=\{(1,E),(-1,-E) \}$, we easily obtain that $\Ker\,\varphi_{{}_{E_6,\gamma_3,\mu'_3}}=\{(1,E),(-1,-E) \} \cong \Z_2$. Thus we have the isomorphism $(E_6)^{\gamma_3} \cap (E_6)^{\mu'_3} \cong (U(1)\times S(U(1)\times U(1)\times U(2)\times U(2)))/\Z_2$. In addition, from Proposition \ref{proposition 4.7.1} we have the isomorphism $(E_6)^{\gamma_3} \cap (E_6)^{\mu_3} \cong (U(1)\times S(U(1)\times U(1)\times U(2)\times U(2)))/\Z_2$. Here, using the mapping $f_{{}_{472}}$ in the proof of Lemma \ref{lemma 4.7.2}, we define a homomorphism $h_{{}_{473}}:U(1)\times (U(1)\times U(1)\times U(1)\times SU(2)\times SU(2)) \to U(1)\times S(U(1)\times U(1)\times U(2)\times U(2))$ by
	\begin{align*}
			h_{{}_{473}}(s,(a,b,c,A,B))=(s,f_{{}_{472}}(a,b,c,A,B)).
	\end{align*}
	Then, the elements $(s,(a,b,c,A,B))$ corresponding to the elements $(1,E), (-1,-E) \in \Ker\,\varphi_{{}_{E_6,\gamma_3,\mu'_3}}$ under the mapping $h_{{}_{473}}$ are as follows.
	\begin{align*}
	&(1,(1,1,1,E,E)),(1,(1,1,-1,-E,-E)),(1,(1,-1,1,E,-E)),(1,(1,-1,-1,-E,E)),
	\\
	&(1,(-1,1,1,E,-E)),(1,(-1,1,-1,-E,E)),(1,(-1,-1,1,E,E)),(1,(-1,-1,-1,-E,-E)),
	\\
	&(-1,(i,i,1,-E,E)),(-1,(i,i,-1,E,-E)),(-1,(i,-i,1,-E,-E)),(-1,(i,i,-1,-E,E)),
	\\
	&(-1,(-i,i,1,-E,\!E)),(-1,(-i,i,-1,E,\!E)),(-1,(-i,-i,1,-E,E)),(-1,(-i,-i,-1,E,-E)).
	\end{align*}
	
	Therefore we have the required isomorphism 
	\begin{align*}
	(E_6)^{\gamma_3} \cap (E_6)^{\mu_3} \!\cong (U(1)\times U(1) \times U(1)\times U(1) \allowbreak \times SU(2)\times SU(2))/(\Z_2\times\Z_2\times\Z_4), 
	\end{align*}
	where 
	\begin{align*}
	&\Z_2=\{(1,1,1,1,E,E), (1,1,-1,1,E,-E) \},
	\\
	&\Z_2=\{(1,1,1,1,E,E), (1,1,-1,-1,-E,E) \},
	\\
	&\Z_4=\{(1,1,1,E,E,E), (1,-1,1-,1,E,E), (-1,i,i,1,-E,E), (-1,-i,-i,1,-E,E) \}.
	\end{align*}
\end{proof}

Thus, since the group $(E_6)^{\gamma_3} \cap (E_6)^{\mu_3}$ is connected from Theorem \ref{theorem 4.7.3}, we have an exceptional $\varmathbb{Z}_3 \times \varmathbb{Z}_3$-symmetric space
\begin{align*}
E_6/((U(1)\times U(1) \times U(1)\times U(1) \allowbreak \times SU(2)\times SU(2))/(\Z_2\times\Z_2\times\Z_4)).
\end{align*}

\subsection{Case 8: $\{1, \tilde{\gamma}_3,  \tilde{\gamma}_3{}^{-1}\} \times \{1, \tilde{w}_3,  \tilde{w}_3{}^{-1}\}$-symmetric space}

Let the $C$-linear transformations $\gamma_3, w_3$ of $\mathfrak{J}^C$  defined in Subsection \ref{subsection 3.3}. 

\noindent From Lemma \ref{lemma 3.3.8} (1), since we can easily confirm that $\gamma_3$ and $w_3$ are commutative, $\tilde{\gamma}_3$ and $\tilde{w}_3$ are commutative in $\Aut(E_6)$: $\tilde{\gamma}_3\tilde{w}_3=\tilde{w}_3\tilde{\gamma}_3$.
\vspace{1mm}

Before determining the structure of the group $(E_6)^{\gamma_3} \cap (E_6)^{w_3}$, we prove proposition and lemma needed in the proof of theorem below.
\vspace{1mm}

We define a $C$-linear transformation $w'_3$ of $\mathfrak{J}^C$ by
\begin{align*}
		w'_3=\varphi_{{}_{E_6,\gamma_3}}(1,\diag(\tau\omega,\tau\omega,\tau\omega,\omega,\omega,\omega)) \in (E_6)^{\gamma_3} \subset E_6.
\end{align*}

Let an element 
\begin{align*}N:=\scalebox{0.8}{$\begin{pmatrix}
	1&&&&&\\
	&&&&1&\\
	&&1&&&\\
	&&&1&&\\
	&-1&&&&\\
	&&&&&1
	\end{pmatrix}$} \in SO(6) \subset SU(6), 
\end{align*}
where the blanks are $0$, and we consider an element $\varphi_{{}_{E_6,\gamma_3}}(1,N) \in (E_6)^{\gamma_3} \subset E_6$. Here, we denote this element by $\delta_N$: $\delta_N=\varphi_{{}_{E_6,\gamma_3}}(1,N)$.
Then by doing straightforward computation, we have that $w_3\delta_Q=\delta_Q w'_3$, that is, $w_3$ is conjugate to $w'_3$ under $\delta_N \in (E_6)^{\gamma_3} \subset E_6$: $w_3 \sim w'_3$. Moreover, $w'_3$ induces the automorphism $\tilde{w'}_3$ of order $3$ on $E_6$: $\tilde{w'}_3(\alpha)={w'_3}^{-1}\alpha w'_3, \alpha \in E_6$.
\vspace{1mm}

Then we have the following proposition.

\begin{proposition}\label{proposition 4.8.1}
	The group $(E_6)^{\gamma_3} \cap (E_6)^{w_3}$ is isomorphic to the group $(E_6)^{\gamma_3} \cap (E_6)^{w'_3}${\rm :} $(E_6)^{\gamma_3} \cap (E_6)^{w_3} \cong (E_6)^{\gamma_3} \cap (E_6)^{w'_3}$.
\end{proposition}
\begin{proof}
	We define a mapping $g_{{}_{481}}: (E_6)^{\gamma_3} \cap (E_6)^{w'_3} \to (E_6)^{\gamma_3} \cap (E_6)^{w_3}$ by
	\begin{align*}
	g_{{}_{481}}(\alpha)=\delta_N\alpha{\delta_N}^{-1}.	
	\end{align*}
	In order to prove this isomorphism, it is sufficient to show that $g_{{}_{481}}$ is well-defined. 
	
	\noindent First, we will show that $g_{{}_{481}} \in (E_6)^{\gamma_3}$. Since it follows from $\delta_N=\varphi_{{}_{E_6,\gamma_3}}(1,N)$ and $\gamma_3=\varphi_{{}_{E_6,\gamma_3}}(\omega,E)$ that $\delta_N\gamma_3=\gamma_3\delta_N$, we have that $g_{{}_{481}} \in (E_6)^{\gamma_3}$. Similarly, from $w_3\delta_N=\delta_N w'_3$ we have that $g_{{}_{481}} \in (E_6)^{w_3}$. Hence $g_{{}_{481}}$ is well-defined. With above, the proof of this proposition is completed.	
\end{proof}

Subsequently, we will prove the following lemma. 

\begin{lemma}\label{lemma 4.8.2}
	The group $S(U(3)\times U(3))$ is isomorphic to the group $(U(1) \times SU(3)\times SU(3))/\Z_3${\rm :} $S(U(3)\times U(3)) \cong (U(1) \times SU(3)\times SU(3))/\Z_3, \Z_3=\{(1,E,E), (\omega,{\omega}^{-1} E,\omega E),\allowbreak (\omega,\omega E,{\omega}^{-1} E)\}$, where $\omega=(-1/2)+(\sqrt{3}/2)i \in C$.
\end{lemma}
\begin{proof}
	We define a mapping $f_{{}_{482}}:U(1)\times SU(3)\times SU(3) \to S(U(3)\times U(3))$ by
	\begin{align*}
	f_{{}_{482}}(a,A,B)=\left( 
	\begin{array}{cc}
	aA &{\raisebox{-5pt}[0pt]{\large $0$}}
	\\[4mm]
	{\raisebox{1pt}[0pt]{\large $0$}}& a^{-1}B
	\end{array}\right) \in SU(6).
	\end{align*}
	Then it is clear that $f_{{}_{482}}$ is well-defined and a homomorphism. 
	
	Now, we will prove that $f_{{}_{482}}$ is surjective. Let $P \in S(U(3)\times U(3))$. Then $P$ takes the form of $\diag(P_1,P_2),P_j \in U(3), (\det\,P_1)(\det\,P_2)=1$. Here, since $P_1 \in U(3)$, we see that $\det\,P_1 \in U(1)$, and so we choose $a \in U(1)$ such that $a^3=\det\,P_1$. Set $A=a^{-1}P_1$, then we have that $A \in SU(3)$. Similarly, for $P_2 \in U(2)$, set $B=aP_2$, we have that $B \in SU(3)$. With above, the proof of surjective is completed.
	
	Finally, we will determine $\Ker\,f_{{}_{482}}$. It follows from the kernel of definition that
	\begin{align*}
	\Ker\,f_{{}_{482}}&=\{(a,A,B)\in U(1)\times SU(3)\times SU(3) \,|\,f_{{}_{482}}(a,A,B)=E \}
	\\
	&=\{(a,A,B)\in U(1)\times SU(3)\times SU(3) \,|\,a^3=1,A=a^{-1}E, B=aE \}
	\\
	&=\{(1,E,E), (\omega,{\omega}^{-1}E,\omega E),({\omega}^{-1},\omega E,{\omega}^{-1}E) \}
	\\
	& \cong \Z_3.
	\end{align*}
	
	Therefore we have the required isomorphism 
	\begin{align*}
	S(U(3)\times U(3)) \cong (U(1) \times SU(3)\times SU(3))/\Z_3.
	\end{align*}
\end{proof}

Now, we will determine the structure of the group $(E_6)^{\gamma_3} \cap (E_6)^{w_3}$.

\begin{theorem}\label{theorem 4.8.3}
	The group $(E_6)^{\gamma_3} \cap (E_6)^{w_3}$ is isomorphic the group $(U(1) \times U(1) \times SU(3)\times SU(3)))/(\Z_2 \times \Z_3)${\rm :} $(E_6)^{\gamma_3} \cap (E_6)^{w_3} \cong ((U(1) \times U(1) \times SU(3)\times SU(3)))/(\Z_2 \times \Z_3), \Z_2=\{(1,1,E,E), (-1,-1,E,E)\},
	\Z_3=\{(1,1,E,E), (1,\omega,{\omega}^{-1}E,\omega E),(1,{\omega}^{-1},\omega E,{\omega}^{-1}E)\}$.
\end{theorem}
\begin{proof}
	Let $S(U(3)\times U(3)) \subset SU(6)$. 
	We define a mapping $\varphi_{{}_{E_6,\gamma_3,w'_3}}: U(1)\times S(U(3)\times U(3)) \to (E_6)^{\gamma_3} \cap (E_6)^{w'_3}$ by
	\begin{align*}
	\varphi_{{}_{E_6,\gamma_3,w'_3}}(s, P)(M+\a)&={k_J}^{-1}(P(k_J M){}^t\!P)+s\a k^{-1}(\tau \,{}^t\!P), 
	\\
	&\hspace*{40mm}M+\a \in \mathfrak{J}(3, \H)^C \oplus (\H^3)^C=\mathfrak{J}^C.
	\end{align*}
	Needless to say, this mapping is the restriction of the mapping $\varphi_{{}_{E_6,\gamma_3}}$, that is, $\varphi_{{}_{E_6,\gamma_3,w'_3}}(s, P)=\varphi_{{}_{E_6,\gamma_3}}(s,P)$ (Theorem \ref{theorem 3.3.2}). 
	
	First, we will prove that $\varphi_{{}_{E_6,\gamma_3,w'_3}}$ is well-defined. It is clear that $\varphi_{{}_{E_6,\gamma_3,w'_3}}(s,P) \in (E_6)^{\gamma_3}$, and it follows from $w'_3=\varphi_{{}_{E_6,\gamma_3}}(1,\diag(\tau\omega,\tau\omega,\tau\omega,\omega,\omega,\omega))$ that
	\begin{align*}
	&\quad {w'_3}^{-1}\varphi_{{}_{E_6,\gamma_3,\nu'_3}}(s,P) w'_3
	\\
	&=\varphi_{{}_{E_6,\gamma_3}}(1,\diag(\tau\omega,\tau\omega,\tau\omega,\omega,\omega,\omega))^{-1}\varphi_{{}_{E_6,\gamma_3,w'_3}}(s,P)\varphi_{{}_{E_6,\gamma_3}}(1,\diag(\tau\omega,\tau\omega,\tau\omega,\omega,\omega,\omega))
	\\
	&=\varphi_{{}_{E_6,\gamma_3}}(1,\diag(\omega,\omega,\omega,\tau\omega,\tau\omega,\tau\omega))\varphi_{{}_{E_6,\gamma_3}}(s,P)\varphi_{{}_{E_6,\gamma_3}}(1,\diag(\tau\omega,\tau\omega,\tau\omega,\omega,\omega,\omega))
	\\
	&=\varphi_{{}_{E_6,\gamma_3}}(s,\diag(\omega,\omega,\omega,\tau\omega,\tau\omega,\tau\omega)P\,\diag(\tau\omega,\tau\omega,\tau\omega,\omega,\omega,\omega)),P=\diag(P_1,P_2)
	\\
	&=\varphi_{{}_{E_6,\gamma_3}}(s,\diag((\omega E)P_1(\tau\omega E), \tau(\omega E) P_2 (\omega E) ))
	\\
	&=\varphi_{{}_{E_6,\gamma_3}}(s,P)
	\\
	&=\varphi_{{}_{E_6,\gamma_3,w'_3}}(s,P).
	\end{align*}
	Hence we have that $\varphi_{{}_{E_6,\gamma_3,w'_3}}(s,P) \in (E_6)^{w'_3}$. Thus $\varphi_{{}_{E_6,\gamma_3,w'_3}}$ is well-defined. Subsequently, since $\varphi_{{}_{E_6,\gamma_3,w'_3}}$ is the restriction of the mapping $\varphi_{{}_{E_6,\gamma_3}}$, we easily see that $\varphi_{{}_{E_6,\gamma_3,w'_3}}$ is a homomorphism.
	
	Next, we will prove that $\varphi_{{}_{E_6,\gamma_3,w'_3}}$ is surjective. Let $\alpha \in (E_6)^{\gamma_3} \cap (E_6)^{w'_3} \subset (E_6)^{\gamma_3}$. There exist $s \in U(1)$ and $A \in SU(6)$ such that $\alpha=\varphi_{{}_{E_6,\gamma_3}}(s,A)$ (Theorem \ref{theorem 3.3.2}). Moreover, from the condition $\alpha \in (E_6)^{w'_3}$, that is, ${w'_3}^{-1}\varphi_{{}_{E_6,\gamma_3}}(s,A)w'_3=\varphi_{{}_{E_6,\gamma_3}}(s,A)$, and using ${w'_3}^{-1}\varphi_{{}_{E_6,\gamma_3}}(s,A)w'_3=\varphi_{{}_{E_6,\gamma_3}}(s,\diag(\omega,\omega,\omega,\tau\omega,\tau\omega,\tau\omega)A\,\diag(\tau\omega,\tau\omega,\tau\omega,\omega,\omega,\omega))$, we have that
	\begin{align*}
	&\left\{
	\begin{array}{l}
	s=s \\
	\diag(\omega,\omega,\omega,\tau\omega,\tau\omega,\tau\omega)A\,\diag(\tau\omega,\tau\omega,\tau\omega,\omega,\omega,\omega)=A   
	\end{array}\right. 
	\\
	&\hspace*{50mm}{\text{or}}
	\\
	&\left\{
	\begin{array}{l}
	s=-s \\
	\diag(\omega,\omega,\omega,\tau\omega,\tau\omega,\tau\omega)A\,\diag(\tau\omega,\tau\omega,\tau\omega,\omega,\omega,\omega)=-A.   
	\end{array}\right.     
	\end{align*}
	The latter case is impossible because of $s\not=0$. As for the former case, from the second condition, by doing straightforward computation $A$ takes the following form $\diag(C, D), C, D \in U(2), (\det\,C)(\det\,D)=1$, that is, $A \in S(U(3)\times U(3))$. Needless to say, $s \in U(1)$.
	Hence there exist $s \in U(1)$ and $P \in S(U(3)\times U(3))$ such that $\alpha=\varphi_{{}_{E_6,\gamma_3}}(s,P)$. Namely, there exist $s \in U(1)$ and $P \in S(U(3)\times U(3))$ such that $\alpha=\varphi_{{}_{E_6,\gamma_3,w'_3}}(s,P)$. The proof of surjective is completed.
	
	Finally, we will determine $\Ker\,\varphi_{{}_{E_6,\gamma_3,w'_3}}$. However, from $\Ker\,\varphi_{{}_{E_6,\gamma_3}}=\{(1,E),(-1,-E) \}$, we easily obtain that $\Ker\,\varphi_{{}_{E_6,\gamma_3,w'_3}}=\{(1,E),(-1,-E) \} \cong \Z_2$. Thus we have the isomorphism $(E_6)^{\gamma_3} \cap (E_6)^{w'_3} \cong (U(1)\times S(U(3)\times U(3)))/\Z_2$. In addition, from Proposition \ref{proposition 4.8.1} we have the isomorphism $(E_6)^{\gamma_3} \cap (E_6)^{w_3} \cong (U(1)\times S(U(3)\times U(3)))/\Z_2$. Here, using the mapping $f_{{}_{482}}$ in the proof of Lemma \ref{lemma 4.8.2}, we define a homomorphism $h_{{}_{484}}:U(1)\times (U(1)\times SU(3)\times SU(3)) \to U(1)\times S(U(3)\times U(3))$ by
	\begin{align*}
	h_{{}_{483}}(s,(a,A,B))=(s,f_{{}_{482}}(a,A,B)).
	\end{align*}
	Then, the elements $(s,(a,A,B))$ corresponding to the elements $(1,E), (-1,-E) \in \\
	\Ker\,\varphi_{{}_{E_6,\gamma_3,w'_3}}$ under the mapping $h_{{}_{483}}$ are as follows.
	\begin{align*}
	& (1,1,E,E), (1,\omega,{\omega}^{-1}E,\omega E),(1,{\omega}^{-1},\omega E,{\omega}^{-1}E) 
	\\
	& (-1,-1,E,E), (-1,-\omega,{\omega}^{-1}E,\omega E),(-1,-{\omega}^{-1},\omega E,{\omega}^{-1}E).
	\end{align*}
	Therefore we have the required isomorphism
	\begin{align*}
	  (E_6)^{\gamma_3} \cap (E_6)^{w_3} \cong (U(1) \times U(1) \times SU(3)\times SU(3))/(\Z_2 \times \Z_3), 
	 \end{align*}
	 where
	 \begin{align*}
	   &\Z_2=\{(1,1,E,E), (-1,-1,E,E)\},
	   \\
	   &\Z_3=\{(1,1,E,E), (1,\omega,{\omega}^{-1}E,\omega E),(1,{\omega}^{-1},\omega E,{\omega}^{-1}E)\}.
	 \end{align*}
\end{proof}

Thus, since the group $(E_6)^{\gamma_3} \cap (E_6)^{w_3}$ is connected from Theorem \ref{theorem 4.8.3}, we have an exceptional $\varmathbb{Z}_3 \times \varmathbb{Z}_3$-symmetric space
\begin{align*}
E_6/((U(1) \times U(1) \times SU(3)\times SU(3))/(\Z_2 \times \Z_3)).
\end{align*}

\subsection{Case 9: $\{1, \tilde{\sigma}_3,  \tilde{\sigma}_3{}^{-1}\} \times \{1, \tilde{\nu}_3,  \tilde{\nu}_3{}^{-1}\}$-symmetric space}

Let the $C$-linear transformations $\sigma_3, \nu_3$ of $\mathfrak{J}^C$  defined in Subsection \ref{subsection 3.3}. 

\noindent From Lemma \ref{lemma 3.3.8} (1), since we can easily confirm that $\sigma_3$ and $\nu_3$ are commutative, $\tilde{\sigma}_3$ and $\tilde{\nu}_3$ are commutative in $\Aut(E_6)$: $\tilde{\sigma}_3\tilde{\nu}_3=\tilde{\nu}_3\tilde{\sigma}_3$.

Before determining the structure of the group $(E_6)^{\sigma_3} \cap (E_6)^{\nu_3}$, we confirm that useful lemma holds and prove proposition needed in the proof of theorem below.

\begin{lemma}\label{lemma 4.9.1}
	The mapping $\varphi_{{}_{E_6,\nu_3}}:Sp(1) \times S(U(1)\times U(5)) \to (E_6)^{\nu_3}$ of \,Theorem {\rm \ref{theorem 3.3.5}} satisfies the relational formulas 
	\begin{align*}
	\sigma_3&=\varphi_{{}_{E_6,\nu_3}}(1, \diag(1,1,\tau\omega,\omega,\omega,\tau\omega)),
	\\
	\nu_3&=\varphi_{{}_{E_6,\nu_3}}(1,\diag(\nu^5,\nu^{-1},\nu^{-1},\nu^{-1},\nu^{-1},\nu^{-1})),
	\end{align*}
	where ${\omega}=-(1/2)+(\sqrt{3}/2)i \in U(1)$.
\end{lemma}
\begin{proof}
	From Lemma \ref{lemma 3.3.8} (1), these results are trivial. 
\end{proof}

The $C$-linear transformation $\sigma'_3$ defined in the Case 5 is expressed by
\begin{align*}
\sigma'_3=\varphi_{{}_{E_6,\nu_3}}(1, \diag(1,1,\omega,\omega,\tau\omega,\tau\omega)),
\end{align*}
and note that $\delta_R=\varphi_{{}_{E_6, \nu_3}}(1,R)(=\varphi_{{}_{E_6, \gamma_3}}(1,R))$, where $\delta_R$ is also defined in the Case 5, moreover needless to say, $\sigma_3$ is conjugate to $\sigma'_3$ under $\delta_R=\varphi_{{}_{E_6, \nu_3}}(1,R)$.

\begin{proposition}\label{proposition 4.9.2}
	The group $(E_6)^{\sigma_3} \cap (E_6)^{\nu_3}$ is isomorphic to the group $(E_6)^{\sigma'_3} \cap (E_6)^{\nu_3}${\rm :} $(E_6)^{\sigma_3} \cap (E_6)^{\nu_3} \cong (E_6)^{\sigma'_3} \cap (E_6)^{\nu_3}$.
\end{proposition}
\begin{proof}
	We define a mapping $g_{{}_{492}}: (E_6)^{\sigma_3} \cap (E_6)^{\nu_3} \to (E_6)^{\sigma'_3} \cap (E_6)^{\nu_3}$ by
	\begin{align*}
	g_{{}_{492}}(\alpha)={\delta_R}^{-1}\alpha\delta_R,
	\end{align*}
	where $\delta_R$ is same one above. Since it is easy to verify that $\delta_R\nu_3=\nu_3\delta_R$ using $\nu_3=\varphi_{{}_{E_6,\nu_3}}(1,\diag(\nu^5,\nu^{-1},\nu^{-1},\nu^{-1},\nu^{-1},\nu^{-1}))$ (Lemma \ref{lemma 4.9.1}), we can prove this proposition as in the proof of Proposition \ref{proposition 4.5.1}
\end{proof}

Now, we will determine the structure of the group $(E_6)^{\sigma_3} \cap (E_6)^{\nu_3}$.

\begin{theorem}\label{theorem 4.9.3}
	The group $(E_6)^{\sigma_3} \cap (E_6)^{\nu_3}$ is isomorphic the group $(Sp(1)\times U(1) \times U(1)\allowbreak \times U(1) \times SU(2)\times SU(2))/(\Z_2\times\Z_2\times\Z_4)${\rm :} $(E_6)^{\sigma_3} \cap (E_6)^{\nu_3} \cong (Sp(1)\times U(1) \times U(1)\times U(1) \allowbreak \times SU(2)\times SU(2))/(\Z_2\times\Z_2\times\Z_4), \Z_2=\{(1,1,1,1,E,E), (1,1,-1,1,E,-E) \},
	\Z_2=\{(1,1,1,1,E,E), (1,1,-1,-1,-E,E) \},
	\Z_4=\{(1,1,1,E,E,E), (1,-1,1-,1,E,E), (-1,i,i,\allowbreak 1,-E,E), (-1,-i,-i,1,-E,E) \}$.
\end{theorem}
\begin{proof}
	Let $S(U(1)\times U(1)\times U(2)\times U(2)) \subset S(U(1)\times U(5))$ as in the proof of Theorem \ref{theorem 4.7.3}. 
	We define a mapping $\varphi_{{}_{E_6,\sigma'_3,\nu_3}}: Sp(1)\times S(U(1)\times U(1)\times U(2)\times U(2)) \to (E_6)^{\sigma'_3} \cap (E_6)^{\nu_3}$ by
	\begin{align*}
	\varphi_{{}_{E_6,\sigma'_3,\nu_3}}(q, P)(M+\a)&={k_J}^{-1}(P(k_J M){}^t\!P)+q\a k^{-1}(\tau \,{}^t\!P), 
	\\
	&\hspace*{40mm}M+\a \in \mathfrak{J}(3, \H)^C \oplus (\H^3)^C=\mathfrak{J}^C.
	\end{align*}
	Needless to say, this mapping is the restriction of the mapping $\varphi_{{}_{E_6,\nu_3}}$, that is, $\varphi_{{}_{E_6,\sigma'_3,\nu_3}}(q, P)=\varphi_{{}_{E_6,\nu_3}}(q,P)$ (Theorem \ref{theorem 3.3.5}). 
	
	As usual, we will prove that $\varphi_{{}_{E_6,\sigma'_3,\nu_3}}$ is well-defined. It is clear that $\varphi_{{}_{E_6,\sigma'_3,\nu_3}}(q,P) \in (E_6)^{\nu_3}$, and it follows from $\sigma'_3=\varphi_{{}_{E_6,\nu_3}}(1,\diag(1,1,\omega,\omega,\tau\omega,\tau\omega))$ that
	\begin{align*}
	&\quad {\sigma'_3}^{-1}\varphi_{{}_{E_6,\sigma'_3,\nu_3}}(q,P)\sigma'_3
	\\
	&=\varphi_{{}_{E_6,\nu_3}}(1,\diag(1,1,\omega,\omega,\tau\omega,\tau\omega))^{-1}\varphi_{{}_{E_6,\sigma'_3,\nu_3}}(q,P)\varphi_{{}_{E_6,\nu_3}}(1,\diag(1,1,\omega,\omega,\tau\omega,\tau\omega))
	\\
	&=\varphi_{{}_{E_6,\nu_3}}(1,\diag(1,1,\tau\omega,\tau\omega,\omega,\omega))\varphi_{{}_{E_6,\nu_3}}(q,P)\varphi_{{}_{E_6,\nu_3}}(1,\diag(1,1,\omega,\omega,\tau\omega,\tau\omega))
	\\
	&=\varphi_{{}_{E_6,\nu_3}}(q,\diag(1,1,\tau\omega,\tau\omega,\omega,\omega)P\,\diag(1,1,\omega,\omega,\tau\omega,\tau\omega)),P=\diag(a,b,P_1,P_2)
	\\
	&=\varphi_{{}_{E_6,\nu_3}}(q,\diag(a, b, (\tau\omega E)P_1(\omega E), (\omega E) P_2 (\tau\omega E) ))
	\\
	&=\varphi_{{}_{E_6,\nu_3}}(q,P)
	\\
	&=\varphi_{{}_{E_6,\sigma'_3,\nu_3}}(q,P).
	\end{align*}
	Hence we have that $\varphi_{{}_{E_6,\sigma'_3,\nu_3}}(q,P) \in (E_6)^{\sigma'_3}$. Thus $\varphi_{{}_{E_6,\sigma'_3,\nu_3}}$ is well-defined. Subsequently, since $\varphi_{{}_{E_6,\sigma'_3,\nu_3}}$ is the restriction of the mapping $\varphi_{{}_{E_6,\nu_3}}$, we easily see that $\varphi_{{}_{E_6,\sigma'_3,\nu_3}}$ is a homomorphism.
	
	Next, we will prove that $\varphi_{{}_{E_6,\sigma'_3,\nu_3}}$ is surjective. Let $\alpha \in (E_6)^{\sigma'_3} \cap (E_6)^{\nu_3} \subset (E_6)^{\nu_3}$. There exist $q \in Sp(1)$ and $A \in S(U(1)\times U(5))$ such that $\alpha=\varphi_{{}_{E_6,\nu_3}}(q,A)$ (Theorem \ref{theorem 3.3.5}). Moreover, from the condition $\alpha \in (E_6)^{\sigma'_3}$, that is, ${\sigma'_3}^{-1}\varphi_{{}_{E_6,\nu_3}}(q,A)\sigma'_3=\varphi_{{}_{E_6,\nu_3}}(q,A)$, and using ${\sigma'_3}^{-1}\varphi_{{}_{E_6,\nu_3}}(q,A)\sigma'_3=\varphi_{{}_{E_6,\nu_3}}(q,\diag(1,1,\tau\omega,\tau\omega,\omega,\omega)A\,\diag(1,1,\omega,\omega,\tau\omega,\tau\omega))$, we have that
	\begin{align*}
	&\left\{
	\begin{array}{l}
	q=q \\
	\diag(1,1,\tau\omega,\tau\omega,\omega,\omega)A\,\diag(1,1,\omega,\omega,\tau\omega,\tau\omega)=A   
	\end{array}\right. 
	\\
	&\hspace*{50mm}{\text{or}}
	\\
	&\left\{
	\begin{array}{l}
	q=-q \\
	\diag(1,1,\tau\omega,\tau\omega,\omega,\omega)A\,\diag(1,1,\omega,\omega,\tau\omega,\tau\omega)=-A.   
	\end{array}\right.     
	\end{align*}
	The latter case is impossible because of $q\not=0$. As for the former case, from the second condition, by doing straightforward computation $A$ takes the following form $\diag(a, b, C, D), a,b \in U(1),C, D \in U(2),  (ab)(\det\,C)(\det\,D)=1$, that is, $A \in S(U(1)\times U(1)\times  U(2)\times U(2))$. Needless to say, $q \in Sp(1)$.
	Hence there exist $q \in Sp(1)$ and $P \in S(U(1)\times U(1)\times U(2)\times U(2))$ such that $\alpha=\varphi_{{}_{E_6,\nu_3}}(s,P)$. Namely, there exist $q \in Sp(1)$ and $P \in S(U(1)\times U(1)\times U(2)\times U(2))$ such that $\alpha=\varphi_{{}_{E_6,\sigma'_3,\nu_3}}(s,P)$. The proof of surjective is completed.
	
	Finally, we will determine $\Ker\,\varphi_{{}_{E_6,\sigma'_3,\nu_3}}$. However, from $\Ker\,\varphi_{{}_{E_6,\nu_3}}=\{(1,E),(-1,-E) \}$, we easily obtain that $\Ker\,\varphi_{{}_{E_6,\sigma'_3,\nu_3}}=\{(1,E),(-1,-E) \} \cong \Z_2$. Thus we have the isomorphism $(E_6)^{\sigma'_3} \cap (E_6)^{\nu_3} \cong (Sp(1)\times S(U(1)\times U(1)\times U(2)\times U(2)))/\Z_2$. In addition, from Proposition \ref{proposition 4.8.1} we have the isomorphism $(E_6)^{\sigma_3} \cap (E_6)^{\nu_3} \cong (Sp(1)\times S(U(1)\times U(1)\times U(2)\times U(2)))/\Z_2$. 
	Therefore, as in the proof of Theorem \ref{theorem 4.7.3}, we have the required isomorphism 
	\begin{align*}
	(E_6)^{\sigma_3} \cap (E_6)^{\nu_3} \!\cong (Sp(1)\times U(1) \times U(1)\times U(1) \allowbreak \times SU(2)\times SU(2))/(\Z_2\times\Z_2\times\Z_4), 
	\end{align*}
	where 
	\begin{align*}
	&\Z_2=\{(1,1,1,1,E,E), (1,1,-1,1,E,-E) \},
	\\
	&\Z_2=\{(1,1,1,1,E,E), (1,1,-1,-1,-E,E) \},
	\\
	&\Z_4=\{(1,1,1,E,E,E), (1,-1,1-,1,E,E), (-1,i,i,1,-E,E), (-1,-i,-i,1,-E,E) \}.
	\end{align*}
\end{proof}

Thus, since the group $(E_6)^{\sigma_3} \cap (E_6)^{\nu_3}$ is connected from Theorem \ref{theorem 4.9.3}, we have an exceptional $\varmathbb{Z}_3 \times \varmathbb{Z}_3$-symmetric space
\begin{align*}
E_6/((Sp(1)\times U(1) \times U(1)\times U(1) \times SU(2)\times SU(2))/(\Z_2\times\Z_2\times\Z_4)).
\end{align*}

\subsection{Case 10: $\{1, \tilde{\sigma}_3,  \tilde{\sigma}_3{}^{-1}\} \times \{1, \tilde{\mu}_3,  \tilde{\mu}_3{}^{-1}\}$-symmetric space}

Let the $C$-linear transformations $\sigma_3, \mu_3$ of $\mathfrak{J}^C$  defined in Subsection \ref{subsection 3.3}. 

\noindent From Lemma \ref{lemma 3.3.8} (1), since we can easily confirm that $\sigma_3$ and $\mu_3$ are commutative, $\tilde{\sigma}_3$ and $\tilde{\mu}_3$ are commutative in $\Aut(E_6)$: $\tilde{\sigma}_3\tilde{\mu}_3=\tilde{\mu}_3\tilde{\sigma}_3$.

Before determining the structure of the group $(E_6)^{\sigma_3} \cap (E_6)^{\mu_3}$, we prove proposition needed in the proof of theorem below.

\begin{proposition}\label{proposition 4.10.1}
	The group $(E_6)^{\sigma_3}$ is a subgroup of the group $(E_6)^\sigma${\rm: } $(E_6)^{\sigma_3} \subset (E_6)^\sigma$.
\end{proposition}
\begin{proof}
	Let $\alpha \in (E_6)^{\sigma_3}$. Then, from Theorem \ref{theorem 3.3.4}, there exist $\theta \in U(1), D_a \in Spin(2)$ and $\beta \in Spin(8)$ such that $\alpha=\phi_{{}_{6,\sigma}}(\theta) D_a \beta$. Here, note that $(E_6)_{E_1} \subset (E_6)^\sigma$ (\cite[Theorem 3.10.2]{iy0}), and so since $Spin(8)$ as the group $(E_6)_{E_1,F_1(1),F_1(e_1)} \subset (E_6)_{E_1} \subset (E_6)^\sigma$, it follows that
	\begin{align*}
			\sigma\alpha=\sigma(\phi_{{}_{6,\sigma}}(\theta) D_a \beta)=\phi_{{}_{6,\sigma}}(\theta)\sigma D_a\beta=\phi_{{}_{6,\sigma}}(\theta)D_a \sigma\beta=(\phi_{{}_{6,\sigma}}(\theta) D_a \beta)\sigma=\alpha\sigma.
	\end{align*}
	Hence we have that $\alpha \in (E_6)^\sigma$, that is, $(E_6)^{\sigma_3} \subset (E_6)^\sigma$.
\end{proof}

Now, we will determine the structure of the group $(E_6)^{\sigma_3} \cap (E_6)^{\mu_3}$.

\begin{theorem}\label{theorem 4.10.2}
	The group $(E_6)^{\sigma_3} \cap (E_6)^{\mu_3}$ coincides with the group $(E_6)^{\sigma_3}$, that is, the group $(E_6)^{\sigma_3} \cap (E_6)^{\mu_3}$ is isomorphic to the group $(U(1)\times Spin(2)\times Spin(8))/(\Z_2 \times \Z_4),\\ \Z_2=\!\{(1,1,1),(1,\sigma,\sigma) \}, \Z_4\!=\!\{(1,1,1),(i,D_{e_1},\phi_{{}_{6,\sigma}}(-i)D_{-e_1}),(-1,\allowbreak \sigma,-1),(-i,D_{-e_1}, \phi_{{}_{6,\sigma}}(i) \allowbreak D_{e_1}) \}$. 
\end{theorem}
\begin{proof}
	From Proposition \ref{proposition 3.3.3} and Theorem \ref{theorem 3.3.6}, we have that the group $(E_6)^{\sigma_3} \cap (E_6)^{\mu_3}$ coincides with the group $(E_6)^{\sigma_3} \cap (E_6)^{\sigma}$: $(E_6)^{\sigma_3} \cap (E_6)^{\mu_3}=(E_6)^{\sigma_3} \cap (E_6)^{\sigma}$. In addition, from Proposition \ref{proposition 4.10.1} above, we have that 
	\begin{align*}
	      (E_6)^{\sigma_3} \cap (E_6)^{\mu_3}=(E_6)^{\sigma_3} \cap (E_6)^{\sigma}=(E_6)^{\sigma_3}.
	\end{align*}
	Therefore, by Theorem \ref{theorem 3.3.4}, we have the required isomorphism 
	\begin{align*}
	 		   (E_6)^{\sigma_3} \cap (E_6)^{\mu_3} \cong (U(1)\times Spin(2)\times Spin(8))/(\Z_2\times \Z_4).
	\end{align*}
\end{proof}

Thus, since the group $(E_6)^{\sigma_3} \cap (E_6)^{\mu_3}$ is connected from Theorem \ref{theorem 4.10.2}, we have an exceptional $\varmathbb{Z}_3 \times \varmathbb{Z}_3$-symmetric space
\begin{align*}
E_6/((U(1)\times Spin(2)\times Spin(8))/(\Z_2\times \Z_4)).
\end{align*}

\subsection{Case 11: $\{1, \tilde{\sigma}_3,  \tilde{\sigma}_3{}^{-1}\} \times \{1, \tilde{w}_3,  \tilde{w}_3{}^{-1}\}$-symmetric space}

Let the $C$-linear transformations $\sigma_3, w_3$ of $\mathfrak{J}^C$  defined in Subsection \ref{subsection 3.3}. 

\noindent From Lemma \ref{lemma 3.3.8} (2), since we can easily confirm that $\sigma_3$ and $w_3$ are commutative, $\tilde{\sigma}_3$ and $\tilde{w}_3$ are commutative in $\Aut(E_6)$: $\tilde{\sigma}_3\tilde{w}_3=\tilde{w}_3\tilde{\sigma}_3$.

Now, we will determine the structure of the group $(E_6)^{\sigma_3}\cap (E_6)^{w_3}$.

\begin{theorem}\label{theorem 4.11.1}
	The group $(E_6)^{\sigma_3}\cap (E_6)^{w_3}$ is isomorphic to the group $(SU(3)\times U(1)\times U(1)\times U(1)\times U(1))/\Z_3${\rm :} $(E_6)^{\sigma_3}\cap (E_6)^{w_3} \cong (SU(3)\times U(1)\times U(1)\times U(1)\times U(1))/\Z_3, \Z_3=\{(E,1,1,1,1),(\bm{\omega}E,\bm{\omega},\bm{\omega},\bm{\omega},\bm{\omega}),(\bm{\omega}^{-1}E,\bm{\omega}^{-1},\bm{\omega}^{-1},\bm{\omega}^{-1},\bm{\omega}^{-1})\}$.
\end{theorem}
\begin{proof}
	Let $S(U(1)\times U(1)\times U(1)) \subset SU(3)$. We define a mapping $\varphi_{{}_{E_6,\sigma_3,w_3}}: SU(3)\times S(U(1)\times U(1)\times U(1))\times S(U(1)\times U(1)\times U(1)) \to (E_6)^{\sigma_3}\cap (E_6)^{w_3}$ by
	\begin{align*}
			\varphi_{{}_{E_6,\sigma_3,w_3}}(L,P,Q)(X_{C}+M)&=h(P,Q)X_{C}h(P,Q)^*+LM\tau h(P,Q)^*, 
			\\
			&\hspace*{20mm} X_{C}+M \in \mathfrak{J}(3, \C)^C \oplus 
			M(3,\C)^C=\mathfrak{J}^C.
	\end{align*}
	Needless to say, this mapping is the restriction of the mapping $\varphi_{{}_{E_6,w_3}}$, that is, $\varphi_{{}_{E_6,\sigma_3,w_3}}(L,P,\allowbreak Q)=\varphi_{{}_{E_6,w_3}}(L,P,Q)$ (Theorem \ref{theorem 3.3.7}). 
	
	We will prove that $\varphi_{{}_{E_6,\sigma_3,w_3}}$ is well-defined. It is clear that $\varphi_{{}_{E_6,\sigma_3,_3}}(L,P,Q) \in (E_6)^{w_3}$, and it follows from $\sigma_3=\varphi_{{}_{E_6,w_3}}(E,\diag(1,\ov{\bm{\omega}},\bm{\omega}), \diag(1,\ov{\bm{\omega}},\bm{\omega}))$ (Lemma \ref{lemma 3.3.8} (2)) that 
	\begin{align*}
	&\quad {\sigma_3}^{-1}\varphi_{{}_{E_6,\sigma_3,w_3}}(L,P,Q)\sigma_3
	\\
	&=\varphi_{{}_{E_6,w_3}}(E,\diag(1,\ov{\bm{\omega}},\bm{\omega}), \diag(1,\ov{\bm{\omega}},\bm{\omega}))^{-1}\varphi_{{}_{E_6,\sigma_3,w_3}}(L,P,Q)
	\\
	&\hspace*{70mm}\varphi_{{}_{E_6,w_3}}(E,\diag(1,\ov{\bm{\omega}},\bm{\omega}), \diag(1,\ov{\bm{\omega}},\bm{\omega}))
	\\
	&=\varphi_{{}_{E_6,w_3}}(E,\diag(1,\bm{\omega},\ov{\bm{\omega}}), \diag(1,\bm{\omega},\ov{\bm{\omega}}))\varphi_{{}_{E_6,w_3}}(L,P,Q)
	\\
	&\hspace*{70mm}\varphi_{{}_{E_6,w_3}}(E,\diag(1,\ov{\bm{\omega}},\bm{\omega}), \diag(1,\ov{\bm{\omega}},\bm{\omega}))
	\\
	&=\varphi_{{}_{E_6,w_3}}(L,\diag(1,\bm{\omega},\ov{\bm{\omega}})P\diag(1,\ov{\bm{\omega}},\bm{\omega}),\diag(1,\bm{\omega},\ov{\bm{\omega}})Q\diag(1,\ov{\bm{\omega}},\bm{\omega})),
	\\
	&\hspace*{85mm}P=\diag(a,b,c), Q=\diag(s,t,v)
	\\
	&=\varphi_{{}_{E_6,w_3}}(L,P,Q)
	\\
	&=\varphi_{{}_{E_6,\sigma_3,w_3}}(L,P,Q).
	\end{align*}
	Hence we have that $\varphi_{{}_{E_6,\sigma_3,w_3}}(L,P,Q) \in (E_6)^{\sigma_3}$. Thus $\varphi_{{}_{E_6,\sigma_3,w_3}}$ is well-defined. Subsequently, since $\varphi_{{}_{E_6,\sigma_3,w_3}}$ is the restriction of the mapping $\varphi_{{}_{E_6,w_3}}$, we easily see that $\varphi_{{}_{E_6,\sigma_3,w_3}}$ is a homomorphism.
	
	Next we will prove that $\varphi_{{}_{E_6,\sigma_3,w_3}}$ is surjective. Let $\alpha \in (E_6)^{\sigma_3}\cap (E_6)^{w_3} \subset (E_6)^{w_3}$. There exist $L, A, B \in SU(3)$ such that $\alpha=\varphi_{{}_{E_6,w_3}}(L,A,B)$ (Theorem \ref{theorem 3.3.7}). Moreover, from the condition $\alpha \in (E_6)^{\sigma_3}$, that is, ${\sigma_3}^{-1}\varphi_{{}_{E_6,w_3}}(L,A,B)\sigma_3=\varphi_{{}_{E_6,w_3}}(L,A,B)$, and using 
	\begin{align*}
				&\quad {\sigma_3}^{-1}\varphi_{{}_{E_6,w_3}}(L,A,B)\sigma_3
				\\
				&=\varphi_{{}_{E_6,w_3}}(L,\diag(1,\bm{\omega},\ov{\bm{\omega}})A\diag(1,\ov{\bm{\omega}},\bm{\omega}),\diag(1,\bm{\omega},\ov{\bm{\omega}})B\diag(1,\ov{\bm{\omega}}, \bm{\omega}))
	\end{align*}
	(Lemma \ref{lemma 3.3.8} (2)) we have that 
	\begin{align*}
	&\,\,\,{\rm(i)}\,\left\{
	\begin{array}{l}
	L=L\\
	\diag(1,\bm{\omega},\ov{\bm{\omega}})A\diag(1,\ov{\bm{\omega}},\bm{\omega})=A \\
	\diag(1,\bm{\omega},\ov{\bm{\omega}})B\diag(1,\ov{\bm{\omega}}, \bm{\omega})=B,
	\end{array} \right.
	\qquad
	{\rm(ii)}\,\left\{
	\begin{array}{l}
	L=\bm{\omega}L\\
	\diag(1,\bm{\omega},\ov{\bm{\omega}})A\diag(1,\ov{\bm{\omega}},\bm{\omega})=\bm{\omega}A \\
	\diag(1,\bm{\omega},\ov{\bm{\omega}})B\diag(1,\ov{\bm{\omega}}, \bm{\omega})=\bm{\omega}B,
	\end{array} \right.
	\\[2mm]
	&{\rm(iii)}\,\left\{
	\begin{array}{l}
	L=\bm{\omega}^{-1}L\\
	\diag(1,\bm{\omega},\ov{\bm{\omega}})A\diag(1,\ov{\bm{\omega}},\bm{\omega})=\bm{\omega}^{-1}A \\
	\diag(1,\bm{\omega},\ov{\bm{\omega}})B\diag(1,\ov{\bm{\omega}}, \bm{\omega})=\bm{\omega}^{-1}B.
	\end{array} \right.
	\end{align*}
	The Cases (ii) and (iii) are impossible because $L\not=0$. As for the Case (i), from the second and third conditions, it is easy to see that $A,B \in S(U(1)\times U(1) \times U(1))$. Needless to say, $L \in SU(3)$. 
	Hence there exist $L \in SU(3)$ and $A,B \in S(U(1)\times U(1)\times U(1))$ such that $\alpha=\varphi_{{}_{E_6,w_3}}(L,P,Q)$. Namely, there exist $L \in SU(3)$ and $A,B \in S(U(1)\times U(1)\times U(1))$ such that $\alpha=\varphi_{{}_{E_6,\sigma_3,w_3}}(L,P,Q)$. The proof of surjective is completed.
	
	Finally, we will determine $\Ker\,\varphi_{{}_{E_6,\sigma_3,w_3}}$. However, from $\Ker\,\varphi_{{}_{E_6,w_3}}=\{(E,E,E),(\bm{\omega}E,\allowbreak \bm{\omega}E,\bm{\omega}E),(\bm{\omega}^{-1}E,\bm{\omega}^{-1}E, \bm{\omega}^{-1}E) \}$, we easily obtain that $\Ker\,\varphi_{{}_{E_6,\sigma_3,w_3}}=\{(E,E,E),(\bm{\omega}E,\allowbreak \bm{\omega}E,\bm{\omega}E),(\bm{\omega}^{-1}E,\bm{\omega}^{-1}E, \bm{\omega}^{-1}E) \} \cong \Z_3$.
	Thus we have the isomorphism $(E_6)^{\sigma_3}\cap (E_6)^{w_3} \cong SU(3)\times S(U(1)\times U(1)\times U(1))\times S(U(1)\times U(1)\times U(1))/\Z_3$. 
	
	Therefore, by Lemma \ref{lemma 4.3.1} we have the required isomorphism 
	\begin{align*}
	(E_6)^{\sigma_3}\cap (E_6)^{w_3} \cong (SU(3)\times U(1)\times U(1)\times U(1)\times U(1))/\Z_3,
	\end{align*}
	where $\Z_3=\{(E,1,1,1,1),(\bm{\omega}E,\bm{\omega},\bm{\omega},\bm{\omega},\bm{\omega}),(\bm{\omega}^{-1}E,\bm{\omega}^{-1},\bm{\omega}^{-1},\bm{\omega}^{-1},\bm{\omega}^{-1})\}$.
\end{proof}

Thus, since the group $(E_6)^{\sigma_3} \cap (E_6)^{w_3}$ is connected from Theorem \ref{theorem 4.11.1}, we have an exceptional $\varmathbb{Z}_3 \times \varmathbb{Z}_3$-symmetric space
\begin{align*}
E_6/((SU(3)\times U(1)\times U(1)\times U(1)\times U(1))/\Z_3).
\end{align*}

\subsection{Case 12: $\{1, \tilde{\nu}_3,  \tilde{\nu}_3{}^{-1}\} \times \{1, \tilde{\mu}_3,  \tilde{\mu}_3{}^{-1}\}$-symmetric space}

Let the $C$-linear transformations $\nu_3, \mu_3$ of $\mathfrak{J}^C$  defined in Subsection \ref{subsection 3.3}. 

\noindent From Lemma \ref{lemma 3.3.8} (1), since we can easily confirm that $\nu_3$ and $\mu_3$ are commutative, $\tilde{\nu}_3$ and $\tilde{\mu}_3$ are commutative in $\Aut(E_6)$: $\tilde{\nu}_3\tilde{\mu}_3=\tilde{\mu}_3\tilde{\nu}_3$.

Before determining the structure of the group $(E_6)^{\nu_3} \cap (E_6)^{\mu_3}$, we confirm that useful lemma holds and prove proposition needed in the proof of theorem below.

\begin{lemma}\label{lemma 4.12.1}
	The mapping $\varphi_{{}_{E_6,\nu_3}}:Sp(1) \times S(U(1)\times U(5)) \to (E_6)^{\nu_3}$ of \,Theorem {\rm \ref{theorem 3.3.5}} satisfies the relational formulas 
	\begin{align*}
			\nu_3&=\varphi_{{}_{E_6,\nu_3}}(1,\diag(\nu^5,\nu^{-1},\nu^{-1},\nu^{-1},\nu^{-1},\nu^{-1})),
			\\
			\mu_3&=\varphi_{{}_{E_6,\nu_3}}(1, \diag(\nu^{-2},\nu^{2},\nu^{-1},\nu,\nu^{-1},\nu)),
	\end{align*}
	 where $\nu=\exp(2\pi i/9)\in U(1)$.
\end{lemma}
\begin{proof}
	 From Lemma \ref{lemma 3.3.8} (1), these results are trivial. 
\end{proof}

It goes with out saying that $\delta_Q=\varphi_{{}_{E_6, \nu_3}}(1,Q)(=\varphi_{{}_{E_6, \gamma_3}}(1,Q))$, where $\delta_Q$ is defined in the Case 7, and so from Lemma \ref{lemma 3.3.8} (1) the $C$-linear transformation $\mu'_3$ which is conjugate to $\mu_3$ under $\delta_Q \in (E_6)^{\nu_3}$ is also expressed by
\begin{align*}
\mu'_3=\varphi_{{}_{E_6,\nu_3}}(1, \diag(\nu^{-2},\nu^{2},\nu^{-1},\nu^{-1},\nu,\nu)).
\end{align*}

Then we have the following proposition.

\begin{proposition}\label{proposition 4.12.2}
	The group $(E_6)^{\nu_3} \cap (E_6)^{\mu_3}$ is isomorphic to the group $(E_6)^{\nu_3} \cap (E_6)^{\mu'_3}${\rm :} $(E_6)^{\nu_3} \cap (E_6)^{\mu_3} \cong (E_6)^{\nu_3} \cap (E_6)^{\mu'_3}$.
\end{proposition}
\begin{proof}
	We define a mapping $g_{{}_{4122}}: (E_6)^{\nu_3} \cap (E_6)^{\mu'_3} \to (E_6)^{\gamma_3} \cap (E_6)^{\mu_3}$ by
	\begin{align*}
	g_{{}_{4122}}(\alpha)=\delta_Q\alpha{\delta_Q}^{-1}.	
	\end{align*}
	Since it is easily to verify that $\delta_Q\nu_3=\nu_3\delta_Q$ using $\nu_3=\varphi_{{}_{E_6,\nu_3}}(1,\diag(\nu^5,\nu^{-1},\nu^{-1},\nu^{-1},\nu^{-1}, \allowbreak \nu^{-1}))$ (Lemma \ref{lemma 4.12.1}), we can prove this proposition as in the proof of Proposition \ref{proposition 4.7.1}.
\end{proof}
\vspace{1mm}

Now, we will determine the structure of the group $(E_6)^{\nu_3} \cap (E_6)^{\mu_3}$.

\begin{theorem}\label{theorem 4.12.3}
	The group $(E_6)^{\nu_3} \cap (E_6)^{\mu_3}$ is isomorphic the group $(Sp(1)\times U(1) \times U(1)\allowbreak \times U(1) \times SU(2)\times SU(2))/(\Z_2\times\Z_2\times\Z_4)${\rm :} $(E_6)^{\nu_3} \cap (E_6)^{\mu_3} \cong (Sp(1)\times U(1) \times U(1)\times U(1) \allowbreak \times SU(2)\times SU(2))/(\Z_2\times\Z_2\times\Z_4), \Z_2=\{(1,1,1,1,E,E), (1,1,-1,1,E,-E) \},
	\Z_2=\{(1,1,1,1,E,E), (1,1,-1,-1,-E,E) \},
	\Z_4=\{(1,1,1,E,E,E), (1,-1,1-,1,E,E), (-1,i,i,\allowbreak 1,-E,E), (-1,-i,-i,1,-E,E) \}$.
\end{theorem}
\begin{proof}
	Let $S(U(1)\times U(1)\times U(2)\times U(2)) \subset S(U(1)\times U(5))$. 
	We define a mapping $\varphi_{{}_{E_6,\nu_3,\mu'_3}}: Sp(1)\times S(U(1)\times U(1)\times U(2)\times U(2)) \to (E_6)^{\nu_3} \cap (E_6)^{\mu'_3}$ by
	\begin{align*}
	\varphi_{{}_{E_6,\nu_3,\mu'_3}}(q, P)(M+\a)&={k_J}^{-1}(P(k_J M){}^t\!P)+q\a k^{-1}(\tau \,{}^t\!P), 
	\\
	&\hspace*{40mm}M+\a \in \mathfrak{J}(3, \H)^C \oplus (\H^3)^C=\mathfrak{J}^C.
	\end{align*}
	Needless to say, this mapping is the restriction of the mapping $\varphi_{{}_{E_6,\nu_3}}$, that is, $\varphi_{{}_{E_6,\nu_3,\mu'_3}}(q, P)=\varphi_{{}_{E_6,\nu_3}}(q,P)$ (Theorem \ref{theorem 3.3.5}). 
	
	As usual, we will prove that $\varphi_{{}_{E_6,\nu_3,\mu'_3}}$ is well-defined. It is clear that $\varphi_{{}_{E_6,\nu_3,\mu'_3}}(q,P) \in (E_6)^{\nu_3}$, and it follows from $\mu'_3=\varphi_{{}_{E_6,\gamma_3}}(1,\diag({\nu}^{-2},\nu^2,{\nu}^{-1},{\nu}^{-1},\nu,\nu))$ that
	\begin{align*}
	&\quad {\mu'_3}^{-1}\varphi_{{}_{E_6,\nu_3,\mu'_3}}(q,P)\mu'_3
	\\
	&=\varphi_{{}_{E_6,\nu_3}}(1,\diag({\nu}^{-2},\nu^2,{\nu}^{-1},{\nu}^{-1},\nu,\nu))^{-1}\varphi_{{}_{E_6,\nu_3,\mu'_3}}(q,P)\varphi_{{}_{E_6,\nu_3}}(1,\diag({\nu}^{-2},\nu^2,{\nu}^{-1},{\nu}^{-1},\nu,\nu))
	\\
	&=\varphi_{{}_{E_6,\nu_3}}(1,\diag({\nu}^2,\nu^{-2},\nu,\nu,{\nu}^{-1},{\nu}^{-1}))\varphi_{{}_{E_6,\nu_3}}(q,P)\varphi_{{}_{E_6,\nu_3}}(1,\diag({\nu}^{-2},\nu^2,{\nu}^{-1},{\nu}^{-1},\nu,\nu))
	\\
	&=\varphi_{{}_{E_6,\nu_3}}(q,\diag({\nu}^2,\nu^{-2},\nu,\nu,{\nu}^{-1},{\nu}^{-1})P\diag({\nu}^{-2},\nu^2,{\nu}^{-1},{\nu}^{-1},\nu,\nu)),P=\diag(a,b,P_1,P_2)
	\\
	&=\varphi_{{}_{E_6,\nu_3}}(q,\diag(\nu^2 a \nu^{-2}, {\nu}^{-2} b \nu^2, (\nu E) P_1(\nu^{-1}E), ({\nu}^{-1}E) P_2 (\nu E) ))
	\\
	&=\varphi_{{}_{E_6,\nu_3}}(q,P)
	\\
	&=\varphi_{{}_{E_6,\nu_3,\mu'_3}}(q,P).
	\end{align*}
	Hence we have that $\varphi_{{}_{E_6,\nu_3,\mu'_3}}(q,P) \in (E_6)^{\mu'_3}$. Thus $\varphi_{{}_{E_6,\nu_3,\mu'_3}}$ is well-defined. Subsequently, since $\varphi_{{}_{E_6,\nu_3,\mu'_3}}$ is the restriction of the mapping $\varphi_{{}_{E_6,\nu_3}}$, we easily see that $\varphi_{{}_{E_6,\nu_3,\mu'_3}}$ is a homomorphism.
	
	Next, we will prove that $\varphi_{{}_{E_6,\nu_3,\mu'_3}}$ is surjective. Let $\alpha \in (E_6)^{\nu_3} \cap (E_6)^{\mu'_3} \subset (E_6)^{\nu_3}$. There exist $q \in Sp(1)$ and $A \in S(U(1)\times U(5))$ such that $\alpha=\varphi_{{}_{E_6,\nu_3}}(q,A)$ (Theorem \ref{theorem 3.3.5}). Moreover, from the condition $\alpha \in (E_6)^{\mu'_3}$, that is, ${\mu'_3}^{-1}\varphi_{{}_{E_6,\nu_3}}(q,A)\mu'_3=\varphi_{{}_{E_6,\nu_3}}(q,A)$, and using ${\mu'_3}^{-1}\varphi_{{}_{E_6,\nu_3}}(q,A)\mu'_3=\varphi_{{}_{E_6,\nu_3}}(q,\diag({\nu}^2,\nu^{-2},\nu,\nu,{\nu}^{-1},{\nu}^{-1}) A \,\diag({\nu}^{-2},\nu^2,{\nu}^{-1},{\nu}^{-1},\nu,\nu))$, we have that
	\begin{align*}
	&\left\{
	\begin{array}{l}
	q=q \\
	\diag({\nu}^2,\nu^{-2},\nu,\nu,{\nu}^{-1},{\nu}^{-1}) A\, \diag({\nu}^{-2},\nu^2,{\nu}^{-1},{\nu}^{-1},\nu,\nu)=A   
	\end{array}\right. 
	\\
	&\hspace*{50mm}{\text{or}}
	\\
	&\left\{
	\begin{array}{l}
	q=-q \\
	\diag({\nu}^2,\nu^{-2},\nu,\nu,{\nu}^{-1},{\nu}^{-1}) A \,\diag({\nu}^{-2},\nu^2,{\nu}^{-1},{\nu}^{-1},\nu,\nu)=-A.   
	\end{array}\right.     
	\end{align*}
	The latter case is impossible because of $q\not=0$. As for the former case, from the second condition, by doing straightforward computation $A$ takes the following form $\diag(a, b, C, D), a,b \in U(1),C, D \in U(2),  (ab)(\det\,C)(\det\,D)=1$, that is, $A \in S(U(1)\times U(1)\times  U(2)\times U(2))$. Needless to say, $q \in Sp(1)$.
	Hence there exist $q \in Sp(1)$ and $P \in S(U(1)\times U(1)\times U(2)\times U(2))$ such that $\alpha=\varphi_{{}_{E_6,\nu_3}}(q,P)$. Namely, there exist $q \in Sp(1)$ and $P \in S(U(1)\times U(1)\times U(2)\times U(2))$ such that $\alpha=\varphi_{{}_{E_6,\nu_3,\mu'_3}}(q,P)$. The proof of surjective is completed.
	
	Finally, we will determine $\Ker\,\varphi_{{}_{E_6,\nu_3,\mu'_3}}$. However, from $\Ker\,\varphi_{{}_{E_6,\nu_3}}=\{(1,E),(-1,-E) \}$, we easily obtain that $\Ker\,\varphi_{{}_{E_6,\nu_3,\mu'_3}}=\{(1,E),(-1,-E) \} \cong \Z_2$. Thus we have the isomorphism $(E_6)^{\nu_3} \cap (E_6)^{\mu'_3} \cong (Sp(1)\times S(U(1)\times U(1)\times U(2)\times U(2)))/\Z_2$. In addition, by Proposition \ref{proposition 4.12.2} we have the isomorphism $(E_6)^{\nu_3} \cap (E_6)^{\mu_3} \cong (Sp(1)\times S(U(1)\times U(1)\times U(2)\times U(2)))/\Z_2$. 
	
	Therefore, as in the proof of Theorem \ref{theorem 4.7.3}, we have the required isomorphism 
	\begin{align*}
	(E_6)^{\nu_3} \cap (E_6)^{\mu_3} \!\cong (Sp(1)\times U(1) \times U(1)\times U(1) \allowbreak \times SU(2)\times SU(2))/(\Z_2\times\Z_2\times\Z_4), 
	\end{align*}
	where 
	\begin{align*}
	&\Z_2=\{(1,1,1,1,E,E), (1,1,-1,1,E,-E) \},
	\\
	&\Z_2=\{(1,1,1,1,E,E), (1,1,-1,-1,-E,E) \},
	\\
	&\Z_4=\{(1,1,1,E,E,E), (1,-1,1-,1,E,E), (-1,i,i,1,-E,E), (-1,-i,-i,1,-E,E) \}.
	\end{align*}
\end{proof}

Thus, since the group $(E_6)^{\nu_3} \cap (E_6)^{\mu_3}$ is connected from Theorem \ref{theorem 4.12.3}, we have an exceptional $\varmathbb{Z}_3 \times \varmathbb{Z}_3$-symmetric space
\begin{align*}
E_6/((Sp(1)\times U(1) \times U(1)\times U(1) \allowbreak \times SU(2)\times SU(2))/(\Z_2\times\Z_2\times\Z_4)).
\end{align*}

\subsection{Case 13: $\{1, \tilde{\nu}_3,  \tilde{\nu}_3{}^{-1}\} \times \{1, \tilde{w}_3,  \tilde{w}_3{}^{-1}\}$-symmetric space}

Let the $C$-linear transformations $\nu_3, w_3$ of $\mathfrak{J}^C$  defined in Subsection \ref{subsection 3.3}. 

\noindent From Lemma \ref{lemma 3.3.8} (1), since we can easily confirm that $\nu_3$ and $w_3$ are commutative, $\tilde{\nu}_3$ and $\tilde{w}_3$ are commutative in $\Aut(E_6)$: $\tilde{\nu}_3\tilde{w}_3=\tilde{w}_3\tilde{\nu}_3$.

Before determining the structure of the group $(E_6)^{\nu_3} \cap (E_6)^{w_3}$, we confirm that useful lemma holds, and we prove proposition and lemma needed in the proof of theorem below.

\begin{lemma}\label{lemma 4.13.1}
	The mapping $\varphi_{{}_{E_6,\nu_3}}:Sp(1) \times S(U(1)\times U(5)) \to (E_6)^{\nu_3}$ of \,Theorem {\rm \ref{theorem 3.3.5}} satisfies the relational formula 
	\begin{align*}
	w_3=\varphi_{{}_{E_6,\nu_3}}(1, \diag(\tau\omega,\omega,\tau\omega,\omega,\tau\omega,\omega)),
	\end{align*}
	where $\nu=\exp(2\pi i/9)\in U(1)$.
\end{lemma}
\begin{proof}
	From Lemma \ref{lemma 3.3.8} (1), these results are trivial. 
\end{proof}

The $C$-linear transformation $w'_3$ defined in the Case 8 is expressed by
\begin{align*}
w'_3=\varphi_{{}_{E_6,\nu_3}}(1, \diag(\tau\omega,\tau\omega,\tau\omega,\omega,\omega,\omega)),
\end{align*}
and note that $\delta_N=\varphi_{{}_{E_6, \nu_3}}(1,N)(=\varphi_{{}_{E_6, \gamma_3}}(1,N))$, where $\delta_N$ is also defined in the Case 8, needless to say, $w_3$ is conjugate to $w'_3$ under $\delta_N=\varphi_{{}_{E_6, \nu_3}}(1,N)$.
\vspace{1mm}

Then we have the following proposition.

\begin{proposition}\label{proposition 4.13.2}
	The group $(E_6)^{\nu_3} \cap (E_6)^{w_3}$ is isomorphic to the group $(E_6)^{\gamma_3} \cap (E_6)^{w'_3}${\rm :} $(E_6)^{\nu_3} \cap (E_6)^{w_3} \cong (E_6)^{\nu_3} \cap (E_6)^{w'_3}$.
\end{proposition}
\begin{proof}
	We define a mapping $g_{{}_{4132}}: (E_6)^{\nu_3} \cap (E_6)^{w'_3} \to (E_6)^{\nu_3} \cap (E_6)^{w_3}$ by
	\begin{align*}
	g_{{}_{4132}}(\alpha)=\delta_N\alpha{\delta_N}^{-1},	
	\end{align*}
	where $\delta_N$ is same one above. Since it is easy to verify that $\delta_N \nu_3=\nu_3\delta_N$ using $\nu_3=\varphi_{{}_{E_6,\nu_3}}(1,\diag(\nu^5,\nu^{-1},\nu^{-1},\nu^{-1},\nu^{-1},\nu^{-1}))$ (Lemma \ref{lemma 4.9.1}) and $w_3\delta_N=\delta_N w'_3$ (Lemma \ref{lemma 4.13.1}), we can prove this proposition as in the proof of Proposition \ref{proposition 4.8.1}.
\end{proof}

Subsequently, we will prove the following lemma.

\begin{lemma}\label{lemma 4.13.3}
	The group $S(U(1)\times U(2)\times U(3))$ is isomorphic to the group $(U(1)\times U(1)\times SU(2)\times SU(3))/(\Z_2\times \Z_3)${\rm :} $S(U(1)\times U(2)\times U(3)) \cong (U(1)\times U(1)\times SU(2)\times SU(3))/(\Z_2\times \Z_3), \Z_2\!=\{(1,1,E,E),(-1,1,-E,E) \},\Z_3\!=\!\{(1,1,E,E),(1,\omega,E,\omega E),(1,\omega^{-1},E,\omega^{-1}E) \}$.	
\end{lemma}
\begin{proof}
	We define a mapping $f_{{}_{4133}}:U(1) \times U(1)\times SU(2)\times SU(3) \to S(U(1)\times U(2)\times U(3))$ by
	\begin{align*}
	f_{{}_{4133}}(a,b,A,B)=\left( 
	\begin{array}{ccc}
	a^{-2}b^{-3} & & {\raisebox{-7pt}[0pt]{\large $0$}}
	\\[2mm]
	& a\mbox{\large {$A$}} & 
	\\[2mm]
	{\raisebox{1pt}[0pt]{\large $0$}}&& b\mbox{\large {$B$}}
	\end{array}\right) \in SU(6).
	\end{align*}
	Then it is clear that $f_{{}_{4133}}$ is well-defined and a homomorphism. 
	
	We will prove that $f_{{}_{4133}}$ is surjective. Let $P \in S(U(1)\times U(2)\times U(3))$. Then $P$ takes the form of $\diag(s,P_1,P_2),s \in U(1),P_1 \in U(2), P_2 \in U(3),s(\det\,P_1)(\det\,P_2)=1$. Here, since $P_1 \in U(2), P_2 \in U(3)$, we see that $\det\,P_1, \det\,P_2 \in U(1)$. We choose $a,b \in U(1)$ such that $a^2=\det\,P_1, b^3=\det\,P_2$, respectively, and set $A=(1/a)P_1, B=(1/b)P_2$. Then we have that $ A \in SU(2), B \in SU(3)$. With above, the proof of surjective is completed.
	
	Finally, we will determine $\Ker\,f_{{}_{4133}}$. It follows from the kernel of definition that
	\begin{align*}
	\Ker\,f_{{}_{4133}}&=\{(a,b,A,B)\in U(1)\times U(1)\times SU(2) \times SU(3) \,|\,f_{{}_{4133}}(a,b,A,B)=E \}
	\\
	&=\{(a,b,A,B)\in U(1)\times U(1)\times SU(2) \times SU(3)\,|\,a^2b^3=1,aA=bB=E \}
	\\
	&=\{(a,b,a^{-1}E,b^{-1}E)\in U(1)\times U(1)\times SU(2) \times SU(3) \,|\,a^2=b^3=1 \}
	\\
	&=\{(1,1,E,E), (1,\omega,E,{\omega}^{-1}E),(1,{\omega}^{-1},E,\omega E), 
	\\
	&\hspace*{20mm}(-1,1,-E,E), (-1,\omega,-E,{\omega}^{-1}E),(-1,{\omega}^{-1},E,\omega E)\}
	\\
	&=\{(1,1,E,E), (-1,1,-E,E) \} \times \{(1,1,E,E), (1,\omega,E,{\omega}^{-1}E),(1,{\omega}^{-1},E,\omega E) \}
	\\
	& \cong \Z_2 \times \Z_3.
	\end{align*}
	
	Therefore we have the required isomorphism 
	\begin{align*}
	S(U(1)\times U(2)\times U(3)) \cong (U(1) \times U(1)\times SU(2)\times SU(3))/(\Z_2\times\Z_3).
	\end{align*}
\end{proof} 

Now, we will determine the structure of the group $(E_6)^{\nu_3} \cap (E_6)^{w_3}$.

\begin{theorem}\label{theorem 4.13.4}
	The group $(E_6)^{\nu_3} \cap (E_6)^{w_3}$ is isomorphic the group $(Sp(1)\times U(1) \times SU(2)\times SU(3))/(\Z_2 \times \Z_3)${\rm :} $(E_6)^{\nu_3} \cap (E_6)^{w_3} \cong ((Sp(1)\times U(1) \times SU(2)\times SU(3)))/(\Z_2 \times \Z_2 \times \Z_3), \Z_2\{(1,1,1,E,E), (1,-1,1,-E,E)\},\Z_2=\{(1,1,1,E,E), (-1,-1,-1,E,E)\}
	\Z_3=\{(1,\allowbreak 1,1,E,E), (1,1,\omega,E,{\omega}^{-1}E),(1,1,{\omega}^{-1},E,\omega E)\}$.
\end{theorem}
\begin{proof}
	Let $S(U(1)\times U(2)\times U(3)) \subset S(U(1) \times U(5))$. 
	We define a mapping $\varphi_{{}_{E_6,\nu_3,w'_3}}: Sp(1)\times S(U(1)\times U(2)\times U(3)) \to (E_6)^{\nu_3} \cap (E_6)^{w'_3}$ by
	\begin{align*}
	\varphi_{{}_{E_6,\nu_3,w'_3}}(q, P)(M+\a)&={k_J}^{-1}(P(k_J M){}^t\!P)+q\a k^{-1}(\tau \,{}^t\!P), 
	\\
	&\hspace*{40mm}M+\a \in \mathfrak{J}(3, \H)^C \oplus (\H^3)^C=\mathfrak{J}^C.
	\end{align*}
	Needless to say, this mapping is the restriction of the mapping $\varphi_{{}_{E_6,\nu_3}}$, that is, $\varphi_{{}_{E_6,\nu_3,w'_3}}(q, P)=\varphi_{{}_{E_6,\nu_3}}(q,P)$ (Theorem \ref{theorem 3.3.5}). 
	
	As usual, we will prove that $\varphi_{{}_{E_6,\nu_3,w'_3}}$ is well-defined. It is clear that $\varphi_{{}_{E_6,\nu_3,w'_3}}(q,P) \in (E_6)^{\nu_3}$, and it follows from $w'_3=\varphi_{{}_{E_6,\nu_3}}(1,\diag(\tau\omega,\tau\omega,\tau\omega,\omega,\omega,\omega))$ that
	\begin{align*}
	&\quad {w'_3}^{-1}\varphi_{{}_{E_6,\nu_3,\nu'_3}}(q,P) w'_3
	\\
	&=\varphi_{{}_{E_6,\nu_3}}(1,\diag(\tau\omega,\tau\omega,\tau\omega,\omega,\omega,\omega))^{-1}\varphi_{{}_{E_6,\nu_3,w'_3}}(q,P)\varphi_{{}_{E_6,\nu_3}}(1,\diag(\tau\omega,\tau\omega,\tau\omega,\omega,\omega,\omega))
	\\
	&=\varphi_{{}_{E_6,\nu_3}}(1,\diag(\omega,\omega,\omega,\tau\omega,\tau\omega,\tau\omega))\varphi_{{}_{E_6,\nu_3}}(q,P)\varphi_{{}_{E_6,\nu_3}}(1,\diag(\tau\omega,\tau\omega,\tau\omega,\omega,\omega,\omega))
	\\
	&=\varphi_{{}_{E_6,\nu_3}}(q,\diag(\omega,\omega,\omega,\tau\omega,\tau\omega,\tau\omega)P\,\diag(\tau\omega,\tau\omega,\tau\omega,\omega,\omega,\omega)),P=\diag(s,P_1,P_2)
	\\
	&=\varphi_{{}_{E_6,\nu_3}}(q,\diag(\omega s (\tau\omega),(\omega E)P_1(\tau\omega E), \tau(\omega E) P_2 (\omega E) ))
	\\
	&=\varphi_{{}_{E_6,\nu_3}}(q,P)
	\\
	&=\varphi_{{}_{E_6,\nu_3,w'_3}}(q,P).
	\end{align*}
	Hence we have that $\varphi_{{}_{E_6,\nu_3,w'_3}}(s,P) \in (E_6)^{w'_3}$. Thus $\varphi_{{}_{E_6,\nu_3,w'_3}}$ is well-defined. Subsequently, since $\varphi_{{}_{E_6,\nu_3,w'_3}}$ is the restriction of the mapping $\varphi_{{}_{E_6,\nu_3}}$, we easily see that $\varphi_{{}_{E_6,\nu_3,w'_3}}$ is a homomorphism.
	
	Next, we will prove that $\varphi_{{}_{E_6,\nu_3,w'_3}}$ is surjective. Let $\alpha \in (E_6)^{\nu_3} \cap (E_6)^{w'_3} \subset (E_6)^{\nu_3}$. There exist $q \in Sp(1)$ and $A \in S(U(1)\times U(5))$ such that $\alpha=\varphi_{{}_{E_6,\nu_3}}(q,A)$ (Theorem \ref{theorem 3.3.5}). Moreover, from the condition $\alpha \in (E_6)^{w'_3}$, that is, ${w'_3}^{-1}\varphi_{{}_{E_6,\nu_3}}(q,A)w'_3=\varphi_{{}_{E_6,\nu_3}}(q,A)$, and using ${w'_3}^{-1}\varphi_{{}_{E_6,\nu_3}}(q,A)w'_3=\varphi_{{}_{E_6,\nu_3}}(q,\diag(\omega,\omega,\omega,\tau\omega,\tau\omega,\tau\omega)A\,\diag(\tau\omega,\tau\omega,\tau\omega,\omega,\omega,\omega))$, we have that
	\begin{align*}
	&\left\{
	\begin{array}{l}
	q=q \\
	\diag(\omega,\omega,\omega,\tau\omega,\tau\omega,\tau\omega)A\,\diag(\tau\omega,\tau\omega,\tau\omega,\omega,\omega,\omega)=A   
	\end{array}\right. 
	\\
	&\hspace*{50mm}{\text{or}}
	\\
	&\left\{
	\begin{array}{l}
	q=-q \\
	\diag(\omega,\omega,\omega,\tau\omega,\tau\omega,\tau\omega)A\,\diag(\tau\omega,\tau\omega,\tau\omega,\omega,\omega,\omega)=-A.   
	\end{array}\right.     
	\end{align*}
	The latter case is impossible because of $q\not=0$. As for the former case, from the second condition, by doing straightforward computation $A$ takes the following form $\diag(s,C, D), C \in U(2),D \in U(3), s(\det\,C)(\det\,D)=1$, that is, $A \in S(U(1)\times U(2)\times U(3))$. Needless to say, $q \in Sp(1)$.
	Hence there exist $q \in Sp(1)$ and $P \in S(U(1)\times U(2)\times U(3))$ such that $\alpha=\varphi_{{}_{E_6,\nu_3}}(q,P)$. Namely, there exist $q \in Sp(1)$ and $P \in S(U(1)\times U(2)\times U(3))$ such that $\alpha=\varphi_{{}_{E_6,\nu_3,w'_3}}(q,P)$. The proof of surjective is completed.
	
	Finally, we will determine $\Ker\,\varphi_{{}_{E_6,\nu_3,w'_3}}$. However, from $\Ker\,\varphi_{{}_{E_6,\nu_3}}=\{(1,E),(-1,-E) \}$, we easily obtain that $\Ker\,\varphi_{{}_{E_6,\nu_3,w'_3}}=\{(1,E),(-1,-E) \} \cong \Z_2$. Thus we have the isomorphism $(E_6)^{\nu_3} \cap (E_6)^{w'_3} \cong (Sp(1)\times S(U(1)\times U(2)\times U(3)))/\Z_2$. In addition, by Proposition \ref{proposition 4.13.2} we have the isomorphism $(E_6)^{\nu_3} \cap (E_6)^{w_3} \cong (Sp(1)\times S(U(1)\times U(2)\times U(3)))/\Z_2$. Here, using the mapping $f_{{}_{4133}}$ in the proof of Lemma \ref{lemma 4.13.3}, we define a homomorphism $h_{{}_{4134}}:Sp(1)\times (U(1)\times U(1)\times SU(2)\times SU(3)) \to Sp(1)\times S(U(1)\times U(2)\times U(3)))$ by
	\begin{align*}
	h_{{}_{4134}}(q,(a,b,A,B))=(q,f_{{}_{4133}}(a,b,A,B)).
	\end{align*}
	Then, the elements of $(q,(a,b,A,B))$ corresponding to the elements 
	 $(1,E), (-1,-E) \in \Ker\,\varphi_{{}_{E_6,\nu_3,w'_3}}$ under the mapping $h_{{}_{4134}}$ are as follows.
	\begin{align*}
	& (1,1,1,E,E), (1,1,\omega,E,{\omega}^{-1}E),(1,1,{\omega}^{-1},E,\omega E), (1,-1,1,-E,E),
	\\
	& (1,-1,\omega,-E,{\omega}^{-1}E),(1,-1,{\omega}^{-1},-E,\omega E),
	\\
	& (-1,1,-1,-E,E), (-1,1,-\omega,-E,{\omega}^{-1}E),(-1,1,-{\omega}^{-1},-E,\omega E), (-1,-1,-1,E,E),
	\\
	& (-1,-1,-\omega,E,{\omega}^{-1}E),(-1,-1,-{\omega}^{-1},E,\omega E).
	\end{align*}
	
	Therefore we have the required isomorphism
	\begin{align*}
	(E_6)^{\nu_3} \cap (E_6)^{w_3} \cong (Sp(1) \times U(1)\times U(1) \times SU(2)\times SU(3))/(\Z_2\times \Z_2\times \Z_3), 
	\end{align*}
	where
	\begin{align*}
	&\Z_2=\{(1,1,1,E,E), (1,-1,1,-E,E)\},
	\\
	&\Z_2=\{(1,1,1,E,E), (-1,-1,-1,E,E)\},
	\\
	&\Z_3=\{(1,1,1,E,E), (1,1,\omega,E,{\omega}^{-1}E),(1,1,{\omega}^{-1},E,\omega E)\}.
	\end{align*}
\end{proof}

Thus, since the group $(E_6)^{\nu_3} \cap (E_6)^{w_3}$ is connected from Theorem \ref{theorem 4.13.4}, we have an exceptional $\varmathbb{Z}_3 \times \varmathbb{Z}_3$-symmetric space
\begin{align*}
E_6/((Sp(1) \times U(1)\times U(1) \times SU(2)\times SU(3))/(\Z_2\times \Z_2\times \Z_3)).
\end{align*}

\subsection{Case 14: $\{1, \tilde{\mu}_3,  \tilde{\mu}_3{}^{-1}\} \times \{1, \tilde{w}_3,  \tilde{w}_3{}^{-1}\}$-symmetric space}

Let the $C$-linear transformations $\mu_3, w_3$ of $\mathfrak{J}^C$  defined in Subsection \ref{subsection 3.3}. 

\noindent From Lemma \ref{lemma 3.3.8} (2), since we can easily confirm that $\mu_3$ and $w_3$ are commutative, $\tilde{\mu}_3$ and $\tilde{w}_3$ are commutative in $\Aut(E_6)$: $\tilde{\mu}_3\tilde{w}_3=\tilde{w}_3\tilde{\mu}_3$.

Now, we will determine the structure of the group $(E_6)^{\mu_3}\cap (E_6)^{w_3}$.

\begin{theorem}\label{theorem 4.14.1}
	The group $(E_6)^{\mu_3}\cap (E_6)^{w_3}$ is isomorphic to the group $(SU(3)\times U(1)\times U(1)\times U(1)\times U(1))/\Z_3${\rm :} $(E_6)^{\mu_3}\cap (E_6)^{w_3} \cong (SU(3)\times U(1)\times U(1)\times U(1)\times U(1))/\Z_3, \Z_3=\{(E,1,1,1,1),(\bm{\omega}E,\bm{\omega},\bm{\omega},\bm{\omega},\bm{\omega}),(\bm{\omega}^{-1}E,\bm{\omega}^{-1},\bm{\omega}^{-1},\bm{\omega}^{-1},\bm{\omega}^{-1})\}$.
\end{theorem}
\begin{proof}
	Let $S(U(1)\times U(1)\times U(1)) \subset SU(3)$. We define a mapping $\varphi_{{}_{E_6,\mu_3,w_3}}: SU(3)\times S(U(1)\times U(1)\times U(1))\times S(U(1)\times U(1)\times U(1)) \to (E_6)^{\nu_3}\cap (E_6)^{w_3}$ by
	\begin{align*}
	\varphi_{{}_{E_6,\mu_3,w_3}}(L,P,Q)(X_{C}+M)&=h(P,Q)X_{C}h(P,Q)^*+LM\tau h(P,Q)^*, 
	\\
	&\hspace*{20mm} X_{C}+M \in \mathfrak{J}(3, \C)^C \oplus 
	M(3,\C)^C=\mathfrak{J}^C.
	\end{align*}
	Needless to say, this mapping is the restriction of the mapping $\varphi_{{}_{E_6,w_3}}$, that is, $\varphi_{{}_{E_6,\nu_3,w_3}}(L,P,\allowbreak Q)\allowbreak=\varphi_{{}_{E_6,w_3}}(L,P,Q)$ (Theorem \ref{theorem 3.3.7}). 
	
	As usual, we will prove that $\varphi_{{}_{E_6,\mu_3,w_3}}$ is well-defined. It is clear that $\varphi_{{}_{E_6,\mu_3,_3}}(L,P,Q) \in (E_6)^{w_3}$, and it follows from $\mu_3=\varphi_{{}_{E_6,w_3}}(E,\diag({\bm{\varepsilon}}^{-2},\bm{\varepsilon},\bm{\varepsilon}), \diag({\bm{\varepsilon}}^{2},\bm{\varepsilon}^{-1},\bm{\varepsilon}^{-1}))$ (Lemma \ref{lemma 3.3.8} (2)) that 
	\begin{align*}
	&\quad {\mu_3}^{-1}\varphi_{{}_{E_6,\sigma_3,w_3}}(L,P,Q)\mu_3
	\\
	&=\varphi_{{}_{E_6,w_3}}(E,\diag({\bm{\varepsilon}}^{-2},\bm{\varepsilon},\bm{\varepsilon}), \diag({\bm{\varepsilon}}^{2},\bm{\varepsilon}^{-1},\bm{\varepsilon}^{-1}))^{-1}\varphi_{{}_{E_6,\mu_3,w_3}}(L,P,Q)
	\\
	&\hspace*{70mm}\varphi_{{}_{E_6,w_3}}(E,\diag({\bm{\varepsilon}}^{-2},\bm{\varepsilon},\bm{\varepsilon}), \diag({\bm{\varepsilon}}^{2},\bm{\varepsilon}^{-1},\bm{\varepsilon}^{-1}))
	\\
	&=\varphi_{{}_{E_6,w_3}}(E,\diag({\bm{\varepsilon}}^{2},\bm{\varepsilon}^{-1},\bm{\varepsilon}^{-1}), \diag({\bm{\varepsilon}}^{-2},\bm{\varepsilon},\bm{\varepsilon}))\varphi_{{}_{E_6,\mu_3,w_3}}(L,P,Q)
	\\
	&\hspace*{70mm}\varphi_{{}_{E_6,w_3}}(E,\diag({\bm{\varepsilon}}^{-2},\bm{\varepsilon},\bm{\varepsilon}), \diag({\bm{\varepsilon}}^{2},\bm{\varepsilon}^{-1},\bm{\varepsilon}^{-1}))
	\\
	&=\varphi_{{}_{E_6,w_3}}(L,\diag({\bm{\varepsilon}}^{2},\bm{\varepsilon}^{-1},\bm{\varepsilon}^{-1})P\diag({\bm{\varepsilon}}^{-2},\bm{\varepsilon},\bm{\varepsilon}),\diag({\bm{\varepsilon}}^{-2},\bm{\varepsilon},\bm{\varepsilon})Q \diag({\bm{\varepsilon}}^{2},\bm{\varepsilon}^{-1},\bm{\varepsilon}^{-1})),
	\\
	&\hspace*{90mm}P=\diag(a,b,c), Q=\diag(s,t,v)
	\\
	&=\varphi_{{}_{E_6,w_3}}(L,P,Q)
	\\
	&=\varphi_{{}_{E_6,\mu_3,w_3}}(L,P,Q).
	\end{align*}
	Hence we have that $\varphi_{{}_{E_6,\mu_3,w_3}}(L,P,Q) \in (E_6)^{\mu_3}$. Thus $\varphi_{{}_{E_6,\mu_3,w_3}}$ is well-defined. Subsequently, since $\varphi_{{}_{E_6,\mu_3,w_3}}$ is the restriction of the mapping $\varphi_{{}_{E_6,w_3}}$, we easily see that $\varphi_{{}_{E_6,\mu_3,w_3}}$ is a homomorphism.
	
	Next we will prove that $\varphi_{{}_{E_6,\mu_3,w_3}}$ is surjective. Let $\alpha \in (E_6)^{\mu_3}\cap (E_6)^{w_3} \subset (E_6)^{w_3}$. There exist $L, A, B \in SU(3)$ such that $\alpha=\varphi_{{}_{E_6,w_3}}(L,A,B)$ (Theorem \ref{theorem 3.3.7}). Moreover, from the condition $\alpha \in (E_6)^{\mu_3}$, that is, ${\mu_3}^{-1}\varphi_{{}_{E_6,w_3}}(L,A,B)\mu_3=\varphi_{{}_{E_6,w_3}}(L,A,B)$, and using 
	\begin{align*}
	&\quad {\mu_3}^{-1}\varphi_{{}_{E_6,w_3}}(L,A,B)\mu_3
	\\
	&=\varphi_{{}_{E_6,w_3}}(L,\diag({\bm{\varepsilon}}^{2},\bm{\varepsilon}^{-1},\bm{\varepsilon}^{-1})A\,\diag({\bm{\varepsilon}}^{-2},\bm{\varepsilon},\bm{\varepsilon}),\diag({\bm{\varepsilon}}^{-2},\bm{\varepsilon},\bm{\varepsilon})B \,\diag({\bm{\varepsilon}}^{2},\bm{\varepsilon}^{-1},\bm{\varepsilon}^{-1}))
	\end{align*}
	(Lemma \ref{lemma 3.3.8} (2)) we have that 
	\begin{align*}
	&\,\,\,{\rm(i)}\,\left\{
	\begin{array}{l}
	L=L
	\\
	\diag({\bm{\varepsilon}}^{2},\bm{\varepsilon}^{-1},\bm{\varepsilon}^{-1})A\diag({\bm{\varepsilon}}^{-2},\bm{\varepsilon},\bm{\varepsilon})=A 
	\\
	\diag({\bm{\varepsilon}}^{-2},\bm{\varepsilon},\bm{\varepsilon})B \diag({\bm{\varepsilon}}^{2},\bm{\varepsilon}^{-1},\bm{\varepsilon}^{-1})=B,
	\end{array} \right.
	\\[2mm]
	&{\rm(ii)}\,\left\{
	\begin{array}{l}
	L=\bm{\omega}L
	\\
	\diag({\bm{\varepsilon}}^{2},\bm{\varepsilon}^{-1},\bm{\varepsilon}^{-1})A\diag({\bm{\varepsilon}}^{-2},\bm{\varepsilon},\bm{\varepsilon})=\bm{\omega}A 
	\\
	\diag({\bm{\varepsilon}}^{-2},\bm{\varepsilon},\bm{\varepsilon})B \diag({\bm{\varepsilon}}^{2},\bm{\varepsilon}^{-1},\bm{\varepsilon}^{-1})=\bm{\omega}B,
	\end{array} \right.
	\\[2mm]
	&{\rm(iii)}\,\left\{
	\begin{array}{l}
	L=\bm{\omega}^{-1}L
	\\
	\diag({\bm{\varepsilon}}^{2},\bm{\varepsilon}^{-1},\bm{\varepsilon}^{-1})A\diag({\bm{\varepsilon}}^{-2},\bm{\varepsilon},\bm{\varepsilon})=\bm{\omega}^{-1}A 
	\\
	\diag({\bm{\varepsilon}}^{-2},\bm{\varepsilon},\bm{\varepsilon})B \diag({\bm{\varepsilon}}^{2},\bm{\varepsilon}^{-1},\bm{\varepsilon}^{-1})=\bm{\omega}^{-1}B.
	\end{array} \right.
	\end{align*}
	The Cases (ii) and (iii) are impossible because $L\not=0$. As for the Case (i), from the second and third conditions, it is easy to see that $A,B \in S(U(1)\times U(1) \times U(1))$. Needless to say, $L \in SU(3)$. Hence there exist $L \in SU(3)$ and $P,Q \in S(U(1)\times U(1) \times U(1))$ such that $\alpha=\varphi_{{}_{E_6,w_3}}(L,P,Q)$. Namely, there exist $L \in SU(3)$ and $P,Q \in S(U(1)\times U(1) \times U(1))$ such that $\alpha=\varphi_{{}_{E_6,\mu_3, w_3}}(L,P,Q)$. The proof of surjective is completed.
	
	Finally, we will determine $\Ker\,\varphi_{{}_{E_6,\mu_3,w_3}}$. However, from $\Ker\,\varphi_{{}_{E_6,w_3}}=\{(E,E,E),(\bm{\omega}E,\allowbreak \bm{\omega}E,\bm{\omega}E),(\bm{\omega}^{-1}E,\bm{\omega}^{-1}E, \bm{\omega}^{-1}E) \}$, we easily obtain that $\Ker\,\varphi_{{}_{E_6,\mu_3,w_3}}=\{(E,E,E),(\bm{\omega}E,\allowbreak \bm{\omega}E,\bm{\omega}E),(\bm{\omega}^{-1}E,\bm{\omega}^{-1}E, \bm{\omega}^{-1}E) \} \cong \Z_3$.
	Thus we have the isomorphism $(E_6)^{\mu_3}\cap (E_6)^{w_3} \cong (SU(3)\times S(U(1)\times U(1)\times U(1))\times S(U(1)\times U(1)\times U(1)))/\Z_3$. 
	
	Therefore, by Lemma \ref{lemma 4.3.1} we have the required isomorphism 
	\begin{align*}
	(E_6)^{\mu_3}\cap (E_6)^{w_3} \cong (SU(3)\times U(1)\times U(1)\times U(1)\times U(1))/\Z_3,
	\end{align*}
	where $\Z_3=\{(E,1,1,1,1),(\bm{\omega}E,\bm{\omega},\bm{\omega},\bm{\omega},\bm{\omega}),(\bm{\omega}^{-1}E,\bm{\omega}^{-1},\bm{\omega}^{-1},\bm{\omega}^{-1},\bm{\omega}^{-1})\}$.
\end{proof}

Thus, since the group $(E_6)^{\mu_3} \cap (E_6)^{w_3}$ is connected from Theorem \ref{theorem 4.14.1}, we have an exceptional $\varmathbb{Z}_3 \times \varmathbb{Z}_3$-symmetric space
\begin{align*}
E_6/((SU(3)\times U(1)\times U(1)\times U(1)\times U(1))/\Z_3).
\end{align*}

\end{document}